\documentclass[11pt,letterpaper]{amsart}

\usepackage[french, english]{babel}
 \usepackage[normalem]{ ulem }
\usepackage{soul}

\usepackage{cite}
\usepackage[latin1]{inputenc}
\usepackage{amssymb}
\usepackage{amsmath}
\usepackage{eucal}
\usepackage[]{epsfig}
\usepackage{wrapfig}
\usepackage[]{pstricks}
\usepackage{xcolor}
\usepackage{flafter}
\usepackage{amscd,amsthm}
\usepackage{amsfonts}
\usepackage{graphicx}
\usepackage{hyperref}

\usepackage{verbatim}

\usepackage{float}

\usepackage{enumerate, xspace}

\usepackage[charter]{mathdesign}

\usepackage{tikz}
\usetikzlibrary{arrows,decorations.pathmorphing,backgrounds,positioning,fit,petri,cd}

\usepackage{xcolor}
\usetikzlibrary{arrows.meta}
\usepackage{pstricks,pst-node}
\usetikzlibrary{backgrounds}
\usetikzlibrary{patterns,fadings}
\usetikzlibrary{arrows,decorations.pathmorphing}
\usetikzlibrary{shapes}
\usetikzlibrary{calc}
\definecolor{OliveGreen}{rgb}{0, 0.26, 0.15}
\definecolor{dark red}{rgb}{0.5, 0.05, 0.13}
\definecolor{dark red 2}{rgb}{0.81, 0.09, 0.13}
\definecolor{dark blue}{rgb}{0, 0.18, 0.39}
\definecolor{dark blue 2}{rgb}{0.03, 0.27, 0.49}
\definecolor{dark green2}{rgb}{0.07, 0.53, 0.03}
\definecolor{dark green}{rgb}{0, 0.44, 0}

\newtheorem{thm}{Theorem}[section]
\newtheorem{lem}[thm]{Lemma}

\newtheorem{prop}[thm]{Proposition}

\newtheorem{definition}[thm]{Definition}

\newtheorem{rem}[thm]{Remark}

\newcommand\LL{{\mathbb L}}
\newcommand\EE{{\mathbb E}}
\newcommand\PP{{\mathbb P}}
\newcommand\NN{{\mathbb N}}
\newcommand\RR{{\mathbb R}}

\newcommand\ZZ{{\mathbb Z}}

\newcommand{\mc}[1]{{\mathcal #1}}

\newcommand{\bb}[1]{{\mathbb #1}}

\definecolor{lightcarminepink}{rgb}{0.9, 0.4, 0.38}
\definecolor{lightcoral}{rgb}{0.94, 0.5, 0.5}
\definecolor{lightcornflowerblue}{rgb}{0.6, 0.81, 0.93}
\definecolor{lightcyan}{rgb}{0.88, 1.0, 1.0}
\definecolor{lavenderblush}{rgb}{1.0, 0.94, 0.96}
\definecolor{deeppeach}{rgb}{1.0, 0.8, 0.64}
\definecolor{darkchampagne}{rgb}{0.76, 0.7, 0.5}
\definecolor{desertsand}{rgb}{0.93, 0.79, 0.69}
\definecolor{classicrose}{rgb}{0.98, 0.8, 0.91}
\definecolor{myyeallow}{rgb}{0.98, 0.91, 0.71}
\definecolor{carolinablue}{rgb}{0.6, 0.73, 0.89}
\definecolor{antiquewhite}{rgb}{0.98, 0.92, 0.84}
\definecolor{mycolor}{rgb}{0.98, 0.91, 0.71}
\usepackage{enumerate}

\definecolor{camel}{rgb}{0.76, 0.6, 0.42}
\definecolor{brightcerulean}{rgb}{0.11, 0.67, 0.84}
\definecolor{cerulean}{rgb}{0.0, 0.48, 0.65}
\definecolor{Gray}{rgb}{0.5, 0.5, 0.5}

\numberwithin{equation}{section}

\usepackage{dsfont}
\usepackage{ulem}
\usepackage{setspace}

\usetikzlibrary{shapes.misc,arrows,decorations.pathmorphing,backgrounds,positioning,fit,petri,shapes}

\title[{Boundary driven super-diffusive symmetric exclusion}
]{Hydrodynamic limit for a boundary driven \\super-diffusive symmetric exclusion}

\author{C\'edric Bernardin, Pedro Cardoso, Patr\'icia   Gon\c calves, Stefano Scotta}

\newcommand{\Addresses}{{
		\bigskip
		\footnotesize
		
		C\'edric Bernardin, \textsc{Universit\'e C\^ote d'Azur, CNRS, LJAD, Parc Valrose, 06108 NICE Cedex 02, France\\
		\& Interdisciplinary Scientific Center Poncelet (CNRS IRL 2615), 119002 Moscow, Russia }\par\nopagebreak
		\textit{E-mail address}: \texttt{cbernard@unice.fr}
		
				\medskip
		
		Pedro Cardoso, \textsc{\noindent Center for Mathematical Analysis,  Geometry and Dynamical Systems \\
			Instituto Superior T\'ecnico, Universidade de Lisboa\\
			Av. Rovisco Pais, no. 1, 1049-001 Lisboa, Portugal}\par\nopagebreak
		\textit{E-mail address}: \texttt{pedro.gondim@tecnico.ulisboa.pt}

		\medskip
		
		Patr\'icia   Gon\c calves, \textsc{\noindent Center for Mathematical Analysis,  Geometry and Dynamical Systems \\
Instituto Superior T\'ecnico, Universidade de Lisboa\\
Av. Rovisco Pais, no. 1, 1049-001 Lisboa, Portugal}\par\nopagebreak
		\textit{E-mail address}: \texttt{pgoncalves@tecnico.ulisboa.pt}
		
		\medskip
		
		Stefano Scotta, \textsc{\noindent Center for Mathematical Analysis,  Geometry and Dynamical Systems \\
Instituto Superior T\'ecnico, Universidade de Lisboa\\
Av. Rovisco Pais, no. 1, 1049-001 Lisboa, Portugal}\par\nopagebreak
		\textit{E-mail address}: \texttt{stefano.scotta@tecnico.ulisboa.pt}
		
}}

\usepackage{a4wide}

\begin{document}
\maketitle{}

\begin{abstract}
{We study the hydrodynamic limit for  symmetric exclusion processes with  heavy-tailed long jumps and in contact with infinitely extended reservoirs. We show how the corresponding hydrodynamic equations are affected by the parameters defining the model. The hydrodynamic equations are characterized by a class of super-diffusive operators that are given by the regional fractional Laplacian with some additional reaction terms and various boundary conditions. Here we answer to all the questions left open in  \cite{BJGO2}  and we prove a conjecture stated in that same article.}
\end{abstract}

\tableofcontents

\section{Introduction}
\label{introduction}
The understanding of the macroscopic behavior of a physical system, from the microscopic description of its molecules, is the main goal of statistical mechanics. Stochastic interacting particle systems (IPS), introduced in the mathematics community by \cite{Spitzer}, have been intensively used to obtain, rigorously, the macroscopic laws which rule the space-time evolution of the conserved quantities of a system. These macroscopic laws take the form of partial differential equations (PDEs) whose nature depends on some features of the underlying IPS. The PDEs associated with an IPS are derived thanks to a scaling parameter $N \to \infty$, which connects the macroscopic space, where the solutions of the PDE are defined,  with the microscopic space, a discrete space where the particles of the IPS evolve. The exclusion process is the most studied IPS in the mathematical literature. Its dynamics consists of a large system of particles evolving on the lattice $\mathbb Z$ like independent continuous-time random walks, but with the rule that any jump, which would result in the occupation of a site by more than one particle, is suppressed. This ``exclusion'' constraint characterizes the interactions between particles and gives its name to this Markov process.  The number of particles is preserved and the goal is then to understand how does the particles' density evolves in space and time. The transition probability of the underlying random walks is denoted by $p:\mathbb Z\times \mathbb Z\to [0,1]$.\\

When the transition probability $p$ is translation invariant, i.e. $p(x,y)=p(y-x)$, with zero mean and finite variance, the density is governed by the heat equation with a constant diffusion coefficient equal to the variance of $p$. Different boundary conditions (Dirichlet, Robin, Neumann) can be imposed by putting the system in contact with some stochastic reservoirs. The nature of the boundary conditions depends on the reservoirs' dynamics  and their intensity via some scaling parameter. If $p$ has heavy tails, in particular, if the variance is infinite, then the density evolves super-diffusively according to a fractional diffusion equation \cite{MJ}, which is a non-local PDE. Putting the system in contact with reservoirs gives then rise to a variety of boundary conditions which are much more difficult to describe than in the diffusive case. This is due to the fact that a boundary condition usually has a local nature while the PDE, in this case, does not have.\\     

The exclusion process we analyze in this work was first introduced in \cite{BGJO, BJGO2, BJ}. We recall here how it is defined. Let $N >1$ and call ``bulk'' the lattice $\Lambda_N=\{1,\dots,N-1\}$. To the right and to the left of the bulk we add infinitely many stochastic reservoirs, which means that, according to some rules that we will explain in detail later, the sets $\{x\le 0\; ;\; x \in \ZZ\}$ and $\{x\ge N\; ;\; x \in \ZZ\}$ interact with the particles in the bulk, creating or annihilating them. Particles in the bulk move according to an exclusion process with long jumps: this means that particles can perform jumps of any length but they can not jump to a site where there is already a particle (exclusion rule). This kind of dynamics is determined by a symmetric transition probability  
\begin{equation}
  p: z \in {\mathbb Z} \to p(z)=\frac{c_{\gamma}}{|z|^{\gamma+1}} \mathbb{1}_{\{z\neq 0\}}, 
\end{equation}
depending only on the length of the jump and on a parameter $\gamma>0$. Hence the constant $c_\gamma$ is a normalizing constant equal to
\begin{equation}
\label{probability}
c_\gamma = \sum_{z \ne 0}\frac{1}{|z|^{\gamma+1}}.
\end{equation}
Associated with this transition probability we denote its second moment by $\sigma^2:=\sum_{z \in \mathbb{Z}}z^2p(z)\in (0, \infty]$ and we also introduce the quantity $m:=\sum_{z \geq 0}z p(z)$ belonging to $(0,\infty]$. The strength of the reservoirs is regulated by a factor $\kappa/N^{\theta}$ depending on two parameters $\kappa>0$ and $\theta \in \RR$. We will see how the macroscopic behavior of the system changes drastically according to the values of $\theta$. We call the reservoirs ``slow'' (or ``weak'') if $\theta\geq 0$ and ``fast'' (or ``strong'') if $\theta<0$. Moreover, we introduce two more parameters $\alpha \in (0,1)$ and $\beta \in (0,1)$ regulating, respectively, the density  of the left reservoir and of the right reservoir. The precise description of the evolution of the process  is given in Section \ref{model}. Figure \ref{fig1}  gives an illustration of the dynamics just described.

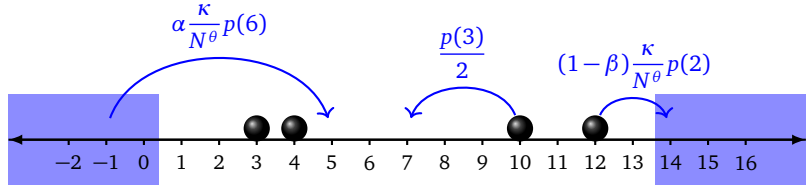
\begin{figure}[htb!]
	
	\begin{center}
		\begin{tikzpicture}[thick]

		\draw[shift={(-5.01,-0.15)}, color=black] (0pt,0pt) -- (0pt,0pt) node[below]{\dots};
		\draw[shift={(5.01,-0.15)}, color=black] (0pt,0pt) -- (0pt,0pt) node[below]{\dots};
		
		\fill [color=blue!45] (-5.3,-0.6) rectangle (-3.3,0.6);
		\fill [color=blue!45] (5.3,-0.6) rectangle (3.3,0.6);
		
		\draw[-latex] (-5.3,0) -- (5.3,0) ;
		\draw[latex-] (-5.3,0) -- (5.3,0) ;
		\foreach \x in {-4.5,-4,-3.5,...,4.5}
		\pgfmathsetmacro\result{\x*2+7}
		\draw[shift={(\x,0)},color=black] (0pt,0pt) -- (0pt,-2pt) node[below]{\scriptsize \pgfmathprintnumber{\result}};
		
		\node[ball color=black, shape=circle, minimum size=0.3cm] (B) at (-1.5,0.15) {};
		
		\node[ball color=black, shape=circle, minimum size=0.3cm] (C) at (1.5,0.15) {};
		
		\node[ball color=black, shape=circle, minimum size=0.3cm] (D) at (2.5,0.15) {};
		
		\node[ball color=black, shape=circle, minimum size=0.3cm] (E) at (-2,0.15) {};
		
		\node[draw=none] (S) at (3.5,0.15) {};
		\node[draw=none] (R) at (0,0.15) {};
		\node[draw=none] (L) at (-4,0.15) {};
		\node[draw=none] (M) at (-1,0.15) {};

		\path [<-] (S) edge[bend right =70, color=blue]node[above] {\footnotesize $(1-\beta)\dfrac{\kappa}{N^{\theta}}p(2)$}(D);
		\path [->] (C) edge[bend right =70, color=blue]node[above] {\footnotesize $\dfrac{p(3)}{2}$}(R);			
		\path [<-] (M) edge[bend right =70, color=blue]node[above] {\footnotesize $\alpha\dfrac{\kappa}{N^{\theta}}p(6)$}(L);

		\end{tikzpicture}
		\caption{Example of the dynamics of the model with $N=14$ and the  configuration $\eta=(0,0,1,1,0,0,0,0,0,1,0,1,0)$; the sites in the region colored in blue act as reservoirs.}	
		\label{fig1}
	\end{center}	
\end{figure}

In \cite{Adriana} the exclusion process evolving on $\Lambda_N$, with nearest-neighbour jumps ($p(1)=p(-1)=1/2$) and only one reservoir at each endpoint of the bulk is considered and depending on the value of $\theta$, three different hydrodynamic equations are derived: the heat equation, either with Dirichlet ($0\leq \theta< 1$), Robin ($\theta=1$) or Neumann ($\theta>1$)  boundary conditions. The hydrodynamic behavior of the exclusion process with long jumps given by a symmetric $p$ was studied in \cite{BGJO} when $p$ has finite variance ($\gamma>2$)  for all the regimes of $\theta$.  There, it is shown that the hydrodynamic equations involve the standard Laplacian operator and in particular, they are very similar to the ones found in \cite{Adriana}, where the reservoirs are present only at $\{0\}$ and at $\{N\}$. In {\cite{BGJO, BJGO2}
} when the reservoirs are strong enough ($\theta < 0$ in \cite{BGJO} and $\theta < 2 - \gamma$ in \cite{BJGO2}),  the hydrodynamic equation is simply a reaction equation where the diffusive operator is not present since the dynamics caused by the reservoirs are influencing the system much more than the interactions between particles inside $\Lambda_N$.

In \cite{BJGO2} the authors treat the case when $p$ has infinite variance and {{$m$ is finite}} ($1<\gamma<2$) and the reservoirs are quite strong, i.e.  $\theta\leq 0$.  As in the equilibrium setting of \cite{MJ}, in this regime, the standard Laplacian operator is replaced by the fractional Laplacian, but since in \cite{BJGO2} the process evolves in a macroscopic finite domain, the fractional operator on $\mathbb R$ is replaced by the \textit{regional fractional Laplacian}, which is defined in Subsection \ref{sec:notation}, with some specific boundary conditions. Also there, it is possible to see that when the reservoirs are very strong ($\theta<0$) the fractional operator disappears and the PDE involved is a simple reaction equation. All these results are reviewed in the lecture notes \cite{patricianote}. Partial results in the case of {\textit{asymmetric}} heavy tailed $p$ with infinite mean and without boundary conditions have been obtained in \cite{Seth1} (see also \cite{BGS16}). \\

The aim of this work is to complete the analysis of this model for all the values of the parameters involved, namely, the regimes in which the variance is infinite and which were not treated in \cite{BJGO2}. In \cite{BJGO2} the authors proved that for $\gamma \in(1,2)$ and $\theta<0$,  the hydrodynamic equation is a reaction equation, while for $\theta=0$, the hydrodynamic equation is a ``regional fractional reaction-diffusion equation'' with Dirichlet boundary conditions. The other regimes ($\theta>0$ and $1<\gamma<2$; and $0 < \gamma \leq 1$) were left open. Here we complete the scenario of the hydrodynamic limits for this process for all the values of $\theta$ and $0<\gamma<2$, apart one critical point where we could not prove uniqueness of the weak solution. In particular, in \cite{BJGO2}, it was conjectured that the hydrodynamic equation, in a regime where $\theta >0$ sufficiently small and $\gamma\in(1,2)$, is still a regional fractional diffusion equation with Dirichlet boundary conditions, by looking at the convergence as $\kappa \to 0$ of the solution of the hydrodynamic equation in the case $\theta=0$. In this article, we prove this conjecture and identify precisely its domain of validity, which is $0<\theta<\gamma-1$. On the way, we propose some weak formulation of fractional diffusion equations with various boundary conditions which seem to be new in the PDE's literature. We believe that this work could also be of interest to understand the non-equilibrium properties of systems of oscillators in contact with Langevin baths at their boundaries (see e.g. \cite{KBSD, KOR}).    \\

The hard part of the proof of the conjecture formulated in \cite{BJGO2} is to derive the Dirichlet boundary conditions from the particle system, i.e. to show, for example, when $\gamma \in (1,2)$ and $0< \theta <\gamma -1$, that the empirical density in a box of size $\varepsilon N$ around site $1$ is close to $\alpha$ in the limit $N\to \infty$ and then $\varepsilon \to 0$. This is obtained via a combination of two replacement lemmas{\footnote{Roughly speaking, a replacement lemma consists to control the error when we replace the empirical density in a box of size $k$ around the site $x$ by the empirical density in a box of size $\ell$ around the site $y$ or by a given fixed deterministic density. }}: the first one localized at boundary points with a box of order one; the second one following from a (so-called) one-block and a two-blocks estimates, which allows replacing the occupation variables at a site by averages in a microscopic box of size $\epsilon N$ around that site.  In \cite{BJGO2} the authors manage to avoid a proof of this claim by imposing the boundary condition at the PDE level in some indirect way. This argument fails when $\theta>0$. The derivation of the Dirichlet boundary condition we obtain in this paper, valid even for $\theta=0$, is explained in Section \ref{sec:fixing} and differs drastically from the method used in \cite{BJGO2}. It relies on the result that we present in Section \ref{MPL} and on a renormalization procedure. In the first step of this renormalization procedure we show that we can replace the occupation number at site $1$ by its average in a box of length $\ell_0$ for a certain value of $\ell_0$, which is, unfortunately, not big enough to conclude the proof. Then, we use a multi-scale argument (reminiscent of \cite{G,GJ}) which has the following role: we increase the size of $\ell_0$ by doubling it until we reach the needed size $\left\lfloor \varepsilon N \right\rfloor$ of the box. Finally, we show that the occupation number at site $1$ can be replaced by $\alpha$ in a proper regime of $\theta$.  The error that we get by increasing the size of the box from $\ell_0$ to $\left\lfloor \varepsilon N \right\rfloor$ is estimated in Section \ref{MPL}.  In particular, it permits us to estimate the change of energy of the system when we change the position of particles. This result  (Lemma \ref{posMPL}) is technically  challenging because we need to make an effective transport of mass in the system and  since particles can perform long jumps, there are several ways of  transporting  particles: either with nearest-neighbor jumps, one big jump or all the other possible paths with several jumps of different size. In fact, our proof works because we first localize the box where we do an effective transport of mass, by showing that we can make an average on  paths with only two jumps but with all the possible sizes of the jump, and then we show that this mechanism of transporting particles can, in fact, propagate to a microscopic box of size $\epsilon N$, as desired.\\

Another difficulty we faced in the remaining regimes was to derive the convergence of the discrete operator that appears from the random underlying dynamics to the regional fractional Laplacian. In \cite{BJGO2}, the authors prove a uniform convergence on smooth test functions with compact support. In this work, we show that this convergence holds in the $L^1$ sense for any smooth test function (even if not compactly supported) so that at the hydrodynamics level we may consider test functions that do not have compact support and therefore, see the boundary conditions. This result is proved in Subsection \ref{sec:convergence_disc_cont_oper}.

Our results are summarized in Figure \ref{fig:mesh1} and the proper statement is given below in Theorem \ref{theo:hydro_limit}.\\

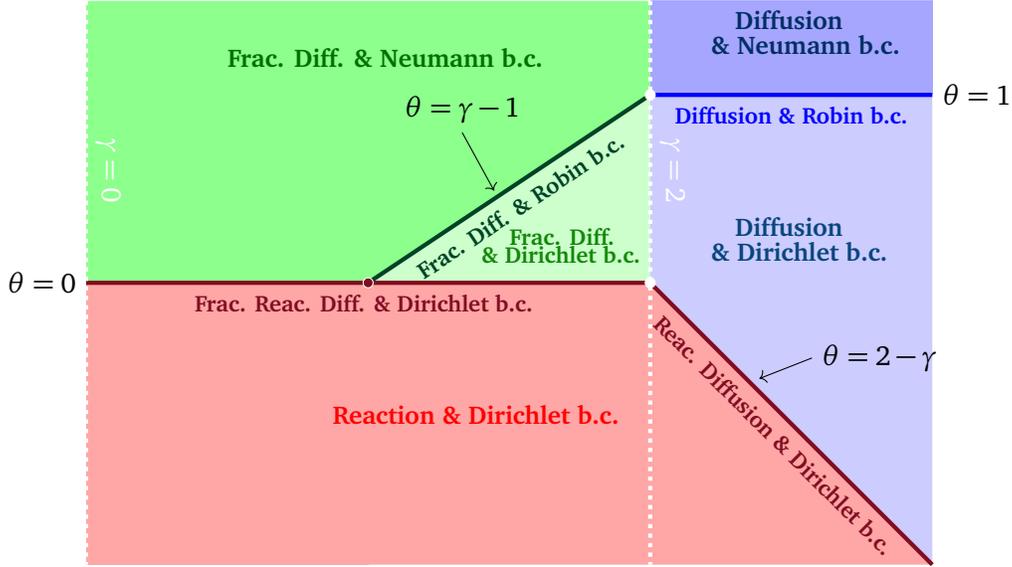
\begin{figure}[htb!]
\centering
	\begin{tikzpicture}[scale=0.25]
	\fill[color=red!35] (-25,-5) rectangle (-10,-20);
	\fill[color=red!35] (-10,-5) -- (5,-5) -- (20,-20) -- (-10,-20)-- cycle;
	\fill[color=green!20] (-10,-5) -- (5,-5) -- (5,5) -- cycle;
	\fill[color=green!45] (-25,-5) -- (-10,-5) -- (5,5) -- (5,10) -- (-25,10) -- cycle;
	\fill[color=blue!20] (5,-5) -- (20,-20) -- (20,5) -- (5,5) -- cycle;
	\fill[color=blue!35] (5,5) rectangle (20,10);
	\draw[-,=latex,dark red,ultra thick] (-25,-5) -- (5, -5) node[midway, below, sloped] {\footnotesize{{\textbf{\textcolor{dark red}{Frac. Reac. Diff. \& Dirichlet b.c. }}}}};
	\draw[dotted, ultra thick, white] (5,10) -- (5,-20);
	\draw[dotted, ultra thick, white] (-25,10) -- (-25,-20);
	\draw[-,=latex,dark red,ultra thick] (5,-5) -- (20, -20) node[midway, below, sloped] {\footnotesize{\textbf{{\textcolor{dark red}{Reac. Diffusion  \& Dirichlet b.c. }}}}};
	\draw[-,=latex,blue,ultra thick] (5,5) -- (20, 5) node[midway, sloped, below] {\footnotesize{{\textbf{\textcolor{blue}{Diffusion \& Robin b.c.}}}}};
	\draw[-,=latex,OliveGreen ,ultra thick] (-10,-5) -- (5, 5) node[midway, sloped, below] {\footnotesize{{\textbf{\footnotesize{\textcolor{OliveGreen}{Frac. Diff. \& Robin b.c.}}}}}};
	\node[right, black] at (9,9) {\textbf{\small{\textcolor{dark blue}{Diffusion}}}};
	\node[right, black] at (7.7,7.7) {\textbf{\small{\textcolor{dark blue}{\& Neumann b.c.}}}};
	\node[right, white] at (9,-2) {\textbf{\small{\textcolor{dark blue 2}{Diffusion}}}} ;
	\node[right, white] at (7.7,-3.3) {\textbf{\small{\textcolor{dark blue 2}{\& Dirichlet b.c.}}}} ;
	\node[right, white] at (-18,7) {\textbf{\small{\textcolor{dark green}{Frac. Diff. \& Neumann b.c.}}}} ;
	\node[right, white] at (-3,-2.5) {\textbf{\footnotesize{\textcolor{dark green2}{Frac. Diff. }}}} ;
	\node[right, white] at (-4.5,-3.5) {\textbf{\footnotesize{\textcolor{dark green2}{\& Dirichlet b.c. }}}} ;
	\node[right, white] at (-12.4,-12) {\textbf{\small{\textcolor{red}{Reaction \& Dirichlet b.c.}}}} ;
	\node[rotate=270, above] at (5,1) {\textcolor{white}{$\gamma = 2$}};
	\node[rotate=270, above] at (-25,1) {\textcolor{white}{$\gamma = 0$}};
	\node[right] at (20,5) {\textcolor{black}{$\theta=1$}};
	\node[left] at (-25,-5) {\textcolor{black}{$\theta=0$}};
	\draw[white,fill=white] (5,5) circle (1.5ex);
	\draw[white,fill=white] (5,-5) circle (1.5ex);
	\draw[white,fill=dark red] (-10,-5) circle (1.5ex);
	
	\draw[<-,black] (10.8, -10.1) -- (13.6,-9) node[right] {\textcolor{black}{$\theta=2-\gamma$}};
	\draw[<-,black] (-3.3, -0.1) -- (-5,3) node[above] {\textcolor{black}{$\theta=\gamma-1$}};
	
	
	\end{tikzpicture}
	\caption{Hydrodynamic behavior depending on the values of  $\theta$ (vertical axis) and $\gamma$ (horizontal axis). The case $\gamma >2$ is treated in \cite{BGJO}. The case $\gamma \in (1,2)$ and $\theta \leq 0$ is treated in \cite{BJGO2}. In this work we analyse all the other regimes, apart from the critical case $\gamma=2$,  which was solved in \cite{GS}.}
	\label{fig:mesh1}
\end{figure}

The paper is organized as follows: in Section 2, we define precisely the model and we present the hydrodynamic equations we obtained and the definition of weak solutions that we use. Moreover, in the same section, we state the main result of this paper. We also comment on the stationary profiles associated with the hydrodynamic equations. In order to derive these stationary profiles from the microscopic dynamics, we advise that it is possible to follow the strategies presented in \cite{tsu} and \cite{BJ}. At the end of Section 2, we present also some heuristic results on possible variations of the model under consideration. Section 3 contains the proof of the convergence of the empirical measure to a weak solution of the PDEs introduced in Section 2. In Section 4 we prove that the weak solutions of the hydrodynamic equations we define are unique, which completes the proof of the main result. We note however, that we are not able to prove uniqueness of weak solutions in the critical case corresponding to the hydrodynamic equation when $\theta=0$ and $\gamma=1$ and, for that reason, we cannot completely characterize the limit in this particular point. Section 5 is devoted to the proof of some technical lemmas that we need for the proofs of the results.

\section{Statement of results}

\subsection{The model}\label{model}
The Markov process we consider is analogous to the one introduced in \cite{BGJO} and \cite{BJGO2}, so we will use the same notation. 

The dynamics of the process we are interested in is determined by the interactions between particles inside the set, defined for $N\geq 2$ by $\Lambda_N:=\{1,\cdots, N-1\}$ called ``bulk''. The evolution of the configuration of particles is described by a function of time $\eta_{\cdot}$, taking values in the space $\Omega_N:=\{0,1\}^{\Lambda_N}$ of functions which associate to any point of $\Lambda_N$ either the value $1$ or $0$. Thus, if we fix some horizon time $T>0$, at time $t \in [0,T]$, the configuration $\eta_t$ is fully described by the values of $\eta_t(x)$ for $x \in \Lambda_N$: if $\eta_t(x)=1$ it means that there is a particle at site $x$ at time $t$, otherwise if $\eta_t(x)=0$ it means that the site $x$ is empty at time $t$.

The dynamics of this process is explained in detail in \cite{BGJO, BJGO2}. 
We recall it here briefly. It is easier to understand it separating bulk and reservoirs dynamics:
\begin{itemize}
    \item \textbf{bulk dynamics:} each pair of sites $(x,y) \in \Lambda_N\times \Lambda_N$ carries a Poisson process of intensity one. When there is an occurrence in this process we exchange the value of $\eta(x)$ and $\eta(y)$ with probability $p(y-x)/2$;\\
    \item \textbf{reservoirs dynamics}: let us start describing the dynamics associated to the left reservoirs. Every pair of points $(x,y) \in \{x\in \ZZ: x\leq 0\}\times \Lambda_N$ carries a Poisson process of intensity one. When there is an occurrence in this process we exchange $\eta(y)$ with $1-\eta(y)$ with rate $\alpha\tfrac{\kappa}{N^\theta}p(y-x)$ if $\eta(y)=0$ and with rate $(1-\alpha)\tfrac{\kappa}{N^\theta}p(x,y)$ if $\eta(y)=1$. The right reservoirs act in an analogous way: every pair of points $(x,y) \in \{x\in \ZZ: x\geq N\}\times \Lambda_N$ carries a Poisson process of intensity one. When there is an occurrence in this process we exchange $\eta(y)$ with $1-\eta(y)$ with rate $\beta\tfrac{\kappa}{N^\theta}p(y-x)$ if $\eta(y)=0$ and with rate $(1-\beta)\tfrac{\kappa}{N^\theta}p(y-x)$ if $\eta(y)=1$.
\end{itemize}

The infinitesimal generator of the process considered, which depends on some parameter $\alpha, \beta\in (0,1)$, $\kappa>0$ and   $\theta \in \mathbb R$, is defined as 
\begin{equation}
	\label{Generator}
	L_{N} = L_{N}^{0}+ \kappa N^{-\theta} L_{N}^{l}+\kappa N^{-\theta} L_{N}^{r},
\end{equation}
and acts on functions $f:\Omega_N \to \RR$ as
\begin{equation}\label{generators}
	\begin{split}
		&(L^0_N f)(\eta) =\cfrac{1}{2} \, \sum_{x,y \in \Lambda_N} p(x-y) [ f(\sigma^{x,y}\eta) -f(\eta)],\\
		&(L_N^{l} f)(\eta) =\sum_{\substack{x \in \Lambda_N\\ y \le 0}} p(x-y)c_{x}(\eta;\alpha) [f(\sigma^x\eta) - f(\eta)],\\
		&(L_N^{r} f)(\eta)= \sum_{\substack{x \in \Lambda_N \\ y \ge N}} p(x-y) c_{x}(\eta;\beta)  [f(\sigma^x\eta) - f(\eta)]
	\end{split}
\end{equation}
where 
\begin{equation*}
	(\sigma^{x,y}\eta)(z) = 
	\begin{cases}
		\eta(z),& \textrm{if}\;\; z \ne x,y,\\
		\eta(y),& \textrm{if}\;\; z=x,\\
		\eta(x),& \textrm{if}\;\; z=y
	\end{cases}
	, \quad (\sigma^x\eta)(z)= 
	\begin{cases}
		\;\; \eta(z), &\textrm{if}\;\; z \ne x,\\
		1-\eta(x),& \textrm{if}\;\; z=x,
	\end{cases}
\end{equation*}
and  for  $\delta \in (0,1)$ and $x\in \Lambda_{N}$, we define
\begin{equation}
\label{rate_c}
	c_{x} (\eta;\delta) :=\left[ \eta(x) \left(1-\delta \right) + (1-\eta(x))\delta\right].
\end{equation}

We will consider the Markov process in the time scale $t\Theta(N)$, where $\Theta(N)$ is defined in \eqref{timescale}. In order to have a lighter notation we will denote $\eta_t^N:=\eta_{t \Theta(N)}$, so that the process $\eta_{\cdot}^N$ has as infinitesimal generator $\Theta(N)L_N$.

\subsection{Topological setting and fractional operators in bounded domains}\label{sec:notation}
In this subsection, we introduce the notation, the operators and the spaces of functions that we will use in the rest of this work.

First, let us fix some notation. For $h:[0,1]\to [0,\infty)$ a Borel function we will consider frequently the Hilbert space $L^2([0,1], h(u)du)$ that we denote simply by $L^2_h([0,1])$. The inner product and the norm associated to this space will be denoted respectively by $\langle \cdot, \cdot \rangle_h$ and $||\cdot||_h$. When $h=1$, we will work with the standard $L^2([0,1],du)$ space, and in the norm and the inner product,  we will omit the index $h=1$. For any interval $I \subseteq \RR$, we denote $C^k(I)$ (resp. $C_c^k (I)$)  the space of continuous real-valued functions (resp. with compact support included in $I$) with the first $k$-th derivatives being continuous as well. Moreover, for $T>0$, we will say that a function $H \in C^{m,n}([0,T]\times I)$ if $H(\cdot,x)\in C^m([0,T])$ and $H(t,\cdot) \in C^n(I)$ for any $x \in I$ and $t \in [0,T]$. Analogously $H \in C_c^{m,n}([0,T]\times I)$ if $H \in C^{m,n} ([0,T]\times I)$ and  $H(t,\cdot) \in C_c^n(I)$ for any $t \in [0,T]$. If one (or more) superscript is equal to $\infty$ it means that the function considered is smooth in the respective variable. We will use equivalently the notation $H_t(u)$ and $H(t,u)$, so the subscripts have not to be confused with partial derivatives. We will denote the derivatives of a function $H \in C^{m,n}([0,T]\times I)$, by $\partial_t H$ if the derivative is in the time variable and $\partial_u H$ if the derivative is in the space variable.

We will write hereinafter $f(u) \lesssim g(u)$ if there exists a constant $C$ independent of $u$ such that $f(u) \le C g(u)$ for every $u$; moreover, we will write $f(u) = {O} (g(u) )$ if the condition $|f (u) | \lesssim |g(u) |$ is satisfied. \\

We recall now the definition of the fractional operators and the spaces that we will use.

The fractional Laplacian $(-\Delta)^{\gamma/2}$ with exponent $\gamma/2$, for $\gamma \in (0,2)$, is the operator acting on functions $f:\RR\rightarrow\RR$ such that
\begin{equation}
\label{eq:integ1}
\int_{-\infty}^{\infty} \cfrac{|f(u)|}{(1 +|u|)^{1+\gamma}} du < \infty
\end{equation}
as
\begin{equation}
-(-\Delta)^{\gamma/2} f (u) = c_\gamma  \lim_{\epsilon \to 0} \int_{-\infty}^{\infty} {\bb 1}_{|u-v| \ge \epsilon}\cfrac{f(v) -f(u)}{|u-v|^{1+\gamma}} dv,
\end{equation}
for any $u \in \RR$, if the previous limit exists and where $c_{\gamma}$ is defined in \eqref{probability}. The interested reader may have a look at \cite{Poz} for more details.
In a  finite domain we introduce the regional fractional Laplacian $\mathbb L$ on the interval $I\subset \RR$ which acts on functions $f:I \rightarrow \RR$ such that
\begin{equation*}
\int_{I} \cfrac{|f(u)|}{(1 +|u|)^{1+\gamma}} du < \infty
\end{equation*}
as
\begin{equation}
\label{definition}
({\bb L} f)(u) = c_\gamma  \lim_{\epsilon \to 0} \int_I {\bb 1}_{|u-v| \ge \epsilon} \cfrac{f(v) -f(u)}{|u-v|^{1+\gamma}} dv,
\end{equation}
for any $u \in I$ provided the limit exists. In this work we will always consider the case in which $I$ is the interval $(0,1)$ and in this particular case $\mathbb L f$ is well defined, for example, if $f \in C^2([0,1])$.\\

We  introduce the semi inner-product $\langle\cdot,\cdot\rangle_{\gamma/2}$, and  corresponding semi-norm $\|\cdot \|_{\gamma/2}:= \langle \cdot, \cdot \rangle_{\gamma/2}$, defined by
\begin{equation}
\langle f, g \rangle_{\gamma/2} =  \cfrac{c_{\gamma}}{2} \iint_{[0,1]^2} \cfrac{(f(u) -f(v)) (g(u) -g(v))}{|u-v|^{1+\gamma}} du dv,
\end{equation}  
where $f,g:(0,1) \rightarrow \RR$ are functions such that $\| f\|_{\gamma/2}<\infty$ and $\|g\|_{\gamma/2}<\infty$.\\

We have the following integration by parts formula (Theorem 3.3) in \cite{guan}:


\begin{prop}[Theorem 3.3 \cite{guan}]
\label{prop:ip-guan}
If $f \in C^2([0,1])$, then $\bb Lf \in L^1 ([0,1], du)$ for every $\gamma \in (0,2)$. Moreover, if $g$ is bounded and $\|g\|_{\gamma/2}<\infty$, then 
\begin{equation} \label{intpart}
\langle {-\bb L} f, g \rangle = \langle f, g \rangle_{\gamma/2}, \forall \gamma \in (0,2).
\end{equation}
\end{prop}

In Corollary 7.6 of \cite{guan} the following generalization is also established for $\gamma \in (1,2)$.

\begin{prop} [\cite{guan}]
\label{prop:guan0}
Let $\gamma \in (1,2)$ and $g\in C^2([0,1])$. Let $f:[0,1] \to \mathbb R$ be such that $u\to f(u)u^{1-\gamma}$ and $u \to f(u) (1-u)^{1-\gamma}$ are in $C^2([0,1])$. Then,
\begin{equation}
\langle - \mathbb L g ,f \rangle = \chi_\gamma \left[g(1) D^\gamma f (1) -g(0)D^\gamma f (0)\right]   + \langle - \mathbb L f ,g \rangle
\end{equation}
where 
\begin{align*}
\chi_\gamma:= \frac{c_{\gamma}}{\gamma (\gamma - 1)} \Bigg[ \int_0^1 \frac{1}{u^{2-\gamma} (1-u )^{\gamma-1} } - \frac{1}{\gamma-1} + \lim_{\delta \rightarrow 0^{+}} \int_{\delta}^{1} \Big( \frac{1}{u^{2-\delta} (u - \delta)^{\gamma - 1}} - \frac{1}{u} \Big) du \Bigg]
\end{align*}
 and 
\begin{equation}
(D^\gamma f)(0)=\lim_{u\to 0^+} f' (u) u^{2-\gamma}, \quad (D^\gamma f)(1)=\lim_{u\to 1^-} f' (u) (1-u)^{2-\gamma}.
\end{equation}
\end{prop}

In order to have later a lighter notation it is useful to introduce a family of fractional operators $\{\mathbb{L}_{\kappa}\}_{\kappa \geq 0}$, defined in \cite{BJGO2} as
\begin{equation*}
{\bb L}_{\kappa}:= \LL -\kappa V_1,\quad \text{where} \quad V_1 (u) :=r^-(u)+r^+(u)
\end{equation*}
and,  for any $u \in (0,1)$,
\begin{equation}
      r^- (u):=c_\gamma \gamma^{-1} u^{-\gamma},\;\;\;  r^+ (u) := c_\gamma \gamma^{-1} (1-u)^{-\gamma}.
\end{equation}
We also define, for any $u \in (0,1)$ and $\alpha, \beta \in (0,1)$,
\begin{equation}
    V_0(u)=\alpha r^-(u)+ \beta r^+(u).
\end{equation}

\begin{rem}
   Observe that the operators $\mathbb L_{\kappa}$ defined above are symmetric operators in $L^2([0,1], du).$ Moreover, for $\kappa=1$, we recover the so-called \textit{restricted fractional Laplacian} (see \cite{piotr}), for any $f\in C^\infty_c((0,1))$:
\begin{align}
&\forall u \in (0,1), \quad -(-\Delta)^{\gamma/2}  f (u)  =({\bb L} f)(u) -  V_1(u) f(u)=(\LL_{1}f)(u)\label{Operator_LL}  
\end{align}
while for $\kappa = 0$ we get the  regional fractional Laplacian $\bb L$ defined in \eqref{definition}. 
\end{rem}

\begin{rem}
	It is well known that the fractional Laplacian $-(-\Delta)^{\gamma/2}$ is the infinitesimal generator of a $\gamma$-stable L\'evy jump process on $\mathbb R$. In the finite domain $(0,1)$, the regional fractional Laplacian $\mathbb L$ is also associated to a Markov process which roughly corresponds to the L\'evy jump process above reflected at the boundaries of $(0,1)$. Because it is a jump process, making sense of the previous sentence is not trivial (see \cite{guan}). Following \cite{guan}, the process with generator $\bb L$ is called the ``reflected symmetric $\gamma$-stable process". Its behavior is strongly dependent on the exponent $\gamma$. For $\gamma \in (0,1]$, it is essentially the same as the ``symmetric $\gamma$-stable censored process" defined in \cite{Bogdan} and which is obtained from the L\'evy jump process on $\bb R$ by suppressing its jumps from $(0,1)$ to its complement. It is proved that for $\gamma \le 1$, the censored (or equivalently the reflected) process  will never touch the boundaries, so that in some sense the process does not see the boundaries. This is not the case if $\gamma>1$.

	
	
\end{rem}


\begin{definition}
	\label{Def. Sobolev space}
	Let us denote by $\mathcal{H}^{\gamma/2}:=\mathcal{H}^{\gamma/2}([0,1])$ the Sobolev space containing all the functions $g\in L^2([0,1])$ such that $\| g \|_{\gamma/2} <\infty$. This is a Hilbert space endowed with the norm $\|\cdot\|_{{\mc H}^{\gamma/2}}$ defined by
	$$\| g \|_{\mathcal{H}^{\gamma/2}}^{2}:= \| g \|^{2} + \| g \|_{\gamma/2}^{2} .$$
	If $\gamma \in (1,2)$ its elements coincide a.e. with continuous functions. The space ${\mc H}_0^{\gamma/2}:={\mc H}_0^{\gamma/2}([0,1])$ is defined as the completion of $C_c^{\infty} ((0,1))$ for the norm just introduced. If $\gamma \in (1,2)$, its elements coincide a.e. with continuous functions vanishing at $0$ and $1$ and as pointed out in \cite{Val}, on ${\mc H}_0^{\gamma/2}$, the norms $ \|\cdot \|_{{\mc H}^{\gamma/2}}$ and  $\| \cdot \|_{\gamma/2}$ are equivalent. If { $\gamma \in (0,1] $  }, $\mc H_0^{\gamma/2}={\mc H}^{\gamma/2}$  (see Theorem 3, Section 2.4 of \cite{RS-book} ).    
	\end{definition}
	
	\begin{rem}
	The main difference between the case { $ \gamma \in (0,1] $ } and the case $\gamma \in (1,2)$ is that the trace on $\{0,1\}$ of a function $u\in {\mathcal H}^{\gamma/2}$ exists if $\gamma>1$ while it does not exist (a priori) if {$\gamma \leq 1$}.      
	\end{rem}

	The weak solutions of the hydrodynamic equations that we will define belong to the space $L^{2}(0,T;\mathcal{H}^{\gamma/2})$ which is the set of measurable functions $f:[0,T]\rightarrow  \mathcal{H}^{\gamma/2}$ such that 
	$$\int^{T}_{0} \Vert f_{t} \Vert^{2}_{\mathcal{H}^{\gamma/2}}dt< \infty. $$
	The spaces $L^{2}(0,T;\mathcal{H}_0^{\gamma/2})$ and   $L^{2}(0,T;L^{2}_h([0,1]) )$  are defined in an analogous way, using the spaces  {$\mathcal{H}_0^{\gamma/2}$ and $L^{2}_h([0,1])$} instead of $\mathcal{H}^{\gamma/2}$.

The regional fractional Laplacian can be extended to the space $\mathcal{H}^{\gamma/2}$. Indeed for any $\rho \in \mathcal{H}^{\gamma/2}$ we can define the distribution $\mathbb{L}\rho$ on $(0,1)$ by its action on functions $f \in C_c^{\infty}((0,1))$:
  $$\langle \LL\rho,f\rangle = \langle \rho,\LL f\rangle. $$
 The proof that $\mathbb{L}\rho$ is indeed a well defined distribution can be found in \cite{BJGO2}.

\subsection{Hydrodynamic equations}
\label{subsec:hyd_eq}
Now, that we have introduced all the notation and the spaces of functions that we will use, we can define the PDEs and respective notions of weak solutions,  which are involved in the hydrodynamic limit of this model. 


\begin{definition}
	\label{Def. Neumann Condition}
	Let { $\gamma \in (0,2)$ } and $g:[0,1]\rightarrow [0,1]$ be a measurable function. We say that  $\rho:[0,T]\times[0,1] \to [0,1]$ is a weak solution of the {regional} fractional diffusion equation with fractional Neumann boundary conditions and initial condition $g$:
	\begin{equation}
	\label{eq:Neumann Equation}
	\begin{cases}
	&\partial_{t} \rho_{t}(u)= \LL \rho_t(u),  \quad (t,u) \in [0,T]\times(0,1),\\
&(D^\gamma \rho^{\hat \kappa}_{t})(0)=  (D^\gamma \rho^{\hat \kappa}_{t})(1) =0,  \quad t \in (0,T], \\
	&{\rho}_{0}(u)= g(u),\quad u \in (0,1),
	\end{cases}
	\end{equation}
	if : 
	\begin{enumerate}[i)] 
		
		\item  $\rho \in L^{2}(0,T;\mathcal{H}^{\gamma/2})$.
		\item For all $t\in [0,T]$ and {all functions} $G \in C^{1,\infty} ([0,T]\times (0,1))$ we have that 
		\begin{equation}
		\label{eq:Neumann integral}
		\begin{split}
		F_{\rm{Neu}}(t, \rho,G,g):=&\left\langle \rho_{t},  G_{t} \right\rangle -\left\langle g,   G_{0}\right\rangle - \int_0^t\left\langle \rho_{s},\Big(\partial_s + \bb L \Big) G_{s}  \right\rangle ds =0.
		\end{split}   
		\end{equation}
	\end{enumerate}
\end{definition}

\begin{definition}
	\label{Def. Dirichlet Condition2}
	Let $\gamma \in (0,2)$, $\hat \kappa > 0$ be some parameter and $g:[0,1]\rightarrow [0,1]$ be a measurable function. We say that  $\rho^{\hat \kappa}:[0,T]\times[0,1] \to [0,1]$ is a weak solution of the {regional} fractional reaction-diffusion equation with non-homogeneous Dirichlet boundary conditions and initial condition $g$:
	\begin{equation}
	\label{eq:Dirichlet Equation2}
	\begin{cases}
	&\partial_{t} \rho_{t}^{\hat \kappa}(u)= \LL_{\hat \kappa} \rho_t^{\hat \kappa}(u)+\hat \kappa V_0(u),  \quad (t,u) \in [0,T]\times(0,1),\\
& \rho_{t}^{\hat \kappa}(0)=\alpha, \quad  \rho_{t}^{\hat \kappa}(1)=\beta,\quad t \in (0,T], \\
	& \rho_{0}^{\hat \kappa}(u)= g(u),\quad u \in (0,1),
	\end{cases}
	\end{equation}
	if : 
	\begin{enumerate}[i)] 
		
		\item  $\rho^{\hat \kappa} \in L^{2}(0,T;\mathcal{H}^{\gamma/2})$.
		
		\item   $\int_0^T \int_0^1 \Big\{ \frac{(\alpha-\rho_t^{\hat\kappa}(u))^2}{u^\gamma}+\frac{(\beta-\rho_t^{\hat\kappa}(u))^2}{(1-u)^\gamma}\Big\} \, du\, dt <\infty$.

		\item For all $t\in [0,T]$ and {all functions} $G \in C_c^{1,\infty} ([0,T]\times (0,1))$ we have that 
		\begin{equation}
		\label{eq:Dirichlet integral2}
		\begin{split}
		&F_{RD}(t, \rho^{\hat \kappa},G,g)\\
		&:=\left\langle \rho^{\hat \kappa}_{t},  G_{t} \right\rangle -\left\langle g,   G_{0}\right\rangle - \int_0^t\left\langle \rho^{\hat \kappa}_{s},\Big(\partial_s + \bb L_{\hat \kappa} \Big) G_{s}  \right\rangle ds -\hat \kappa\int_0^t\langle G_{s},  V_0\rangle ds =0.
		\end{split}   
		\end{equation}
	\end{enumerate}
\end{definition}

\begin{rem}
Observe that in Definition \ref{Def. Dirichlet Condition2}, the second item is automatically satisfied if $\gamma \in (0,1)$ since $\rho^{\hat \kappa}$ is uniformly bounded by $1$. If $\gamma \in (1,2)$ the second item implies that for a.e. time $t$ we have that $\rho^{\hat \kappa}_t (0) =\alpha$ and $\rho^{\hat \kappa}_t (1) =\beta$. We conjecture that this remains true for $\gamma \in (0,1]$ even if the weak formulation \eqref{eq:Dirichlet integral2} does not seem to imply trivially these boundary conditions.      
\end{rem}

\begin{definition}
	\label{Def. Dirichlet0 Condition}
	Let $\gamma \in (1,2)$. Let $g:[0,1]\rightarrow [0,1]$ be a measurable function. We say that  $\rho:[0,T]\times[0,1] \to [0,1]$ is a weak solution of the {regional} fractional diffusion equation with non-homogeneous Dirichlet boundary conditions and initial condition $g$:
	\begin{equation}
	\label{eq:Dirichlet0 Equation}
	\begin{cases}
	&\partial_{t} \rho_{t}(u)= \LL \rho_t(u),  \quad (t,u) \in [0,T]\times(0,1), \\
		&{ \rho_{t}}(0)=\alpha, \quad { \rho_{t}}(1)=\beta,\quad t \in (0,T], \\
	&{ \rho}_{0}(u)= g(u),\quad u \in (0,1),
	\end{cases}
	\end{equation}
	if : 
	\begin{enumerate}[i)] 
		
		\item  $\rho \in L^{2}(0,T;\mathcal{H}^{\gamma/2})$.
		\item For all $t\in [0,T]$ and {all functions} $G \in C^{1,\infty}_c ([0,T]\times (0,1))$ we have that 
		\begin{equation}
		\label{eq:Dirichlet Neumann}
		\begin{split}
		F_{Dir}(t, \rho,G,g):=&\left\langle \rho_{t},  G_{t} \right\rangle -\left\langle g,   G_{0}\right\rangle - \int_0^t\left\langle \rho_{s},\Big(\partial_s + \bb L \Big) G_{s}  \right\rangle ds=0.
		\end{split}   
		\end{equation}
		\item for  $t$ a.s. in $(0,T]$, $\rho_t(0)=\alpha$ and $\rho_t(1)=\beta$.
	\end{enumerate}
\end{definition}

\begin{definition}
	\label{Def. Robin Condition}
	Let $\hat \kappa\geq 0$, $\gamma \in (1,2)$ and let $g:[0,1]\rightarrow [0,1]$ be a measurable function. We say that  $\rho^{\hat \kappa}:[0,T]\times[0,1] \to [0,1]$ is a weak solution of the regional fractional diffusion equation with fractional Robin boundary conditions and initial condition $g$:
	\begin{equation}
	\label{eq:Robin Equation}
	\begin{cases}
	&\partial_{t} \rho^{\hat \kappa}_{t}(u)= \LL \rho^{\hat \kappa}_t(u),  \quad (t,u) \in [0,T]\times(0,1), \\
&\chi_\gamma (D^\gamma \rho^{\hat \kappa}_{t})(0) = {\hat \kappa} ( \alpha-\rho^{\hat \kappa}_t (0)), \quad t \in (0,T], \\
&   \chi_\gamma (D^\gamma \rho^{\hat \kappa}_{t})(1) ={\hat \kappa}  ( \beta-\rho^{\hat \kappa}_t (1)), \quad t \in (0,T], \\
	&{ \rho}^{\hat \kappa}_{0}(u)= g(u),\quad u \in (0,1),
	\end{cases} 
	\end{equation}
	if: 
	\begin{enumerate}[i)] 
		
		\item  $\rho^{\hat \kappa} \in L^{2}(0,T;\mathcal{H}^{\gamma/2})$.
		\item For all $t\in [0,T]$ and {all functions} $G \in C^{1,\infty} ([0,T]\times (0,1))$ we have that 
		\begin{equation}
		\label{eq:Robin Neumann}
		\begin{split}
		F_{Rob}(t, \rho^{\hat \kappa},G,g):=&\left\langle \rho^{\hat \kappa}_{t},  G_{t} \right\rangle -\left\langle g,   G_{0}\right\rangle - \int_0^t\left\langle \rho^{\hat \kappa}_{s},\Big(\partial_s + \bb L \Big) G_{s}  \right\rangle ds\\
		&  -\hat \kappa \int_0^t \big\{G_s(0)(\alpha-\rho^{\hat \kappa}_s(0))+G_s(1)(\beta-\rho^{\hat \kappa}_s(1))\big\}ds=0.  
		\end{split}   
		\end{equation}
	\end{enumerate}
\end{definition}

\begin{rem}
Observe that in \eqref{eq:Robin Neumann}, the terms $\rho_s^{\hat \kappa} (0)$ and $\rho_s^{\hat \kappa} (1)$ are well defined for a.e. time $s$ since $\rho^{\hat \kappa} \in L^{2}(0,T;\mathcal{H}^{\gamma/2})$ so that for a.e. time $s$ $\rho^{\hat \kappa}_s$ has a continuous representative on $[0,1]$.
\end{rem}

\begin{definition}
	\label{Def. Dirichlet Condition_kappa^infty}
	Let {$\gamma \in (0,2)$ } and $\hat \kappa > 0$ be some parameter and let  $g:[0,1]\rightarrow [0,1]$ be a measurable function. We say that  $\rho^{\hat \kappa}:[0,T]\times[0,1] \to [0,1]$ is a weak solution of the reaction equation with non-homogeneous Dirichlet boundary condition and initial condition $g$:  
	
	\begin{equation}\label{eq:Dirichlet Equation_infty}
	\begin{cases}
	&\partial_{t} \rho_{t}^{\hat  \kappa}(u)= -\hat \kappa\rho^{\hat \kappa}_{t}(u)V_{1}(u) +\hat\kappa V_{0}(u),  \quad (t,u) \in [0,T]\times(0,1),\\
	& \rho_{t}^{\hat \kappa}(0)=\alpha, \quad { \rho_{t}^{\hat \kappa}}(1)=\beta,\quad t \in (0,T], \\
	&{ \rho}_{0}^{\hat \kappa}(u)= g(u), \quad u \in (0,1),
	\end{cases}
	\end{equation}
	if: 
	\begin{enumerate}[i)]
	\item   $\int_0^T \int_0^1 \Big\{ \frac{(\alpha-\rho_t^{\hat\kappa}(u))^2}{u^\gamma}+\frac{(\beta-\rho_t^{\hat\kappa}(u))^2}{(1-u)^\gamma}\Big\} \, du\, dt <\infty$.
\item For all $t\in [0,T]$ and {all functions} $G \in C_c^{1,\infty} ([0,T]\times (0,1))$ we have
	\begin{equation}\label{eq:Dirichlet integral_infty}
		\begin{split}
		F_{Reac}(t, \rho^{\hat \kappa},G,g):=&\left\langle \rho_{t}^{\hat \kappa},  G_{t} \right\rangle  -\left\langle g,   G_{0}\right\rangle- \int_0^t\left\langle \rho_{s}^{\hat \kappa},\partial_s G_{s}  \right\rangle ds
		\\
		&+ \hat \kappa\int^{t}_{0} \left\langle \rho_{s}^{\hat \kappa},G_s \right\rangle_{V_1} ds - \hat \kappa\int^{t}_{0}\left\langle  G_s ,V_0\right\rangle ds=0.  
		\end{split}
		\end{equation}
		\end{enumerate}
\end{definition}
\begin{rem}
	Observe that if $\gamma\in (0,1)$ the first item is trivial since $\rho^{\hat \kappa}$ is uniformly bounded by $1$. Moreover, as proved in \cite{BJGO2}, the equation \eqref{eq:Dirichlet Equation_infty} has an explicit solution in the sense of Definition \ref{Def. Dirichlet Condition_kappa^infty} which is given by 
	\begin{equation}
		\frac{V_0(u)}{V_1(u)}+\bigg[g(u)-\frac{V_0(u)}{V_1(u)}\bigg]e^{-t \hat \kappa \tfrac{V_0(u)}{V_1(u)}},
	\end{equation}
	for any $u \in [0,1]$ and any $t \in [0,T]$.
    Therefore is trivial to show that $\rho_t^{\hat \kappa}(0)=\alpha$ and $\rho_t^{\hat \kappa}(1)=\beta$ for any $t \in (0,T]$.
\end{rem}

\begin{rem}
    The reader can easily  check  that thanks to Proposition \ref{prop:guan0} the notion of weak solution coincides formally {\footnote{It is only formal since the regularity properties of the solution of the PDE are not known.}} with the corresponding strong PDE formulation for each of the definition given above.

\end{rem}

\begin{prop}\label{uniqueness}
   The weak solutions of equations\eqref{eq:Neumann Equation},\eqref{eq:Dirichlet Equation2},  \eqref{eq:Dirichlet0 Equation}, \eqref{eq:Robin Equation} and \eqref{eq:Dirichlet Equation_infty} given in the definitions above are unique.
\end{prop}
\begin{proof}
	It is detailed in Section \ref{unique_sec}.
\end{proof}

\subsection{Hydrodynamic limit}
Let ${\mc M}^+$ be the space of positive measures on $[0,1]$ with total mass bounded by $1$ equipped with the weak topology. For any configuration  $\eta \in \Omega_{N}$, we define the empirical measure $\pi^{N}(\eta,du) \in \mc M^+$ by 
\begin{equation}\label{MedEmp}
\pi^{N}(\eta, du)=\dfrac{1}{N-1}\sum _{x\in \Lambda_{N}}\eta(x)\delta_{\frac{x}{N}}\left( du\right),
\end{equation}
where $\delta_{a}$ is a Dirac mass on $a \in [0,1]$ and we denote
$$\pi^{N}_{t}(du):=\pi^{N}(\eta_t^N, du).$$

Fix $T>0$. We denote by $\PP _{\mu _{N}}$ the probability measure in the Skorohod space $\mathcal D([0,T], \Omega_N)$ induced by the  Markov process $\eta_{\cdot}^N $ with initial distribution $\mu_N$ and we denote by $\EE _{\mu _{N}}$ the expectation with respect to $\PP_{\mu _{N}}$. Let $\lbrace\mathbb{Q}^{N}\rbrace_{N\geq 1}$ be the  sequence of probability measures on $\mathcal D([0,T],\mathcal{M}^{+})$ induced by the  Markov process $\lbrace \pi_{t}^{N}\rbrace_{t\geq 0}$ and by $\mathbb{P}_{\mu_{N}}$.

\begin{definition}
Let $\rho_0: [0,1]\rightarrow[0,1]$ be a measurable function. We say that a sequence of probability measures $\lbrace\mu_{N}\rbrace_{N\geq 1 }$ on $\Omega_{N}$  is associated with the profile $\rho_{0}(\cdot)$ if for any continuous function $G:[0,1]\rightarrow \mathbb{R}$  and every $\delta > 0$ 
\begin{equation*}
\lim _{N\to\infty } \mu _{N}\left( \eta \in \Omega_{N} : \left\vert \dfrac{1}{N}\sum_{x \in \Lambda_{N} }G\left(\tfrac{x}{N} \right)\eta(x) - \int_{0}^1G(u)\rho_{0}(u)du \right\vert    > \delta \right)= 0.
\end{equation*}  
\end{definition}
The next statement is the main theorem of this work. The result is summarized in Figure \ref{fig:mesh1}. The proof is given in details in the next sections.
\begin{thm}[Hydrodynamic limit]

\label{theo:hydro_limit}
\quad

	Let $g:[0,1]\rightarrow[0,1]$ be a measurable function and let $\lbrace\mu _{N}\rbrace_{N\geq 1}$ be a sequence of probability measures in $\Omega_{N}$ associated to $g(\cdot)$. 
	Then, for any $0\leq t \leq T$,
	\begin{equation*}
	\label{limHidreform}
	\PP_{\mu _{N}}\left( \eta_{\cdot}^{N} \in \mathcal D([0,T], {\Omega_{N}}) : \left\vert \dfrac{1}{N-1}\sum_{x \in \Lambda_{N} }G\left(\tfrac{x}{N} \right)\eta_{t}^N(x) - \int_{0}^1G(u)\rho_{t}^{ \kappa}(u)du \right\vert    > \delta \right)
	\end{equation*}
	goes to $0$ as $N$ goes to infinity, where the time scale is given by 
	\begin{equation}
	\Theta(N)=\begin{cases} N^{\gamma + \theta} & \theta<0;\\
	N^{\gamma} & \theta\geq 0;
		\end{cases}
		\label{timescale}
\end{equation}
and  $\rho_{t}^{ \kappa}$ is the unique weak solution of:
	\begin{itemize}
		\item [$\bullet$] (\ref{eq:Neumann Equation}) if $\theta >0$ and { $\gamma \in (0,1]$} or if $\theta>\gamma-1$ and $\gamma \in (1,2)$;
	{\item [$\bullet$] (\ref{eq:Dirichlet Equation2}) with $\hat  \kappa= \kappa $, if $\theta =0$ and { $\gamma \in (0,1)\cup (1,2)$};}
		\item [$\bullet$] (\ref{eq:Dirichlet0 Equation})  if $\theta\in (0,\gamma-1)$ and $\gamma \in (1,2)$;
				\item [$\bullet$] (\ref{eq:Robin Equation}) with $\hat  \kappa= \kappa m$, if $\theta=\gamma-1$ and $\gamma \in (1,2)$;
				\item[$\bullet$] \eqref{eq:Dirichlet Equation_infty} with $\hat  \kappa= \kappa $, if $\theta<0$ and { $\gamma \in (0,2)$}.
	\end{itemize}
\end{thm}

\begin{rem}
We observe that last result for $\theta=0$ and $\gamma=1$, can also be derived only in the sense of weak convergence through subsequences, since we lack a proof of uniqueness of weak solutions.
\end{rem}

To prove this theorem we will use some tools of the entropy method introduced in \cite{GPV}. Therefore, we will prove that the sequence $\{\bb Q^N\}_{N\in \bb N}$ is tight and we will characterize uniquely the limiting point $\bb Q$ by showing that it is a Dirac measure over the trajectory $\pi_t(du)=\rho(t,u)du $, where $\rho(t,u)$ is the unique weak solution of the corresponding hydrodynamic equation.

\subsection{Few words about stationary solutions}
In this subsection we discuss some facts about the study of stationary solutions $\lim_{t \to \infty} \rho_t^\kappa$ of the hydrodynamic limits, presented in Theorem \ref{theo:hydro_limit}, and their possible derivations from the stationary state of the microscopic model (hydrostatic limit).  

\begin{itemize}
\item If $\gamma \in (0,1)$ and $\theta > 0$ or if $\gamma \in (1,2)$ and $\theta> \gamma -1$, we conjecture that the stationary solution depends on the initial condition $g$ and it is given by the constant $\int_0^1 g(u) du$. On the other hand, in these regimes the hydrostatic profile is different and explicitly given by    
$$\rho^{\infty}(u)=\int_0^1\dfrac{V_0(w)}{V_1(w)}dw=\dfrac{\alpha+\beta}{2}$$ for any $u \in [0,1]$. The hydrostatic limits can be derived by following the recent strategy developed in  \cite{tsu}.\\

\item If $\gamma \in (0,2)$  and $\theta<0$ the stationary solution is given by 
$$\rho^{\infty}(u)=\dfrac{V_0(u)}{V_1(u)}$$ for any $u \in [0,1]$ and coincides with the hydrostatic profile which can probably be derived by following the strategy presented in \cite{BJ} combined with the methods used in this paper. \\

\item In the other regimes the stationary solution is expected to be unique, but not explicit, connecting $\alpha$ to $\beta$, apart from the case of Robin boundary conditions ($\gamma \in (1,2)$, $\theta=\gamma -1$). Here it should also coincide with the hydrostatic profile which can, likely, be derived by following the strategy presented in \cite{BJ} and the results of the present paper. \\
\end{itemize}

The hydrostatic profiles obtained numerically are represented in Figure \ref{stationary} below. We observe that apart from the flat case, they are nonlinear and not smooth at the boundaries.  

\begin{figure}[htb!]
\includegraphics[trim={0.05cm  0.05cm 0.05cm 0.05cm}, clip, width=12cm]{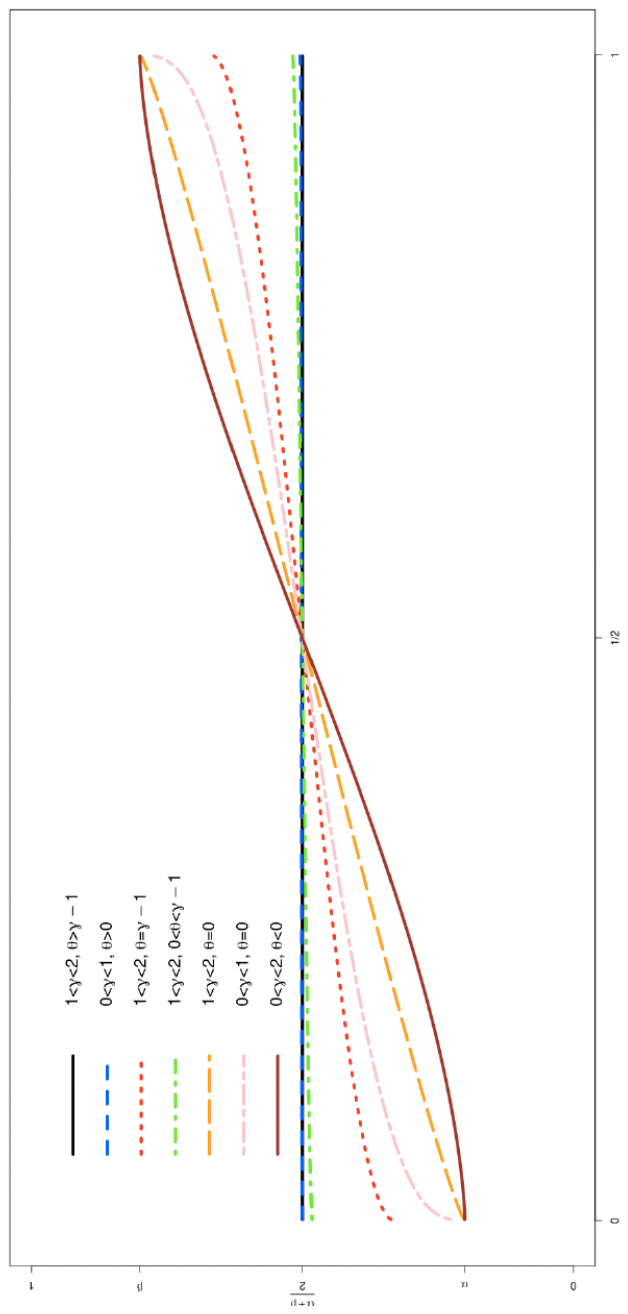} 
\caption{Here we present some numerical simulations of stationary solutions of the PDEs according to the values of $\gamma$ and $\theta$.}
\label{stationary}
\end{figure}
\subsection{Variations of the dynamics}

In this section we present some results about three models with different  dynamics but which can be studied with the techniques developed in  this paper. We do not write the details of the proofs. We consider only the infinite variance case ($\gamma<2$).

\subsubsection{Super-diffusive model with reservoirs acting only on one site}\label{strano1_def}
This particular variant of our model has the same dynamics inside the bulk but the reservoirs are only located at the sites $0$ and $N$ and they act only on the sites $1$ and $N-1$, respectively. It means that the action of the reservoirs is exactly the same as the one presented in \cite{Adriana} while the bulk dynamics is the same as the one presented in Subsection \ref{model}.
Therefore the generator associated to this model is defined by its action on functions $f:\Omega_N \rightarrow \RR$ as:
\begin{equation}
\begin{split}
L_Nf(\eta)&=\frac{1}{2} \sum_{x,y \in \Lambda_N} p(x-y) [f(\sigma^{x,y}\eta) -f(\eta)]\\
&+\frac{\kappa}{N^{\theta}}[\alpha(1-\eta(1))+(1-\alpha)\eta(1)]f(\sigma^{1}\eta) -f(\eta)]\\
&+\frac{\kappa}{N^{\theta}}[\beta(1-\eta(N-1))+(1-\beta)\eta(N-1)]f(\sigma^{N-1}\eta) -f(\eta)].
\end{split}
\end{equation}

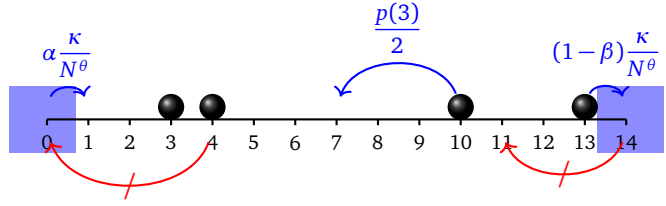
\begin{figure}[htb!]
	
	\begin{center}
		\begin{tikzpicture}[thick, scale=1.1]
		
		\fill [color=blue!45] (-3.95,-0.4) rectangle (-3.15,0.4);
		\fill [color=blue!45] (3.95,-0.4) rectangle (3.15,0.4);
		
		\draw[] (-3.5,0) -- (3.5,0) ;
		\foreach \x in {-3.5,-3,-2.5,...,3.5}
		\pgfmathsetmacro\result{\x*2+7}
		\draw[shift={(\x,0)},color=black] (0pt,0pt) -- (0pt,-2pt) node[below]{\scriptsize \pgfmathprintnumber{\result}};
		
		\node[ball color=black, shape=circle, minimum size=0.3cm] (B) at (-1.5,0.15) {};
		
		\node[ball color=black, shape=circle, minimum size=0.3cm] (C) at (1.5,0.15) {};
		
		\node[ball color=black, shape=circle, minimum size=0.3cm] (D) at (3,0.15) {};
		
		\node[ball color=black, shape=circle, minimum size=0.3cm] (E) at (-2,0.15) {};
		
		\node[draw=none] (H) at (-1.5,-0.15) {};
		\node[draw=none] (W) at (-3.5,-0.15) {};
		\node[draw=none] (S) at (3.5,0.15) {};
		\node[draw=none] (R) at (0,0.15) {};
		\node[draw=none] (L) at (-3.5,0.15) {};
		\node[draw=none] (M) at (-3,0.15) {};
				\node[draw=none] (HHH) at (3.5,-0.15) {};
						\node[draw=none] (HH) at (2,-0.15) {};

		\path [<-] (S) edge[bend right =70, color=blue]node[above] {\footnotesize $(1-\beta)\dfrac{\kappa}{N^{\theta}}$}(D);
		\path [->] (C) edge[bend right =70, color=blue]node[above] {\footnotesize $\dfrac{p(3)}{2}$}(R);			
		\path [<-] (M) edge[bend right =70, color=blue]node[above] {\footnotesize $\alpha\dfrac{\kappa}{N^{\theta}}$}(L);
		\path [<-] (W) edge[bend right =70, color=red]node[above] {} node {/}(H);
		\path [<-] (HH) edge[bend right =70, color=red]node[above] {} node {/}(HHH);

		\end{tikzpicture}
		\caption{Example of the dynamics of the model presented in Section \ref{strano1_def}, with $N=14$.}	
		\label{fig_strano1}
	\end{center}	
\end{figure}

In this case we need to accelerate the system by a factor which does not depend on $\theta$ and which is constantly equal to $\Theta(N)=N^{\gamma}$. Then, following the same strategies used in this work, it is possible to prove that the hydrodynamic limits associated to this model are given by the following hydrodynamic equations (summarized in Figure \ref{strano1}):
\begin{itemize}
	\item if $\theta<\gamma-1$ we get the  equation of Definition \ref{Def. Neumann Condition};
	
	\item if $\theta = \gamma-1$ we get the equation of Definition \ref{Def. Robin Condition} with $m=1$ and $\hat \kappa = \kappa$;
	\item if $\theta > \gamma-1$ we get the equation of Definition \ref{Def. Neumann Condition}.
\end{itemize}

\begin{figure}[htb!]
	\begin{center}
	\begin{tikzpicture}[scale=0.20]
	\fill[color=red!35] (-25,-12) -- (-25,-18)--  (5,-18)-- (5,2)-- cycle;
	\fill[color=green!55] (-25,-12) -- (-10,-5) -- (5,2) -- (5,8) -- (-25,8) -- cycle;
	\draw[dotted, black] (-25,-5) -- (5,-5);
		\draw[dotted, black] (-25,2) -- (5,2);
			\draw[dotted, black] (-25,-12) -- (5,-12);
	\draw[dotted, ultra thick, white] (-25,10) -- (-25,-20);
	\draw[dotted, ultra thick, white] (-10,10) -- (-10,-20);
	\draw[dotted, ultra thick, white] (5,10) -- (5,-20);
	\draw[-,=latex,blue ,ultra thick] (-25,-12) -- (5, 2) node[midway, sloped, above] {{{\textbf{\small{\textcolor{blue}{Frac. Diff. \& Robin b.c.}}}}}};

	\node[right, white] at (-21,6) {\textbf{\small{\textcolor{dark green}{Frac. Diff. \& Neumann b.c.}}}} ;
	\node[right, white] at (-21,-14) {\textbf{\small{\textcolor{red}{Frac. Diff.  \& Dirichlet b.c.}}}} ;
	\node[rotate=270, below] at (5,1) {\textcolor{white}{$\gamma = 2$}};
	\node[rotate=270, above] at (-10,1) {\textcolor{white}{$\gamma = 1$}};
	\node[rotate=270, above] at (-25,1) {\textcolor{white}{$\gamma = 0$}};
	\node[left] at (-25,-5) {\textcolor{black}{$\theta=0$}};
		\node[left] at (-25,-12) {\textcolor{black}{$\theta=-1$}};
			\node[left] at (-25,2) {\textcolor{black}{$\theta=1$}};
	
	\draw[<-,blue]  (-5,-3) -- (-3.5,-7) node[right] {\small{\textcolor{blue}{$\theta=\gamma-1$}}};
	
	\end{tikzpicture}
	\caption{Summary of result on hydrodynamic limits for the model introduced in Section \ref{strano1_def}.}
	\label{strano1}
	\end{center}
\end{figure}
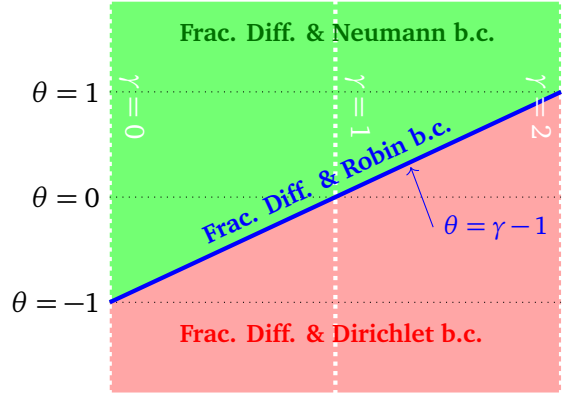

\subsubsection{Diffusive model with reservoirs acting on all the sites}\label{strano2_def}
In this other variation of our model, the reservoirs have the same dynamics as the one described in Section \ref{model}, while in the bulk the particles move according to the nearest-neighbour dynamics defined in \cite{Adriana}, that is in the bulk the transition probability $p(\cdot)$ has range $1$. Therefore the generator associated to this model is given by its action on functions $f:\Omega_N \rightarrow \RR$ as:
\begin{equation}
\begin{split}
L_Nf(\eta)&=\sum_{x=1}^{N-2}[f(\sigma^{x,x+1}\eta)-f(\eta)]\\
&+\frac{\kappa}{N^{\theta}}\sum_{x \in \Lambda_N}\sum_{y\leq 0} p(x-y)c_{x}(\eta;\alpha) [f(\sigma^x\eta) - f(\eta)]\\
&+ \frac{\kappa}{N^{\theta}}\sum_{x \in \Lambda_N }\sum_{y\geq N} p(x-y) c_{x}(\eta;\beta)  [f(\sigma^x\eta) - f(\eta)]
\end{split}
\end{equation}
where $c_x(\eta,\alpha)$ and $c_x(\eta,\beta)$ were defined in \eqref{rate_c}.  

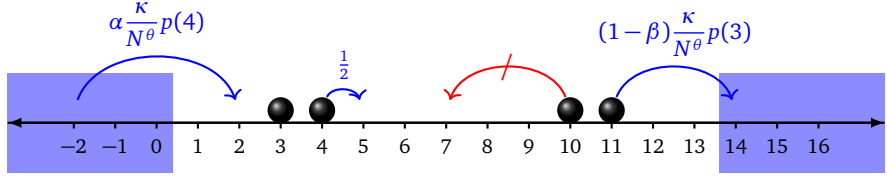
\begin{figure}[htb!]
	
	\begin{center}
		\begin{tikzpicture}[thick, scale=1.1]
		
		\draw[shift={(-5.01,-0.15)}, color=black] (0pt,0pt) -- (0pt,0pt) node[below]{\dots};
		\draw[shift={(5.01,-0.15)}, color=black] (0pt,0pt) -- (0pt,0pt) node[below]{\dots};
		
		\fill [color=blue!45] (-5.3,-0.6) rectangle (-3.3,0.6);
		\fill [color=blue!45] (5.3,-0.6) rectangle (3.3,0.6);

	\draw[-latex] (-5.3,0) -- (5.3,0) ;
	\draw[latex-] (-5.3,0) -- (5.3,0) ;
	\foreach \x in {-4.5,-4,-3.5,...,4.5}
		\pgfmathsetmacro\result{\x*2+7}
		\draw[shift={(\x,0)},color=black] (0pt,0pt) -- (0pt,-2pt) node[below]{\scriptsize \pgfmathprintnumber{\result}};
		
		\node[ball color=black, shape=circle, minimum size=0.3cm] (B) at (-1.5,0.15) {};
		
		\node[ball color=black, shape=circle, minimum size=0.3cm] (C) at (1.5,0.15) {};
		
		\node[ball color=black, shape=circle, minimum size=0.3cm] (D) at (2,0.15) {};
		
		\node[ball color=black, shape=circle, minimum size=0.3cm] (E) at (-2,0.15) {};
		
		\node[draw=none] (S) at (3.5,0.15) {};
		\node[draw=none] (R) at (0,0.15) {};
		\node[draw=none] (L) at (-4.5,0.15) {};
		\node[draw=none] (M) at (-2.5,0.15) {};
		\node[draw=none] (MM) at (-1,0.15) {};
		
		\path [<-] (S) edge[bend right =70, color=blue]node[above] {\footnotesize $(1-\beta)\dfrac{\kappa}{N^{\theta}}p(3)$}(D);
		\path [->] (C) edge[bend right =70, color=red]node[above] {} node{/} (R);			
		\path [<-] (M) edge[bend right =70, color=blue]node[above] {\footnotesize $\alpha\dfrac{\kappa}{N^{\theta}}p(4)$}(L);
		\path [<-] (MM) edge[bend right =70, color=blue]node[above] {\footnotesize $\frac{1}{2}$}(B);

		\end{tikzpicture}
		\caption{Example of the dynamics of the model presented in Section \ref{strano2_def}, with $N=14$.}	
		\label{fig_strano2}
	\end{center}	
\end{figure}

The hydrodynamic equations associated to this model (summarized in Figure \ref{strano2}) 
are given by:
\begin{itemize}
	\item For $\theta>2-\gamma$, if we accelerate the system by a factor $\Theta(N)=N^2$, we get the heat equation with Neumann boundary conditions i.e.
	\begin{equation}
	\begin{cases}
	\partial_t\rho_t(u)=\Delta \rho_t(u) &  (t,u) \in [0,T] \times (0,1)  ;\\
	\partial_u\rho_t(0)=\partial_u\rho_t(1)=0 & t \in (0,T];\\
	\rho_0(u)=g(u) & u \in (0,1);
	\end{cases}
	\end{equation}
	for some arbitrary initial condition $g:[0,1]\rightarrow \RR$ and with the notion of weak solution introduced in Definition 2.7 of \cite{Adriana};
	\item For $\theta=2-\gamma$, if we accelerate the system by a factor $\Theta(N)=N^2$, we get the reaction-diffusion equation with Dirichlet boundary conditions i.e.
	\begin{equation}
	\begin{cases}
	\partial_t\rho_t(u)=\Delta \rho_t(u)+V_0(u)\rho_t(u) &(t,u) \in [0,T] \times (0,1);\\
	\rho_t(0)=\alpha \text{ and } \rho_t(1)=\beta & t \in (0,T];\\
	\rho_0(u)=g(u) & u \in (0,1);
	\end{cases}
	\end{equation}
	for some arbitrary initial condition $g:[0,1]\rightarrow \RR$ and with the notion of weak solution introduced in Definition 2.2 of \cite{BGJO} but with $\sigma^2/2$ that now is equal to $1$;
	\item For $\theta<2-\gamma$, if we accelerate the system by a factor $\Theta(N)=N^{\gamma+\theta}$ we get the reaction equation introduced in Definition \ref{Def. Dirichlet Condition_kappa^infty} with $\hat \kappa = \kappa$.
\end{itemize}

\begin{figure}[htb!]
	\begin{center}
		\begin{tikzpicture}[scale=0.20]
		\fill[color=red!35] (-25,5) -- (-25,-20)--  (5,-20)-- (5,-15)-- cycle;
		\fill[color=blue!20] (-25,5) -- (-10,-5) -- (5,-15) -- (5,10) -- (-25,10) -- cycle;
		\draw[dotted, black] (-25,5) -- (5,5);
		\draw[dotted, black] (-25,-15) -- (5,-15);
		\draw[dotted, ultra thick, white] (-25,10) -- (-25,-20);
		\draw[dotted, ultra thick, white] (-10,10) -- (-10,-20);
		\draw[dotted, ultra thick, white] (5,10) -- (5,-20);
		\draw[-,=latex,blue ,ultra thick] (-25,5) -- (5, -15) node[midway, sloped, above] {{{\textbf{\small{\textcolor{blue}{Reac. Diff.  \& Dirichlet b.c.}}}}}};
		
		\node[right, white] at (-20,6) {\textbf{\small{\textcolor{dark blue}{Diffusion \& Neumann b.c.}}}} ;
		\node[right, white] at (-20,-14) {\textbf{\small{\textcolor{red}{Reaction \& Dirichlet b.c.}}}} ;
		\node[rotate=270, below] at (5,1) {\textcolor{white}{$\gamma = 2$}};
		\node[rotate=270, above] at (-10,1) {\textcolor{white}{$\gamma = 1$}};
		\node[rotate=270, above] at (-25,1) {\textcolor{white}{$\gamma = 0$}};
		\node[left] at (-25,-15) {\textcolor{black}{$\theta=0$}};
		\node[left] at (-25,5) {\textcolor{black}{$\theta=2$}};
		
		\draw[<-,blue]   (-14,-3)--(-16,-6) node[below] {\small{\textcolor{blue}{$\theta=2-\gamma$}}};
		
		\end{tikzpicture}
		\caption{Summary of the results on hydrodynamic limit for the model introduced in Section \ref{strano2_def}.}
		\label{strano2}
	\end{center}
\end{figure}
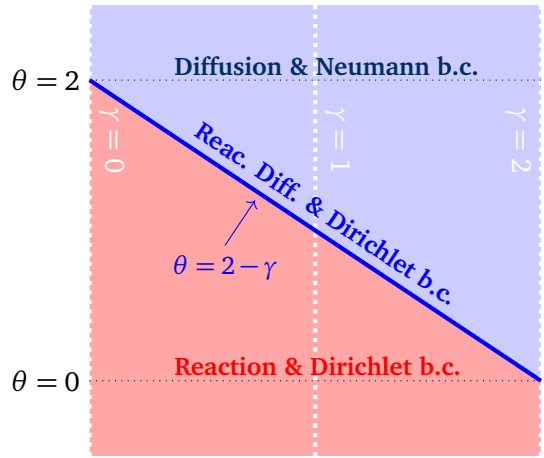

\subsubsection{One site reservoirs' with long jumps}
Another interesting case which deserves a mention is the model with the same dynamics  as the ones introduced in Subsection \ref{model} but with reservoirs only at the sites $0$ and $N$. This was studied in the lecture notes \cite{patricianote} and the analysis is almost analogous to the one we do in this work and therefore we will not treat this case here.

\section{Proof of  Theorem \ref{theo:hydro_limit}}
\subsection{Heuristics for the hydrodynamic equations}\label{heursitcs}
In this section we give the main ideas of the proof of  the derivation of the weak solution of the  hydrodynamic equation for each regime of $\theta$ and $\gamma$. Let us, by now, assume that the probability measure $\mathbb Q^N$ on $\mathcal{D}([0,T],\mathcal{M}^+)$ defined above,  converges to a measure $\mathbb Q$ on the same space, concentrated on a trajectory of measures {$\pi_t(du)$  }, which  is  absolutely continuous with respect to the Lebesgue measure with a density {$\rho(t,u)$ }. Moreover, we will assume Theorem \ref{Energy_Thm1}, which will be proved later. This means that we can assume  that $\rho \in L^{2}(0,T;\mc H^{\gamma/2})$ and that $\rho_s(0)$ is well-defined for almost every $s \in [0,T]$.  We will prove the aforementioned convergence, up to a subsequence, in the next subsections.

For simplicity we will forget the time dependence on the test functions, which would only make the notation heavier without bringing additional difficulties. Therefore,  consider a test function $G:[0,1]\rightarrow \RR$ for which $\mathbb{L}G$ is well defined. From Dynkin's formula (Lemma A.5 of \cite{KL}) we have that 
\begin{equation}\label{Dynkin'sFormula}
M_{t}^{N}(G)= \langle \pi_{t}^{N},G\rangle -\langle \pi_{0}^{N},G\rangle-\int_{0}^{t}\Theta(N) L_{N}\langle \pi_{s}^{N},G\rangle ds,
\end{equation}
is a martingale with respect to the natural filtration $\{\mc F^N_{t} \}_{t\geq 0}$, where  $\mathcal{F}^N_{t}:=\sigma(\lbrace\eta^N_s\rbrace_{s \leq t})$ for all $t\in [0,T]$. In the previous display we denoted by $\langle \pi_{s}^{N},G\rangle$ the integration of the function $G$ on $[0,1]$ with respect to the empirical measure $\pi^N_s$, i.e.
\begin{equation*}
\langle \pi_{s}^{N},G\rangle=\dfrac{1}{N-1}\sum _{x\in \Lambda_{N}}G\left( \tfrac{x}{N}\right)\eta^N_{s}(x).
\end{equation*}
Now, we compute the explicit value of the integrand function in \eqref{Dynkin'sFormula}, which is    given by
\begin{equation}
		\label{gen_action}
		\begin{split}
		\Theta(N) &L_N \langle \pi_s^N, G \rangle  = \cfrac{\Theta(N)}{N-1} \sum_{x\in \Lambda_N}  (\mc L_NG)(\tfrac{x}{N}) \eta^N_s(x) \\
		&\quad \quad \quad + \cfrac{ \kappa \Theta(N)}{(N-1)N^{\theta}} \sum_{x \in \Lambda_N}  G(\tfrac{x}{N}) \left( r_{N}^{-}(\tfrac{x}{N}) (\alpha- \eta^N_s(x))+   r_{N}^{+}(\tfrac{x}{N}){(\beta- \eta^N_s(x) )}\right),
		\end{split}
		\end{equation}
where,  $\mc L_{N} G$ is defined on test functions $G$ by
\begin{equation}
\label{LN}
({\mc L}_N G) (\tfrac{x}{N}) = \sum_{y \in \Lambda_N} p(y-x) \left[ G(\tfrac{y}{N}) -G(\tfrac{x}{N})\right],
\end{equation}
for all  $x\in \Lambda_N$ with $(\mc L_{N} G)(0) = (\mc L _{N}G)(1) = 0$ and $r^\pm(x/N)$ are defined as
\begin{equation}
r^-_N(\tfrac{x}{N})=\sum_{y\geq x}p(y) \quad \text{ and } \quad r^+_N(\tfrac{x}{N})=\sum_{y\leq N- x} p(y).
\end{equation}

We have to study the various terms in \eqref{gen_action} in the different regimes of  $\gamma \in (0,2)$  and $\theta \in \mathbb{R}$. Recall that the cases $\theta \leq 0$ and $\gamma \in (1,2)$ were already  treated in \cite{BJGO2} and,  for that reason, they will not be considered in this work. 

\subsubsection{Case $\theta <0$ and  $\gamma \in (0,1]$ } \label{reactionheur}
In this regime the time scale  is given by $\Theta(N)=N^{\gamma+\theta}$ and the test functions are in $ C_c^{\infty}([0,1])$. Let us first analyze the first term on the right-hand side of \eqref{gen_action}. Recalling \eqref{LN} and performing a Taylor expansion of $G$ around $x/N$, we can bound from above this term by a constant times
\begin{equation}
\cfrac{N^{\gamma+\theta}}{N^2} \sum_{x\in \Lambda_N} \sum_{\substack{y \in \Lambda_N\\y\neq x}} \cfrac{1}{|y-x|^{\gamma}} 
\end{equation}
plus lower order terms with respect to $N$, since $G'$ is uniformly bounded and $|\eta^N_s(x)|\leq 1$. Then, changing variables, that term can be bounded from above by a constant times
\begin{equation} \label{2141}
\cfrac{N^{\gamma+\theta}}{N} \sum_{x=1}^{N-1} \cfrac{1}{N} \sum_{z=1}^{x} \cfrac{1}{z^{\gamma}}\lesssim N^{\theta-1}\sum_{x=1}^{N-1}\int_{\tfrac{1}{N}}^{\tfrac{x}{N}}u^{-\gamma}du\lesssim N^{\theta-1}\sum_{x=1}^{N-1}(\tfrac{x}{N})^{-\gamma+1}\lesssim N^{\theta},
\end{equation}
assuming $\gamma \in (0,1)$. If $\gamma=1$, the left-hand side of \eqref{2141} is 
\begin{equation} \label{2142}
 N^{\theta}  \sum_{x=1}^{N-1} \cfrac{1}{N} \sum_{z=1}^{x} \cfrac{1}{z^{-1}} \lesssim N^{\theta}  \sum_{x=1}^{N-1} \cfrac{1+\log(x)}{N}  \lesssim N^{\theta} [1 + \log(N)].
\end{equation}
So, since $\theta<0$, the term in the last display vanishes in \eqref{2141} and \eqref{2142}, as $N$ goes to infinity.

Let us now analyze the second terms on the right-hand side of \eqref{gen_action}. We treat in details just the term involving $r^-_N$ since the other one is completely analogous. Since $G$ has compact support, that   term can be rewritten as
\begin{equation}
\begin{split}
& \cfrac{ \kappa}{N-1} \sum_{x =aN}^{N-1}  G(\tfrac{x}{N}) \left(N^{\gamma}r_{N}^{-}(\tfrac{x}{N})-r^-(\tfrac{x}{N})\right) (\alpha- \eta^N_s(x))\\
&+\cfrac{ \kappa}{N-1} \sum_{x =aN}^{N-1}  G(\tfrac{x}{N}) r^-(\tfrac{x}{N})(\alpha- \eta^N_s(x)),
\end{split}
\label{reac1}
\end{equation}
for some $a \in (0,1)$. We will show now that the term on the left-hand side in the previous display vanishes as $N$ goes to infinity. In order to do that, we use the estimate $$|N^{\gamma}r^-_N(\tfrac zN)-r^-(\tfrac zN)|\leq c_{\gamma}N^{-1}(\tfrac zN)^{-\gamma-1}$$ when $z/N>a$ for some fixed $a \in (0,1)$, proved in the Appendix of \cite{BJ}. Using this estimate, the fact that $(\alpha-\eta^N_s(x))$ is uniformly bounded and $G$ has compact support, we get that the  term on the left-hand side of \eqref{reac1} is bounded from above by a constant times
 \begin{equation}
\cfrac{\kappa}{N^2} \sum_{x = a N}^{N-1}  (\tfrac{x}{N})^{-\gamma-1}\lesssim \cfrac{1}{N}\int_{a}^{1}u^{-\gamma-1}du\lesssim N^{-1},
\end{equation}
which vanishes as $N$ goes to infinity. Finally, we prove that the second term of \eqref{reac1} converges to 
\begin{equation}
\label{cont2}
\kappa \int_0^1 \alpha G(u)r^-(u)du-\kappa\int_0^1 G(u)r^-(u)\rho^{\kappa}_s(u)du,
\end{equation}
provided that $\pi_s^N$ converges to an absolutely continuous measure, with respect to the Lebesgue measure on $[0,1]$,  with density $\rho_s^{\kappa}(\cdot)$. In order to do that let us denote $H:=G r^-$ which is a function in $C_c^{\infty}([0,1])$. Then, the last term in \eqref{gen_action} is equal to
\begin{equation}
\frac{\kappa}{N-1}\sum_{x\in \Lambda_N}H(\tfrac{x}{N})(\alpha-\eta_s^N(x)),
\end{equation}
plus terms vanishing as $N$ goes to infinity.
Which, under the hypothesis of the convergence of $\pi^N$ that we just did, is exactly the Riemann series converging to \eqref{cont2}.
Repeating the same proof for the term on the right-hand side of \eqref{gen_action} involving $r^+_N$, we get that the whole term converges to
\begin{equation}
\kappa \int_0^1G(u)V_0(u)du-\kappa\int_0^1 G(u)V_1(u)\rho^\kappa_s(u)du,
\end{equation}
as $N$ goes to infinity, which is the term involved in the definition of the weak solution for this regime of $\theta$ and $\gamma$.

\subsubsection{Case $\theta=0$ and  $\gamma \in (0,1]$  } \label{thetazero}
In this regime the time scale is $\Theta(N)=N^{\gamma}$ and the test functions are again in $ C_c^{\infty}([0,1])$. Then, the first term on the right-hand side of \eqref{gen_action} can be written as
\begin{equation}
\label{fraclap}
\cfrac{1}{N-1} \sum_{x\in \Lambda_N}\big( N^{\gamma} (\mc L_NG)(\tfrac{x}{N})-\mathbb{L}G(\tfrac{x}{N}) \big)\eta^N_s(x)+\cfrac{1}{N-1} \sum_{x\in \Lambda_N}\mathbb{L}G(\tfrac{x}{N}) \eta^N_s(x).
\end{equation}
The first term on the left-hand side of the previous display vanishes as $N$ goes to infinity, as a consequence of the fact that $|\eta^N_s(x)|\leq 1$ and Lemma \ref{L1}. Now, using an argument similar to the one used in the proof of Lemma \ref{L1}, it is possible to approximate the operator $\mathbb{L}$ by an operator $\mathbb{L}_{\epsilon}$ in such a way that $\mathbb{L}_{\epsilon}G$ is a continuous function on $(0,1)$ and $\lim_{\epsilon \rightarrow 0}||\mathbb{L}_{\epsilon}G-\mathbb{L}G||_{{L^1(0,1)}}=0$. 

Therefore we can rewrite the term on the right-hand side of \eqref{fraclap} as
\begin{equation*}
\cfrac{1}{N-1} \sum_{x\in \Lambda_N}\big[\mathbb{L}G(\tfrac{x}{N})-\mathbb{L}_{\epsilon}G(\tfrac{x}{N}) \big]\eta^N_s(x)+\cfrac{1}{N-1} \sum_{x\in \Lambda_N}\mathbb{L}_{\epsilon}G(\tfrac{x}{N}) \eta^N_s(x).
\end{equation*}
Then, reasoning as we did in the proof of Lemma \ref{L1}, it is possible to show that the first term of the previous display vanishes as $\epsilon \rightarrow 0$ and the second one converges, as $N\rightarrow \infty$ and $\epsilon \rightarrow 0$, to
\begin{equation}
\int_0^1 \mathbb{L}G(u)\rho^\kappa_s(u)du.
\label{fraclapreacdiff}
\end{equation}

The second term on the right-hand side of \eqref{gen_action} can be treated exactly as in the previous case, where we showed that it converges to \eqref{cont2}
as $N$ goes to infinity. Note that \eqref{fraclapreacdiff} and \eqref{cont2} are exactly the terms involved in the definition of weak solution in this regime of $\theta$ and $\gamma$.

\subsubsection{Case  $\gamma \in (0,1]$   and $\theta>0$ or $\gamma \in (1,2)$ and $\theta>\gamma-1$} \label{thetapos}
In these regimes the time scale is again $\Theta(N)=N^{\gamma}$ and the test functions are  in $ C^{\infty}([0,1])$. In both cases,  the analysis of the first term in $\eqref{gen_action}$ is completely analogous to the one that we did in the previous subsection. Indeed, using Lemma \ref{L1} it is possible to show, with the assumption  on the convergence of $\pi^N$, that this term converges to \eqref{fraclapreacdiff} with $G \in C^{\infty}([0,1])$.
This is the only term involved in the definition of weak solution in these regimes. Indeed, we show now that the second term on the right-hand side of \eqref{gen_action} vanishes, as $N$ goes to infinity. Since $G(\tfrac{x}{N})$ and $\alpha-\eta^N_s(x)$ are bounded, we can {bound from above} the part involving $r^-_N$ on  the second term on the right-hand side of \eqref{gen_action} as a constant times 
\begin{equation}
N^{\gamma-\theta-1}\sum_{x\in \Lambda_N}r^-_N(\tfrac{x}{N}).
\label{bound}
\end{equation}
Now, observe that: 
\begin{itemize}
	\item in the case $\gamma \in (0,1]$   the sum in \eqref{bound} is diverging so we have to bound it in some way. Observing that $y^{-\gamma-1}$ is a decreasing function we can perform the following bound \begin{equation}
	\begin{split}
	\sum_{x\in \Lambda_N}&r^-_N(\tfrac{x}{N})=c_{\gamma}\sum_{x\in \Lambda_N}\sum_{y\geq x}y^{-\gamma-1}\lesssim r^-_N(\tfrac{1}{N}) +\sum_{x=2}^{N-1}\sum_{y\geq x}y^{-\gamma-1}\\&\leq \frac{1}{2} + \sum_{x=2}^{N-1}\int_{x-1}^{+\infty}y^{-\gamma-1}dy\leq \frac{1}{2} + \frac{1}{\gamma}\sum_{x=1}^{N-2}x^{-\gamma}, 
	\end{split}
	\end{equation}
	since $r^-_N(\tfrac{1}{N})=\frac{1}{2}$. We observe that the sum in the last display can be bounded by $N^{-\gamma+1}$ if $\gamma \in (0,1)$ and by $\log(N)$ if $\gamma =1$.  So, we can bound \eqref{bound} by $$\frac{N^{\gamma-\theta-1}}{2}+\frac{N^{-\theta}}{\gamma}\lesssim N^{-\theta}$$ if $\gamma \in (0,1)$ and by $$\frac{N^{-\theta}}{2}+\frac{\log(N)}{\gamma}\lesssim N^{-\theta}[1+\log(N)]$$ if $\gamma=1$ and both vanish, as $N$ goes to infinity; 
	\item in the case $\gamma \in (1,2)$ the sum in \eqref{bound} is converging \textcolor{black}{since, in this regime, $\sum_{x \in \Lambda_N}r^-_N(\tfrac{x}{N})$ converges to $m$ (see \cite{BJGO2})}. So the order of the whole term is $N^{\gamma-\theta-1}$, therefore it vanishes, as $N$ goes to infinity, since $\theta >\gamma-1$.
\end{itemize} \color{black}
The term involving $r^+_N$ in \eqref{gen_action} is treated in the same way. Notice that the integration by parts formula (Proposition \ref{prop:ip-guan}) plus the boundary conditions of Definition \ref{Def. Neumann Condition} implies \eqref{eq:Neumann integral} (under the assumption of convergence of $\pi^N$).

\subsubsection{Case $\gamma \in (1,2)$ and $0<\theta< \gamma-1$}
In this regime the time scale is $\Theta(N)=N^{\gamma}$ and the test functions are in $C^{\infty}_c ([0,1])$. The analysis of the first term at the right-hand side of \eqref{gen_action} is analogous to the one of the previous case. 

Let us then show how we treat the second term on the right-hand side of \eqref{gen_action}. We will use the fact that the test function $G$ satisfies $G(x)=0$ for $x \in \{0,1\}$. Observe indeed that the term in \eqref{gen_action} involving $r^-_N$ can be rewritten as a constant times
\begin{equation*}
N^{\gamma-\theta-1}\sum_{x\in \Lambda_N}x^{-\gamma}G(\tfrac{x}{N}),
\end{equation*}
since the order of $r^-_N(x/N)$ is $x^{-\gamma}$ when we are considering ${x}>aN$ for some fixed $a \in (0,1)$ (see Lemma 3.3 of \cite{BJ}) and this is true whenever the test function has compact support.
Then, performing a Taylor expansion on $G$ around $0$ we can bound it from above  by a constant times
\begin{equation*}
N^{\gamma-\theta-2}\sum_{x\in \Lambda_N}x^{-\gamma+1},
\end{equation*}
plus lower order terms in $N$.
Therefore, since the sum in the last display is  of order $N^{-\gamma+2}$, the order of the whole term in the last display is $N^{-\theta}$,  and it vanishes as $N$ goes to infinity. The same analysis can be done for the term involving $r^+_N$ in \eqref{gen_action}.

The item (iii) of Definition \ref{Def. Dirichlet0 Condition} also holds, indeed
the fact that $\rho_t(0)=\alpha$ and $\rho_t(1)=\beta$ follows by standard arguments thanks to Proposition \ref{prop} (see Appendix A.4 of \cite{patricianote} for details). 

\subsubsection{Case $\gamma \in (1,2)$ and $\theta= \gamma-1$}\label{robinheur}
In this regime the time scale is $\Theta(N)=N^{\gamma}$ and the test functions are in  $C^{\infty}([0,1])$. The analysis of the first term on the right-hand side of \eqref{gen_action} is analogous to the previous cases and is done using Lemma \ref{L1}.

The analysis of the second term on the right-hand side of \eqref{gen_action} is a bit different from the previous cases.

Consider first the part involving $r^-_N$. For any $\epsilon >0$, we can rewrite it as
\begin{equation*}
\begin{split}
&\cfrac{ \kappa N^{\gamma}}{(N-1)N^{\gamma-1}} (\alpha-\overrightarrow \eta_s^{\epsilon N}(0))\sum_{x \in \Lambda_N}  G(\tfrac{x}{N}) r_{N}^{-}(\tfrac{x}{N}) \\
&+\cfrac{ \kappa N^{\gamma}}{(N-1)N^{\gamma-1}} \sum_{x \in \Lambda_N}  G(\tfrac{x}{N}) r_{N}^{-}(\tfrac{x}{N}) (\overrightarrow \eta_s^{\epsilon N}(0)-\eta^N_s(x)),
\end{split}
\end{equation*}
where $\overrightarrow \eta_s^{\epsilon N}(0)$ is defined in \eqref{mean}.
On the other hand, the term on the right-hand side of the previous display vanishes, as $N$ goes to infinity and then as $\epsilon$ goes to $0$, thanks to Remark \ref{robin_rem}.

The term on the left-hand side of the last display, performing a Taylor expansion on $G$, can be written as
\begin{equation}
\cfrac{ \kappa N^{\gamma}}{(N-1)N^{\gamma-1}} G(0)(\alpha- \overrightarrow \eta_s^{\epsilon N}(0))\sum_{x \in \Lambda_N}  r_{N}^{-}(\tfrac{x}{N})+\cfrac{\kappa}{N-1} G'(0)(\alpha- \overrightarrow \eta_s^{\epsilon N}(0))\sum_{x \in \Lambda_N}   xr_{N}^{-}(\tfrac{x}{N})
\label{robin}
\end{equation}
plus lower order terms in $N$. Observe that the term on the right-hand side of the previous display can be bounded by a constant times
\begin{equation*}
\cfrac{1}{N-1} \sum_{x \in \Lambda_N}   x^{-\gamma+1} = O(N^{-\gamma+1}),
\end{equation*}
 which goes to $0$ as $N$ goes to infinity, since $\gamma >1$. The remaining term in \eqref{robin} converges in some sense to $G(0)(\alpha-\rho^\kappa_s(0))m$, as  $N$ goes to infinity, since $\lim_{N \rightarrow \infty}\sum_{x \in \Lambda_N}  r_{N}^{-}(\tfrac{x}{N})=m$ (see, for example, \cite{BJGO2}, for details about this computation) and since, in some sense, well explained in \cite{patricianote}, $\overrightarrow \eta_s^{\epsilon N}(0)\rightarrow \rho^{\kappa}_s(0)$. Then, reasoning in a similar way as before,  we can conclude that, for $N$ going to infinity, the second term on the right-hand side of \eqref{gen_action} converges to
\begin{equation*}
mG(0)(\alpha-\rho^{\kappa}_s(0))+mG(1)(\beta-\rho^{\kappa}_s(1)).
\end{equation*}
This is exactly the term involved in \eqref{eq:Robin Neumann} and also that gives us, using the integration by parts formula (see Proposition \ref{prop:ip-guan}), the boundary conditions in the Definition \ref{Def. Robin Condition}.

\subsection{Tightness} \label{subsec:Tightness}
In this subsection we prove that the sequence $\lbrace  \mathbb {Q}^{N} \rbrace_{N \geq 1} $ is tight. Despite our approach being standard (see \cite{BGJO}), we decided to write all the arguments for the convenience of the reader.  To prove tightness, we use  Proposition 4.1.6 in \cite{KL}, from where it  is enough to show that, for all $\epsilon >0$ 
\begin{equation}
\label{T1}
\displaystyle \lim _{\delta \rightarrow 0} \limsup_{N\rightarrow\infty} \sup_{\tau  \in \mathcal{T}_{T},\bar\tau \leq \delta} {\mathbb{P}}_{\mu _{N}}\Big[\eta_{\cdot}^{N}\in {\mathcal D} ( [0,T], \Omega_{N}) :\left\vert \langle\pi^{N}_{\tau+ \bar\tau},G\rangle-\langle\pi^{N}_{\tau},G\rangle\right\vert > \epsilon \Big]  =0, 
\end{equation}
for any function $G$ belonging to $C^0([0,1])$ . Above $\mathcal{T}_{T}$ is the set of stopping times bounded by $T$ and we implicitly assume that all the stopping times are bounded by $T$, thus, $\tau+ \bar\tau$ should be read as $ (\tau+ \bar\tau) \wedge T$. Indeed, we prove below that (\ref{T1}) is true for any function in $C_{c}^{2}([0,1])$, and then, by using an $L^1$ approximation procedure (as it is done in \cite{BGJO}), we can extend the result  to functions in $ C^0 ([0,1])$.

\begin{prop}\label{Tightness}
	The sequence of measures $\lbrace\mathbb{Q}^{N}\rbrace_{N\geq 1}$ is tight with respect to the Skorohod topology of $\mathcal D([0, T],{\mathcal{M^{+}}})$.
	\begin{proof}
		The proof is almost analogous to the one  presented  in \cite{BJGO2} so we will omit most of the details.
		We are going to prove (\ref{T1}) for functions $G$ in  $C_{c}^{2}([0,1])$. Recall now  (\ref{Dynkin'sFormula}). Then, in order to prove (\ref{T1}), it is enough to show that
		\begin{eqnarray} 
		\label{TC1}
		\displaystyle \lim _{\delta\rightarrow 0} \limsup_{N\rightarrow\infty} \sup_{\tau  \in \mathcal{T}_{T},\bar\tau \leq \delta}\mathcal{\mathbb{E}}_{\mu _{N}}\left[ \Bigg| \int_{\tau}^{\tau+ \bar\tau}\Theta(N) L_{N}\langle \pi_{s}^{N},G\rangle ds \Bigg|\right] = 0
		\end{eqnarray}
		and
		\begin{equation} 
		\label{TC2}
		\displaystyle \lim _{\delta\rightarrow 0} \limsup_{N\rightarrow\infty} \sup_{\tau  \in \mathcal{T}_{T},\bar\tau \leq \delta}\mathcal{\mathbb{E}}_{\mu _{N}}\left[\left( M_{\tau}^{N}(G)- M_{\tau+ \bar\tau}^{N}(G) \right)^{2}  \right]=0.
		\end{equation}
		
We start with 	\eqref{TC1}. Observe that we know the explicit form of the term inside the integral in \eqref{TC1} which is given by \eqref{gen_action}. Then,
		using Lemma \eqref{L1} and the fact that $G\in C_{c}^{2}((0,1))$ we can bound the first term on the right-hand side of (\ref{gen_action}) by a constant, since the regional fractional Laplacian is well defined on functions with compact support. For the second term on the right-hand side of \eqref{gen_action} we can proceed as follows. Using the fact that $|\eta_s^N (x)| \le 1$  and a Taylor expansion on $G$ around $0$ we get that it can be bounded by a constant times
		\begin{equation}\label{Tx}
		\Theta(N)N^{-\theta-2}\sum_{x\in \Lambda_N} x(r_{N}^{-}(\tfrac{x}{N})+r_{N}^{+}(\tfrac{x}{N}))
		\end{equation}
		plus lower order terms. Since $r^-_N(x/N)$ and $r^+_N(x/N)$ are both of order $x^{-\gamma}$, the sum is diverging and can be bounded by something of order $N^{-\gamma+2}$. This means that the order of the whole previous display is $\Theta(N)N^{-\gamma-\theta}$. So, for every value of $\theta$ and $\Theta(N)$ this term is of order at most $1$.
	  Then we have just finished to show that 
		\begin{equation}
	|\Theta(N) L_N ( \langle \pi^N_{s}, G \rangle )| \lesssim 1
		\end{equation}
		for any $s\le T$, which trivially implies \eqref{TC1}. 
		
	Now we prove \eqref{TC2}.
	By Dynkin's formula (see Appendix 1  in \cite{KL}) we know that  $$\displaystyle\left( M^{N}_{t}(G)\right)^{2}-\int^{t}_{0} \Theta(N)\left[ L_{N} \langle\pi^{N}_{s},G \rangle^{2}- 2\langle\pi^{N}_{s},G \rangle L_{N} \langle\pi^{N}_{s},G \rangle\right]ds,$$
		is a martingale with respect to the natural filtration  $\{\mathcal{F}_{t}\}_{t\ge 0}$. 
	    Proceeding as in \cite{BJGO2} it is possible to prove that the previous display is bounded from above by a constant times
		\begin{equation}
		\label{T6}
		\begin{split}
		&\dfrac{\Theta(N)}{(N-1)^{4}} \sum_{x,y\in\Lambda_{N}} (x-y)^{2}p(x-y)
		+\dfrac{\kappa \Theta(N)}{(N-1)^{2}N^{\theta}}\sum_{ x\in\Lambda_{N}} \left(G\left(\tfrac{x}{N}\right)\right)^{2} \left( r_{N}^{-}(\tfrac{x}{N}) +r_{N}^{+}(\tfrac{x}{N})  \right).\\
		\end{split}
		\end{equation}

		Then, we can bound from above the leftmost term in the previous display as
			\begin{equation*}
			\dfrac{\Theta(N)}{(N-1)^{4}} \sum_{x,y\in\Lambda_{N}} (x-y)^{2}p(x-y)\lesssim \dfrac{\Theta(N)}{(N-1)^{4}} \sum_{x \neq y\in\Lambda_{N}} |x-y|^{1-\gamma}\lesssim \frac{\Theta(N)}{N}N^{-\gamma}\leq N^{-1},
			\end{equation*}
if $\gamma \in (0,1]$. In the case $\gamma \in (1,2)$, the same procedure of \cite{BJGO2} leads to an upper bound of order $N^{\gamma-2}$.	The remaining term in (\ref{T6}) is of order less than $\mathcal{O}(1)$, since performing a Taylor expansion on $G$ around $0$ we can rewrite it as a constant times
		\begin{equation*}
		\dfrac{ \kappa \Theta(N)}{(N-1)^{2}N^{2+\theta}}(G'(0))^2\sum_{ x = aN}^{N-1} x^{-\gamma+2}
		\end{equation*}
  plus lower order terms in $N$, for some $a \in (0,1)$. Then, since the sum is of order $N^{3-\gamma}$, this term is at most of order $\Theta(N)N^{-\gamma-\theta-1}$, which, for any  value of $\Theta(N)$ given by \eqref{timescale}, is bounded by $N^{-1}$ and  vanishes as $N$ goes to infinity.
		
		Thus, since $\tau$ is a bounded stopping time we have that  
		\begin{equation*}
		\begin{split} 
			&\displaystyle \lim _{\delta\rightarrow 0} \limsup_{N\rightarrow\infty} \sup_{\tau  \in \mathcal{T}_{T},\bar\tau \leq \delta}\mathcal{\mathbb{E}}_{\mu ^{N}}\left[\left( M_{\tau}^{N,G}- M_{\tau+ \bar\tau}^{N,G} \right)^{2}  \right]\\
			&=\displaystyle \lim _{\delta\rightarrow 0} \limsup_{N\rightarrow\infty} \sup_{\tau  \in \mathcal{T}_{T},\bar\tau \leq \delta}\mathcal{\mathbb{E}}_{\mu ^{N}}\left[ \int^{\tau+\bar\tau}_{\tau}\Theta(N)\left[ L_{N} \langle\pi^{N}_{s},G \rangle^{2}- 2\langle\pi^{N}_{s},G \rangle L_{N} \langle\pi^{N}_{s},G \rangle\right]ds \right]\\
			&=0.
			\end{split}
		\end{equation*}
		
		Therefore, we have proved (\ref{T1}) for functions $G$ in $C_{c}^{2}([0,1])$ and, as we have recalled in the beginning of the proof, this is enough to conclude tightness.
	\end{proof}
\end{prop}

\subsection{Energy estimates}
\label{subsec:EE}
We prove in this subsection that any limit point $\mathbb{Q}$ of the sequence $\lbrace\mathbb{Q}^{N}\rbrace_{N\geq 1}$ is concentrated on trajectories $\{\pi_{t}(u)du\}_{t\in [0,T]}$ with finite energy, i.e., $\pi$ belongs to the space $L^{2}(0,T;\mc H^{\gamma/2})$. The latter is the content of Theorem  \ref{Energy_Thm1}  stated below. Despite our approach being standard (see \cite{BJGO2}), we decided to write the arguments for the convenience of the reader.  Fix a limit point $\mathbb{Q}$ of the sequence $\lbrace\mathbb{Q}^{N}\rbrace_{N\geq 1}$ and  assume, without loss of generality, that the  whole sequence $\mathbb{Q}^{N}$ converges to $\mathbb{Q}$, as $N$ goes to infinity.
		\begin{thm}\label{Energy_Thm1}
	For  $\theta \in \mathbb{R}$  and  $\gamma \in (0,2)$,	the measure $\mathbb{Q}$  is concentrated on trajectories of measures of the form $\{\pi_{t}(u)du\}_{t\in [0,T]}$, such that for any interval $I \subset [0,T]$  the density $\pi$ satisfies $\mathbb{Q}$-a.s. 
	\begin{enumerate}[i)] 
	\item		$$\int _{I} \Vert \pi_{t} \Vert _{\gamma/2}^{2}dt \lesssim \vert I\vert, \; \text{if} \; \theta \geq 0.  $$ 
	\item  $$\int _{I} \int_0^{1} \Bigg\{  \frac{[\alpha -\pi_{t}(u)]^2 }{u^{\gamma}} +  \frac{[\beta -\pi_{t}(u)]^2 }{(1-u)^{\gamma}}   \Bigg\} du dt \lesssim \vert I\vert, \; \text{if} \; \theta \leq 0.  $$ 
	\end{enumerate}
\end{thm}

Before we prove Theorem \ref{Energy_Thm1}, we establish some estimates on the Dirichlet form which are needed for the proof. Let
$\rho : [0, 1] \rightarrow [\alpha, \beta]$. Let $\nu_{\rho(\cdot)}^{N}$ be the inhomogeneous Bernoulli product measure on 
$\Omega_{N}$ with marginals given by
\begin{equation}
\label{product}
\nu_{\rho(\cdot)}^{N}\lbrace\eta: \eta(x) = 1 \rbrace = \rho\left( \tfrac{x}{N}\right).
\end{equation}
We denote  by $H_{N}(\mu\vert \nu_{\rho(\cdot)}^{N})$ the relative entropy of a probability measure $\mu$ on 
$\Omega_{N}$ with respect to the probability measure $\nu_{\rho(\cdot)}^{N} $. For exclusion processes on a finite domain and for any initial measure $\mu_N$,  it is easy to prove the existence of a constant $C_0$, such that
\begin{equation}\label{H}
H(\mu_{N}\vert \nu_{\rho(\cdot)}^{N})\leq C_0 N,
\end{equation}
see for example \cite{BGJO} for the proof. We remark here that the restriction $\alpha,\beta\notin {\{0,1\}}$ comes from last estimate since the constant $C_0$  above is given by $C_0=-\log(\alpha \wedge(1-\beta))$, when $\alpha<\beta$. On the other hand, for a probability measure $\mu$ on $\Omega_N$ and a density function $f:\Omega_N \to [0,\infty)$ with respect to $\mu$ we introduce  
\begin{equation} 
\label{left_rig_form}
D_{N}^{0}(\sqrt{f},\mu):=\tfrac{1}{2}\sum_{x,y\in\Lambda_N}p(y-x)\, I_{x,y}(\sqrt{f},\mu), 
\end{equation}
\begin{equation} 
\label{left_dir_form}
D_{N}^{l}(\sqrt{f},\mu):=\sum_{x\in\Lambda_N}\sum_{y\leq 0}p(y-x)\, I^\alpha_{x}(\sqrt{f},\mu)=\sum_{x\in\Lambda_N}r_N^-(\tfrac{x}{N})I^\alpha_{x}\, (\sqrt{f},\mu),
\end{equation}
\begin{equation} 
\label{right_dir_form}
D_{N}^{r}(\sqrt{f},\mu):=\sum_{x\in\Lambda_N}\sum_{y\geq N }p(y-x)\, I^\beta_{x}(\sqrt{f},\mu)=\sum_{x\in\Lambda_N}r_N^+(\tfrac{x}{N})I^\beta_{x}\, (\sqrt{f},\mu).
\end{equation}
 Above, we used the following notation
\begin{equation}
\label{Ixy}
\begin{split}
	I_{x,y}(\sqrt f,\mu):=& \int \left(\sqrt {f(\sigma^{x,y}\eta)}-\sqrt {f(\eta)}\right)^{2} d\mu,\\
	I_{x}^{\delta}(\sqrt f,\mu):=& \int  c_{x}(\eta;\delta)\left(\sqrt {f(\sigma^{x}\eta)}-\sqrt {f(\eta)}\right)^{2} d\mu, \quad \delta\in\{\alpha, \beta\}.
	\end{split}
\end{equation}
Our goal is to express, for the measure $\nu_{\rho(\cdot)}^{N}$,  a relation between the Dirichlet form defined by $\langle L_N\sqrt{f},\sqrt{f} \rangle_{\nu_{\rho(\cdot)}^{N}}$ and the quantity
\begin{equation}\label{D_false}
	D_{N}(\sqrt{f},\nu_{\rho(\cdot)}^{N} ):= (D_{N}^{0}+\kappa N^{-\theta} D_{N}^{l}+\kappa N^{-\theta}D_{N}^{r})(\sqrt{f},\nu_{\rho(\cdot)}^{N}). \end{equation}

That relation  is given by the following Lemma, already proved in Section 3.3.1 of \cite{BJGO2}. 
\begin{lem}\label{bound_Dir}
	There exists a constant $C>0$ such that for any positive constant $B$ and any density function $f$ with respect to $\nu_{\rho(\cdot)}^N$, we have that 
	\begin{equation}
	\label{dir_est}
	\begin{split}
	& \frac{\Theta(N)}{NB}\langle L_{N}\sqrt{f},\sqrt{f} \rangle_{\nu_{\rho(\cdot)}^N}\\
	 &\leq -\dfrac{\Theta(N)}{4NB}D_{N}(\sqrt{f},\nu_{\rho(\cdot)}^N) + \dfrac{C\Theta(N)}{N	B}\sum_{x,y\in\Lambda_N}p(y-x)\Big(\rho(\tfrac xN)-\rho(\tfrac yN)\Big)^2\\
	&+ \dfrac{C\kappa \Theta(N)}{N^{\theta+1}B} \sum_{x\in\Lambda_N}\left\lbrace \Big(\rho(\tfrac xN)-\alpha\Big)^2r^{-}_{N}(\tfrac{x}{N}) + \Big(\rho(\tfrac xN)-\beta\Big)^2r^{+}_{N}(\tfrac{x}{N}) \right\rbrace.
	\end{split}
	\end{equation}
\end{lem}

\begin{rem}
 Note that, as a consequence of the previous lemma, 
		for a Lipschitz function $\rho(\cdot)$ such that $\alpha \leq \rho(u)\leq \beta$ and $\rho(0)=\alpha$ and $\rho(1)=\beta$,  we have that
		\begin{equation}
		\label{lipschitz}
		\begin{split}
		\frac{\Theta(N)}{NB}\langle L_{N}\sqrt{f},\sqrt{f} \rangle_{\nu_{\rho(\cdot)}^N} &\leq -\dfrac{\Theta(N)}{4NB}D_{N}(\sqrt{f},\nu_{\rho(\cdot)}^N) +\frac{C\Theta(N)( \kappa +N^{\theta})}{BN^{\gamma+\theta}}.
		\end{split}
		\end{equation}
\end{rem}
	\begin{rem}
		Recasting the bound found in Lemma \ref{bound_Dir} when $\rho(\cdot)$ is equal to a constant in $(0,1)$ we get that: 
		
		\begin{itemize}
		\item [a)] if $\gamma\in(1,2)$, then
		\begin{equation}
		\label{dir_est_const}
		\begin{split}
		\frac{\Theta(N)}{NB}\langle L_{N}\sqrt{f},\sqrt{f} \rangle_{\nu_{\rho(\cdot)}^N} &\leq -\dfrac{\Theta(N)}{4NB}D_{N}(\sqrt{f},\nu_{\rho(\cdot)}^N) +\frac{C\kappa \Theta(N) }{BN^{\theta+1}}
		\end{split}
		\end{equation}
\item 	[b)] if  $\gamma=1$, then
		\begin{equation}
		\label{dir_est_const_gamma_1}
		\begin{split}
		\frac{\Theta(N)}{NB}\langle L_{N}\sqrt{f},\sqrt{f} \rangle_{\nu_{\rho(\cdot)}^N} &\leq -\dfrac{\Theta(N)}{4NB}D_{N}(\sqrt{f},\nu_{\rho(\cdot)}^N) +\frac{C \kappa \Theta(N)\log(N) }{BN^{\theta+1}},
		\end{split}
		\end{equation}
\item [c)] if $\gamma\in(0,1)$, then
		\begin{equation}
		\label{dir_est_const_gamma_less_1}
		\begin{split}
		\frac{\Theta(N)}{NB}\langle L_{N}\sqrt{f},\sqrt{f} \rangle_{\nu_{\rho(\cdot)}^N} &\leq -\dfrac{\Theta(N)}{4NB}D_{N}(\sqrt{f},\nu_{\rho(\cdot)}^N) +\frac{C\kappa \Theta(N)}{B N^{\gamma+\theta}}.
		\end{split}
		\end{equation}
		\end{itemize}
		The difference between the previous bounds is a consequence of the behavior of the sum $\sum_{x\in\Lambda_N}r_N^\pm(\tfrac xN)$ according to the value of  $\gamma$.
\end{rem}

\begin{rem}\label{rem_rev}
   Now we can establish the properties that we need to impose on the profile $\rho(\cdot)$ in order to have a good approximation of $\langle L_N\sqrt f, \sqrt f \rangle_{\nu^N_{\rho(\cdot)}}$ via $D_N(\sqrt f, \nu^N_{\rho(\cdot)})$. This is equivalent to find the condition on $\rho(\cdot)$ such that the last two terms on the right-hand side of \eqref{dir_est} vanish when first $N\to+\infty$ and then $B\to+\infty$. Thanks to the previous remarks  we make the choice:
   \begin{itemize}
       \item $\rho(\cdot)=\rho \in (0,1)$ constant if $\gamma \in (0,1)$ and $\theta \geq 0$   or if $\gamma=1$ and $\theta>0$ or if $\gamma \in (1,2)$ and $\theta\geq \gamma-1$;
       \item $\rho(\cdot)$ Lipschitz such that $\rho(0)=\alpha$ and $\rho(1)=\beta$ in all the other regimes.
   \end{itemize}
\end{rem}

	\subsubsection{Proof of Theorem \ref{Energy_Thm1}}

The proof of the second item is exactly the proof of item ii) of Theorem 3.2 in \cite{BJGO2} and will be omitted. The proof of the first item is almost analogous to the proof of item i) of Theorem 3.2 in \cite{BJGO2} (where the authors proved it for the regime $\gamma \in (1,2)$ and $\theta \leq 0$). We write here the main steps.
	We want to show that $\pi_{\cdot} \in L^{2}(0,T;\mc H ^{\gamma/2})$ $\mathbb{Q}$-almost surely. In order to do that we will show the next bound
	\begin{equation}
	\label{stat_en}
	    \mathbb{E}_{\mathbb{Q}}\Bigg[\int_I \|\pi_t\|^2_{\gamma/2}dt\Bigg]\lesssim |I|,
	\end{equation}
	which implies the result in the statement above.
	
Recall that   the system is speeded up on the time scale $\Theta(N)$, defined by \eqref{timescale}. Let $\rho(\cdot)$ be a macroscopic profile chosen according to Remark \ref{rem_rev}. Let $\epsilon>0$ be a small real number. Let $F \in C_c^{0,\infty} (I\times[0,1]^{2})$, where  $I$ is a  subinterval of $[0,T]$. By the  entropy and Jensen's inequality and Feynman-Kac's formula (see Lemma A.7.2 in \cite{KL}),  we have that  
	\begin{equation}
	\label{eq:varfor1}
	\begin{split}
	&{\mathbb E}_{\mu_N} \Bigg[\int_I  \;  \Theta(N)N^{ -1} \sum_{\substack{x,y\in \Lambda_N\\ |x-y|\geq \epsilon N} } F_{t}( \tfrac{x}{N}, \tfrac{y}{N})p(y-x) (\eta^N_{t}(y)-\eta^N_{t}(x))dt\Bigg] \\
	&\leq C_{0} + \int_{I}\sup_{f} \Bigg\{ \Theta(N)N^{ -1}  \sum_{\substack{x,y\in \Lambda_N\\ |x-y|\geq \varepsilon N} } F_{t}( \tfrac{x}{N}, \tfrac{y}{N})p(y-x) \int(\eta(y)-\eta(x))f(\eta)d\nu^{N}_{\rho(\cdot)}\\&  \quad \quad \quad \quad  \quad \quad \quad \quad \quad \quad \quad \quad +  \Theta(N)N^{ -1} \left\langle  L_N   {\sqrt f} , {\sqrt f} \right\rangle_{\nu^{N}_{\rho(\cdot)}}  \Bigg\}dt
	\end{split}
	\end{equation}
	where the supremum is taken over all  densities $f$ on $\Omega_N$ with respect to $\nu_{\rho(\cdot)}^{N}$. Now, denoting the antisymmetric part of $F_t$ by $F_t^a$ with  $$F_{t}^a (u,v) =\cfrac{1}{2} \Big[ F_{t}(u,v) -F_{t}(v,u) \Big]$$ for all $t\in I$ and $(u,v) \in [0,1]^2$, we can rewrite the first term inside the supremum in \eqref{eq:varfor1} as
	\begin{equation}
	\label{eq_ste1}
	\begin{split}
	\Theta(N)&N^{ -1} \sum_{\substack{x,y \in \Lambda_N\\ |x-y| \ge \epsilon N}}  F_{t}^a( \tfrac{x}{N}, \tfrac{y}{N} ) p(y-x) \int (\eta(y)-\eta(x)) f(\eta) d\nu^{N}_{\rho(\cdot)}\\
	= \, &\Theta(N)N^{ -1}\sum_{\substack{x,y \in \Lambda_N\\ |x-y| \ge \epsilon N}}  F_{t}^a( \tfrac{x}{N},\tfrac{y}{N} )  p(y-x) \int \eta(y) \left( f(\eta) - f(\sigma^{x,y}\eta)\right) d\nu^{N}_{\rho(\cdot)}\\
	+ \, &\Theta(N)N^{ -1}\sum_{\substack{x,y \in \Lambda_N\\ |x-y| \ge \epsilon N}} F_{t}^a( \tfrac{x}{N}, \tfrac{y}{N} ) p(y-x) \int \eta(x) f(\eta)\left(\theta^{x,y}(\eta) -1\right)d\nu^{N}_{\rho(\cdot)},
	\end{split}
	\end{equation}
where the equality follows from  the change of variables $\eta \mapsto \sigma^{x,y}\eta$ and $$\theta^{x,y}(\eta)=\frac{d\nu_{\rho(\cdot)}^{N}(\sigma^{x,y}\eta)}{d\nu_{\rho(\cdot)}^{N}(\eta)}.$$

Now, we have to split the proof in two cases: the case in which $\rho(\cdot)=\rho$ is constant (corresponding to the regimes $(\theta, \gamma) \in [0,\infty)\times (0,1)$, $(\theta, \gamma) \in (0,\infty)\times \{1\}$ and $(\theta, \gamma) \in [\gamma-1,\infty)\times (1,2)$) and the case in which $\rho(\cdot)$ is a Lipschitz profile (corresponding to the regime $(\theta, \gamma) \in (0,\gamma-1)\times (1,2)$).

In the case in which $\rho(\cdot)=\rho$ is constant we have that $\theta^{x,y}(\eta) -1=0$. Therefore, by Young's inequality, the fact that $f$ is a density and $|\eta(y)| \le 1$, we have that, for any $A>0$, \eqref{eq_ste1} is bounded from above by a constant times
	\begin{equation*}
	{c_\gamma A \Theta(N)N^{ -2-\gamma}}   \sum_{\substack{x,y \in \Lambda_N\\ |x-y| \ge \epsilon N}}  \cfrac{\Big( F_{t}^a \Big( \tfrac{x}{N}, \tfrac{y}{N} \Big)\Big)^2}{ | \tfrac{x}{N} -\tfrac{y}{N}|^{1+\gamma}} \; + \; \dfrac{2\Theta(N)N^{ -1}}{A}D^{0}_{N}(\sqrt {f}, \nu^N_{\rho(\cdot)}). 
	\end{equation*}

Then, using \eqref{dir_est_const}, \eqref{dir_est_const_gamma_1} and \eqref{dir_est_const_gamma_less_1}, we can bound
\eqref{eq:varfor1} in the following way:
	\begin{equation}\label{eq:imp}
	\begin{split}
	&{\mathbb E}_{\mu_N} \bigg[ \int_I   \Theta(N)N^{ -1}  \sum_{\substack{x,y\in \Lambda_N\\ |x-y|\geq \varepsilon N} } F^a_{t}( \tfrac{x}{N}, \tfrac{y}{N})p(y-x) (\eta_t^{N}(y)-\eta_t^{N}(x))\,dt\bigg]\\
	&={\mathbb E}_{\mu_N} \left[ \int_I -2c_{\gamma}\langle\pi^N_t,g^N_t\rangle\,dt\right] \lesssim \int _I N^{ -2} \sum_{\substack{x,y \in \Lambda_N\\ |x-y| \ge \epsilon N}}  \cfrac{c_\gamma \left( F_{t}^a ( \tfrac{x}{N}, \tfrac{y}{N} ) \right)^2}{| \tfrac{x}{N} -\tfrac{y}{N}|^{1+\gamma}}  dt+ |I|( \kappa  +1),
	\end{split}
	\end{equation}
	where the function $g^{N}$ is defined on $ I\times[0,1]$ by
	$$g_t^{N} (u) = \cfrac{1}{N-1} \sum_{y \in \Lambda_N} \, \textbf{1}_{ \big|\tfrac{y}{N}- u \big| \ge \epsilon }\cfrac{F_{t}^a \big(u, \tfrac{y}{N}\big)}{\vert u -\tfrac{y}{N}\vert^{1+\gamma}}$$
	and it is a discretization of the smooth function  $g$ defined on $(t,u) \in I\times[0,1]$  by
	$$ \quad g_{t}(u) = \int_0^1\textbf{1}_{\{ |v-u| \ge \epsilon\}} \cfrac{F_{t}^a (u,v)}{|u-v|^{1+\gamma}} \, dv.$$
	Now, we can proceed exactly in the same way as in the proof of Theorem 3.2 of \cite{BJGO2} to conclude that
	\begin{equation}
	\label{eq:A4}
	\begin{split}
	&{\bb E}_{\mathbb{Q}} \left[ \sup_F \left\{ \int _{I}  \iint_{Q_{\epsilon} }\cfrac{(\pi_{t} (v) -\pi_{t} (u)) F _{t}(u,v) }{|u-v|^{1+\gamma}}\;  - C \cfrac{ \big( F _{t}(u,v) \big)^2 }{|u-v|^{1+\gamma}} \; du dv dt    \right\} \right] \\
	& \lesssim|I|( \kappa+1),
	\end{split}
	\end{equation}
	where $Q_{\epsilon}=\{(u,v)\in [0,1]^2\; ; \; |u-v| \ge \epsilon\}$. That bound, passing to the limit $ \epsilon\to 0$, by the monotone convergence Theorem, implies the statement \eqref{stat_en}.

The case in which $\rho(\cdot):[0,1]\to [\alpha,\beta]$ is Lipschitz continuous and such that $\rho(0)=\alpha$ and $\rho(1)=\beta$ can be analyzed exactly in the same way as it was proved in Theorem 3.2 of \cite{BJGO2} therefore we will omit this part of the proof.

\subsection{Characterization of the limit points}
In this subsection we characterize all limit points $\mathbb{Q}$ of the sequence $\lbrace\mathbb{Q}^{N} \rbrace _{N\geq 1}$, {which we know that exist from Proposition \ref{Tightness}. Despite our approach being standard (see \cite{BGJO, BJGO2} ), we decided to write the arguments for the convenience of the reader. Let us assume without loss of generality, that  $\lbrace\mathbb{Q}^{N} \rbrace _{N\geq 1}$ converges to $\mathbb{Q}$, as $N$ goes to infinity. Since there is at most one particle per site, it is easy to show that $\mathbb{Q}$ is concentrated on trajectories {of measures} absolutely continuous with respect to the Lebesgue measure, i.e. $\pi_{t}^{\kappa}(du)=\rho_{t}^{ \kappa}(u)du$ (for details see \cite{KL}). In Proposition \ref{prop:weak_sol_car} below we prove, for each range of $\theta$ and $\gamma$, that  $\mathbb{Q}$ is concentrated on trajectories of measures whose density {$\rho_{t}^{ k}$}
satisfies a weak form of the corresponding hydrodynamic equation. Moreover, we have seen in Theorem \ref{Energy_Thm1} that $\mathbb{Q}$ is concentrated on trajectories of measures whose density satisfies the energy estimate, i.e. $\rho^{ \kappa}\in L^{2}(0,T;\mc H^{\gamma/2})$ when $\theta\geq 0$ and $\gamma \in (0,2)$. 
%

\begin{prop}
	\label{prop:weak_sol_car}
	If $\bb Q$ is a limit point of $ \{\bb Q^{N}\}_{N\geq 1}$  then 
	\begin{eqnarray} \nonumber
	&&  \bb Q\left(\pi _{\cdot}\in \mc D([0,T],{\mc M^+}): F_{\theta,\gamma}(t,\rho^{ \kappa},G,g)=0, \forall t \in [0,T], \for] \forall G \in C_{\theta,\gamma} \right)=1, 
		\end{eqnarray}	
	for each $\theta$ and each $\gamma$, where $F_{\theta,\gamma}$ is defined by
	\begin{equation}
	\begin{split}
	&F_{\theta,\gamma}(t,\rho^{\kappa},G,g)\\
	&\quad :=\begin{cases} F_{Reac}(t, \rho^{ \kappa},G,g) &\text { if } \theta<0,   \gamma \in (0,1] ;\\
	F_{RD}(t, \rho^{\kappa},G,g) & \text{ if } \theta=0, \gamma \in (0,1];\\
	F_{Neu}(t, \rho^{0},G,g) & \text{ if } \theta>0, \gamma \in (0,1] \text { or } \theta>\gamma-1, \gamma \in (1,2);\\
	F_{Rob}(t, \rho,G,g) & \text{ if } \theta=\gamma-1, \gamma \in (1,2);\\
	F_{Dir}(t, \rho,G,g)  & \text{ if } \theta \in (0,\gamma-1), \gamma \in (1,2).
\end{cases}
\end{split}
	\end{equation}
where $C_{\theta,\gamma}=C^{1,2} ([0,T]\times[0,1])$ if $\theta >0$ and $\gamma \in (0,1]$ or if $\theta \geq \gamma-1$ and $\gamma \in (1,2)$, while in all the other regimes it is $C_{\theta,\gamma}=C_c^{1,2} ([0,T]\times[0,1])$.
\end{prop}
\begin{proof}
	The proof is very similar to the one in \cite{BJGO2} and for that reason some details are omitted.
	Note that in order to prove the proposition, it is enough to verify, for $\delta > 0$ and $ G$ in the corresponding space of test functions,  that  
	\begin{equation}
	\label{Qprob}
	\bb Q\left(\pi _{\cdot}\in \mc D([0,T],{\mc M^+}): \sup_{0\le t \le T} \left\vert F_{\theta,\gamma}(t,\rho^{ \kappa},G,g) \right\vert>\delta\right)=0,
	\end{equation}
	for each $\theta$ and each $\gamma$ in the respective domains.
	
	We have that
	\begin{equation}
	\label{def_F_theta}
	\begin{split}
	&F_{\theta,\gamma}(t, \rho^{ \kappa},G,g)\\
	&=\left\langle \rho^{ \kappa}_{t},  G_{t} \right\rangle -\left\langle g,   G_{0}\right\rangle - \int_0^t\left\langle \rho^{ \kappa}_{s},\Big(\partial_s + \mathbb{1}_{\{\theta\geq 0 \}}\bb L \Big) G_{s}  \right\rangle ds \\
	&+ \mathbb{1}_{\{\theta\leq 0\}}\mathbb{1}_{\{\gamma \in (0,1] \}}  \kappa  \int^{t}_{0} \left\langle \rho_{s}^{ \kappa}, G_s \right\rangle_{ V_1 } ds 
	- \int^{t}_{0}\left\langle G_s , V_0\right\rangle ds  \\
	&- \mathbb{1}_{\{ \theta=\gamma-1 \}}\mathbb{1}_{\{\gamma \in (1,2)\}}  \kappa \int_0^t mG_s(0)(\alpha-\rho^{\kappa}_s(0))+mG_s(1)(\beta-\rho^{\kappa}_s(1))ds.
	\end{split}   
	\end{equation}
	From here on, in order to simplify notation, we will erase $\pi_\cdot$ from the sets that we have to look at. 
	By \eqref{def_F_theta}  we can bound from above \eqref{Qprob} by the sum of
	\begin{equation}\label{RD6}
	\bb Q\left(  \sup_{0\le t \le T} \left|F_{\theta,\gamma}(t,\rho^{ \kappa},G,\rho_{0}) \right|>\dfrac{\delta}{2}\right)
	\end{equation}
	and
	\begin{equation}\label{1111}
	\bb Q \left(  \left| \left\langle \rho_{0}-g, G_{0}\right\rangle \right|>\dfrac{\delta}{2}\right).
	\end{equation}
	We note that \eqref{1111} is equal to zero since $\mathbb Q$ is a limit point of $\{\mathbb Q^N\}_{N\geq 1}$ and $\mathbb Q^N$ is induced by $\mu_N$ which is associated to $g(\cdot)$. {Now we deal with} \eqref{RD6}.
	 Note that by Proposition A.3 of \cite{FGN}, the set inside the probability in \eqref{RD6} is an open set in the Skorohod topology. Therefore, from Portmanteau's Theorem we bound \eqref{RD6} from above by
	\begin{equation*}
	\liminf_{N\to\infty}\,\bb Q^{N}\left( \sup_{0\le t \le T} \left|F_{\theta,\gamma}(t,\rho^{ \kappa},G,\rho_{0}) \right|>\dfrac{\delta}{2}\right).
	\end{equation*}
	Summing and subtracting $\displaystyle\int_{0}^{t} \Theta(N) L_{N}\langle \pi_{s}^{N},G_{s}\rangle ds$ to the term inside the previous absolute {value}, recalling \eqref{Dynkin'sFormula}, \eqref{def_F_theta} and  the definition of $\mathbb Q^N$,  we can bound the previous probability   from above by the sum of  
	\begin{equation*}
	\bb P_{\mu_{N}} \left(\sup_{0\le t \le T} \left\vert M_{t}^{N}(G) \right\vert>\dfrac{\delta}{4}\right)
	\end{equation*}
	and
	\begin{equation}
	\label{CLP2}
	\begin{split}
	&\bb P_{\mu_{N}}  \bigg( \sup_{0\le t \le T} \Big| \int_{0}^{t} \Theta(N) L_{N}\langle \pi_{s}^{N},G_{s}\rangle ds  - \int_0^t\left\langle \rho^{ \kappa}_{s}, \mathbb{1}_{\{\theta\geq 0 \}}\bb L  G_{s}  \right\rangle ds \\
	&+\mathbb{1}_{\{\theta\leq 0\}}\mathbb{1}_{\{\gamma \in (0,1] \}}  \kappa \int^{t}_{0} \left\langle \rho_{s}^{ \kappa}, G_s \right\rangle_{ V_1 } ds 
	-  \mathbb{1}_{\{\theta\leq 0\}}\mathbb{1}_{\{\gamma \in (0,1] \}}  \kappa \int^{t}_{0}\left\langle G_s , V_0\right\rangle ds  \\- &\mathbb{1}_{\{ \theta=\gamma-1 \}}\mathbb{1}_{\{\gamma \in (1,2)\}}  \kappa \int_0^t mG_s(0)(\alpha-\rho^{\kappa}_s(0))+mG_s(1)(\beta-\rho^{\kappa}_s(1))ds \Big|>\dfrac{\delta}{4}\bigg). 
	\end{split}
	\end{equation}
By Doob's inequality we have that
	\begin{equation*}
	\begin{split}
	&\bb P_{\mu_{N}} \left(\sup_{0\le t \le T} \left\vert M_{t}^{N}(G) \right\vert>\dfrac{\delta}{4}\right) \\
	&\lesssim   \dfrac{1}{\delta^{2}} \bb E_{\mu _{N}} \left[\int_{0}^{T}\Theta(N)\left[ L_{N} \langle\pi^{N}_{s},G \rangle^{2}- 2\langle\pi^{N}_{s},G \rangle L_{N} \langle\pi^{N}_{s},G \rangle\right]ds \right].
	\end{split}
	\end{equation*}
	In the proof of Proposition \ref{Tightness} we have proved that the term inside the time integral in the previous expression vanishes, as $N$ goes to infinity. It remains to prove that (\ref{CLP2}) vanishes as well when we send $N$ to infinity. Recalling (\ref{gen_action}), we can bound (\ref{CLP2}) from above by the sum of the following terms
	\begin{equation}
	\label{DC1}
	\bb P_{\mu_{N}}  \left(\sup_{0\le t \le T} \left| \int_{0}^{t}\cfrac{\Theta(N)}{N-1} \sum_{x\in \Lambda_N}\mathcal{L}_NG_{s}(\tfrac{x}{N})\eta_s^{N}(x) ds-\int_0^t\left\langle \rho^{\kappa}_{s}, \mathbb{1}_{\{\theta\geq 0 \}}\bb L  G_{s}  \right\rangle ds \right|>\dfrac{\delta}{2^{4}}\right),
	\end{equation}
	\begin{multline}
	\label{DC2}
	\bb P_{\mu_{N}}\Bigg(\sup_{0\le t \le T} \bigg|   \int_{0}^{t} \bigg\{ \dfrac{  \kappa \Theta(N)}{N^{\theta}(N-1)} \sum_{x \in \Lambda_N}  (G_{s} r_{N}^{-})(\tfrac{x}{N}){(\alpha-\eta_s^{N}(x))} \\
-\mathbb{1}_{\{\theta\leq 0\}}\mathbb{1}_{\{\gamma \in (0,1] \}}  \kappa \int_{0}^{1}  (G_{s}r^{-})(u)(\alpha - \rho_{s}^{ \kappa}(u))du \\
-\mathbb{1}_{\{ \theta=\gamma-1 \}}\mathbb{1}_{\{\gamma \in (1,2)\}}  \kappa mG_s(0)(\alpha-\rho^{\kappa}_s(0)) \bigg\} ds \bigg| > \dfrac{\delta}{2^{4}}\Bigg).
	\end{multline}
	and
	\begin{multline}
	\label{DC3}
	\bb P_{\mu_{N}}\Bigg(\sup_{0\le t \le T} \bigg|   \int_{0}^{t} \bigg\{ \dfrac{  \kappa \Theta(N)}{N^{\theta}(N-1)} \sum_{x \in \Lambda_N}  (G_{s} r_{N}^{+})(\tfrac{x}{N}){(\beta-\eta_s^{N}(x))} \\
-\mathbb{1}_{\{\theta\leq 0\}}\mathbb{1}_{\{\gamma \in (0,1] \}}  \kappa \int_{0}^{1}  (G_{s}r^{+})(u)(\beta - \rho_{s}^{ \kappa}(u))du \\
-\mathbb{1}_{\{ \theta=\gamma-1 \}}\mathbb{1}_{\{\gamma \in (1,2)\}}  \kappa mG_s(1)(\beta-\rho^{\kappa}_s(1)) \bigg\} ds \bigg| > \dfrac{\delta}{2^{4}}\Bigg).
	\end{multline}
For $\theta\geq 0$, since $\Theta(N)=N^{\gamma}$ and Lemma \ref{L1} holds for any $G$ at least $C^2$ in the space variable,  applying Markov's inequality we have that (\ref{DC1}) goes to $0$, as $N$ goes to infinity. For $\theta < 0$ and$ \gamma \in (0,1]$
	 we also have that \eqref{DC1} vanishes as $N$ goes to infinity thanks to Markov's inequality and the same analysis that we did in Subsection \ref{reactionheur}.
	The boundary terms \eqref{DC2} and \eqref{DC3} can be treated in a similar way, therefore we explain in details why they vanish just referring to \eqref{DC2}. In the case $ \gamma \in (0,1]$ and $\theta \leq 0$, it is enough to use Markov's inequality and what we did in Subsection \ref{reactionheur} to show that \eqref{DC2} vanishes, as $N$ goes to infinity. In the case $\gamma \in (1,2)$ and $\theta =\gamma-1$ we again use Markov's inequality and then we repeat the computations we did in Subsection \ref{robinheur} to see that also in this case \eqref{DC2} vanishes, as $N$ goes to infinity. In all the other regimes that we are considering the only term we have to estimate using Markov's inequality is
	\begin{equation}
	\mathbb{E}_{\mu_N}\Bigg[ \bigg|   \int_{0}^{t} \bigg\{ \dfrac{ \kappa N^{\gamma}}{N^{\theta}(N-1)} \sum_{x \in \Lambda_N}  (G_{s} r_{N}^{-})(\tfrac{x}{N}){(\beta-\eta_s^{N}(x))}ds\bigg| \Bigg]
	\end{equation}
	 and it vanishes, as $N$ goes to infinity, thanks to the analysis that we did in Section \ref{heursitcs} relative to these regimes.
	 This finishes the proof of Proposition \ref{prop:weak_sol_car}. 
\end{proof}

\subsection{Conclusion}  Note that Proposition \ref{Tightness}, Theorem  \ref{Energy_Thm1} and Proposition \ref{prop:weak_sol_car} just proved, imply that the sequence $\{\mathbb{Q}^N\}_{N\geq1}$ converges up to subsequences to a measure $\mathbb{Q}$ concentrated on an absolutely continuous measure with respect to the Lebesgue measure of density $\rho^{\kappa}$, where $\rho^{\kappa}$ is the weak solution of the hydrodynamic equation chosen accordingly to the values of the parameters of the model (as stated in Theorem \ref{theo:hydro_limit}). Moreover, thanks to Proposition \ref{uniqueness} which grants the uniqueness of the weak solutions, we can conclude that this limiting measure is unique and this proves completely Theorem \ref{theo:hydro_limit}.

\section{Uniqueness of weak solutions}
\label{unique_sec}

In this section we prove uniqueness of the weak solutions we defined in Subsection \ref{subsec:hyd_eq}. Since the PDEs are linear, the strategy is standard. We consider two weak solutions starting from the same initial condition, take their difference and write down the  integral formulation for the difference of the two solutions. Then, chose a proper test function, plug it into the integral formulation and from there obtain that the difference is $0$ in some topology, which implies the uniqueness result we need.
 
\subsection{Uniqueness of  weak solutions of \eqref{eq:Neumann Equation}}
Fix $T>0$ and $\gamma \in (0,2)$. Consider $\rho^1$ and $\rho^2$ two weak solutions of  \eqref{eq:Neumann Equation}  starting from the same initial condition and denote by $\rho$ their difference: $\rho=\rho^1-\rho^2 \in L^2(0,T;\mathcal{H}^{\gamma/2})$.  The set $C^{1,\infty}([0,T]\times (0,1))$ is dense in $L^2(0,T;\mathcal{H}^{\gamma/2})$. Therefore we can consider a sequence $\{G_n\}_{n \in \NN}\subset C^{1,\infty}([0,T]\times (0,1))$  converging to $\rho$ with respect to the norm of $L^2(0,T;\mathcal{H}^{\gamma/2})$. Now, we  consider a special set of test functions, defined as $H_n(t,u):=\int_t^TG_n(s,u)ds$, for any $t \in [0,T]$ and any $u \in [0,1]$. It is easy to check that $H_n$ are suitable test functions. Now we state the Lemma 6.1 of \cite{BJGO2} adapting its proof to our case.

	\begin{lem}
	\label{6.1}
		Let $\{H_n\}_{n \in \NN}$ defined as above. Then we have
		\begin{enumerate}
			\item $\lim_{n\rightarrow \infty}\int_0^T \langle \rho_s,\partial_s H_n(s,\cdot)\rangle ds = - \int_0^T\|\rho_s\|^2ds$;
			\item $\lim_{n\rightarrow \infty}\int_0^T \langle \rho_s,\mathbb L H_n(s,\cdot)\rangle ds = -\cfrac{1}{2} \Big| \Big |\int_0^T\rho_sds\Big|\Big|_{\gamma/2}^2$;
		\end{enumerate}
	\end{lem}
	\begin{proof}
		    The proof of item (1) is completely analogous to the one which can be found in \cite{BJGO2}. Let us now prove item (2). Since $H_n(s,\cdot)$ is at least $C^2([0,1])$ and $\rho(s, \cdot) \in L^{\infty}([0,1]) \cap \mathcal{H}^{\gamma/2}$ for almost every $s \in [0,T]$,  we can apply Proposition \ref{prop:ip-guan} and get that 
	    \begin{equation}
	       \int_0^T \langle \rho_s,\mathbb{L} H_n(s,\cdot)\rangle ds = \int_0^T \langle \rho_s, H_n(s,\cdot)\rangle_{\gamma/2} ds.
	    \end{equation}
	 The rest of the proof of item (2) is then analogous to the one given in \cite{BJGO2}.
	\end{proof}

Thanks to this Lemma, we just have to plug in \eqref{eq:Dirichlet Neumann} the test function $H_n$, and then pass to the limit when $n$ goes to infinity, as done in \cite{BJGO2}, to get
\begin{equation*}
\int_0^T||\rho_s||^2ds+\cfrac{1}{2}\Big| \Big |\int_0^T\rho_sds\Big|\Big|_{\gamma/2}^2=0.
\end{equation*}
This implies that for almost every time $s \in [0,T]$ the function $\rho_s=0$ and this means that $\rho^1$ coincides with $\rho^2$ almost at every time. So, the weak solution of \eqref{eq:Neumann Equation} is unique.

\subsection{Uniqueness of weak solutions of  \eqref{eq:Dirichlet Equation2} for $\gamma\neq 1$.}

The uniqueness of the solution to \eqref{eq:Dirichlet Equation2} in the case $\gamma \in (1,2)$  has been proved in \cite{BGJO}. Hence we assume now $\gamma \in (0,1)$ and prove the uniqueness of the weak solution to \eqref{eq:Dirichlet Equation2}. The proof relies on the following Hardy's inequality proved in Corollary 2.4 in \cite{CS-Hardy}:
 \begin{equation}
\label{eq:hardy-01}
    \forall H\in C_c^\infty ((0,1)), \quad 
    \| H\|_{V_1} \lesssim \| H\|_{{\mc H}^{\gamma/2}}. 
\end{equation}
Since $C_c^\infty ((0,1))$ is dense in  $L^2_{V_1}([0,1])$ and in ${\mc H}_0^{\gamma/2}$,  and since for $\gamma < 1$ we have that  ${\mc H}_0^{\gamma/2} = {\mc H}^{\gamma/2}$, then  \eqref{eq:hardy-01} holds for any $H\in L^2_{V_1}([0,1]) \cap {\mc H}^{\gamma/2}$. Once this observed, the proof follows the same arguments as in \cite{BGJO} for the case $\gamma \in (1,2)$.

\begin{rem}
We observe  that the crucial difference between the case $\gamma \in (1,2)$ and the case  $\gamma \in (0,1)$ is that in the former case,  the weak solutions are assumed to satisfy the Dirichlet boundary conditions so that the difference of two weak solutions belongs to ${\mc H}_{0}^{\gamma/2}$ and Hardy's inequality can be used for the difference. For $\gamma \in (0,1) $, we do not know if weak solutions satisfy, in the strong sense, Dirichlet boundary conditions. Hence the difference of two weak solutions does not necessarily vanishes at the boundary (we do not even know if it has a continuous representative). Fortunately if  $\gamma \in (0,1) $, ${\mc H}_0^{\gamma/2}={\mc H}^{\gamma/2}$, so that it is not a real problem.    
\end{rem}

\begin{rem}
We observe that the uniqueness in the case $\gamma=1$ is still open. The crucial point in the proof presented  above is the fractional Hardy's inequality  \eqref{eq:hardy-01}, but we did not find in the literature a result indicating that  \eqref{eq:hardy-01} holds for $\gamma=1$. 
\end{rem}

\subsection{Uniqueness of weak solutions of \eqref{eq:Dirichlet0 Equation}}

We have $\gamma \in (1,2)$. Fix $T>0$. Consider $\rho^1$ and $\rho^2$ two weak solutions of  \eqref{eq:Dirichlet0 Equation}  starting from the same initial condition and denote by $\rho$ their difference: $\rho=\rho^1-\rho^2 \in L^2(0,T;\mathcal{H}^{\gamma/2})$. Since the continuous representatives of $\rho^1$ and $\rho^2$ satisfy Dirichlet boundary conditions, we have in fact that $\rho \in L^2(0,T;\mathcal{H}_0^{\gamma/2})$. The set $C_c^{1,\infty}([0,T]\times (0,1))$ is dense in $L^2(0,T;\mathcal{H}_0^{\gamma/2})$. Therefore we can consider a sequence $\{G_n\}_{n \in \NN}$ of of functions in $C_c^{1,\infty}([0,T]\times (0,1))$  converging to $\rho$ with respect to the norm of $L^2(0,T;\mathcal{H}_0^{\gamma/2})$. Now, we  consider a special set of test functions, defined as $H_n(t,u):=\int_t^TG_n(s,u)ds$, for any $t \in [0,T]$ and any $u \in [0,1]$. It is easy to check that $H_n$ are suitable test functions that can be used in \eqref{eq:Dirichlet Neumann}. Lemma \ref{6.1} holds and we get then 
\begin{equation*}
    \int_0^T||\rho_s||^2ds+\cfrac{1}{2}\, \left\Vert \int_0^T\rho_sds\right\Vert_{\gamma/2}^2=0
\end{equation*}
which implies that $\rho =0$ and shows uniqueness. 

\subsection{Uniqueness of weak solutions of \eqref{eq:Robin Equation}}

We have $\gamma \in (1,2)$. Fix $T>0$. Consider two weak solutions of \eqref{eq:Robin Equation} denoted by $\rho^{\hat\kappa,1}$ and $\rho^{\hat\kappa,2}$, starting from the same initial condition. Let the difference be $\rho^{\hat\kappa}=\rho^{\hat\kappa,1}-\rho^{\hat\kappa,2}$. We have that $\rho^{\hat\kappa} \in L^2(0,T;\mathcal{H}^{\gamma/2})$. Since $\gamma>1$ we may consider a continuous representative (for a.e. time) of $\rho^{\hat \kappa}$. The set $C^{1,\infty}([0,T]\times (0,1))$ is dense in $L^2(0,T;\mathcal{H}^{\gamma/2})$. Therefore, we can consider a sequence $\{G^{\hat\kappa}_n\}_{n \in \NN}\subset C^{1,\infty}([0,T]\times (0,1))$ converging to $\rho^{\hat \kappa}$ with respect to the norm of $L^2(0,T;\mathcal{H}^{\gamma/2})$, and we use as test functions $H^{\hat \kappa}_n(t,u):=\int_t^TG^{\hat \kappa}_n(s,u)ds$, for any $t \in [0,T]$ and any $u \in [0,1]$. Observe that Lemma \ref{6.1} holds for these test functions. Plugging the test functions $H^{\hat \kappa}_n$ in \eqref{eq:Robin Neumann} and passing to the limit as $n$ goes to infinity we get
\begin{equation}
\int_0^T||\rho^{\hat \kappa}_s||^2ds+\cfrac{1}{2}\Big| \Big |\int_0^T\rho^{\hat\kappa}_sds\Big|\Big|_{\gamma/2}^2+\cfrac{\hat \kappa m}{2}\Big| \int_0^T \rho^{\hat \kappa}_s(0)ds\Big|^2+\cfrac{\hat\kappa m}{2}\Big| \int_0^T \rho^{\hat \kappa}_s(1)ds\Big|^2=0,
\end{equation}
if we can prove that
\begin{equation}
\label{eq:Robin09}
    \lim_{n \to \infty} \int_0^T H_n^{\hat \kappa} (s,0) \rho^{\hat \kappa}_s (0) ds  =   \cfrac{1}{2}\Big| \int_0^T \rho^{\hat \kappa}_s(0)ds\Big|^2
\end{equation}
and similarly with the point $0$ replaced by the point $1$. This will imply the uniqueness of weak solutions for \eqref{eq:Robin Equation}. 

To prove \eqref{eq:Robin09} we observe, by using Fubini's Theorem, that
\begin{equation*}
    \int_0^T H_n^{\hat \kappa} (s,0) \rho^{\hat \kappa}_s (0)ds=  \iint_{0\le s\le r \le T} {\hat G}_n^{\hat \kappa} (r, 0) {\rho}^{\hat \kappa}_s (0) dr ds.
\end{equation*}
To conclude it is sufficient to prove that $G_n^{\hat \kappa} (\cdot, 0)$ converges to $\rho_{\cdot}^{\hat \kappa} (0)$ in the space $L^2 ([0,T], dt)$. Let us denote $f_n = G_n^{\hat \kappa} -\rho^{\hat \kappa}$. We know that $\{f_n\}_{n \in \mathbb N}$ converges to $0$ in the space $L^2 ([0,T]; {\mc H}^{\gamma/2})$, as $n$ goes to infinity. By Theorem 8.2 in \cite{Val} we have that for any $s\in [0,T]$ 
\begin{equation*}
    \sup_{u,v \in [0,1]^2} | f_n (s, u) -f_n(s,v)|\;  \lesssim \;  |u-v|^{\tfrac{\gamma-1}{2}} \, \| f_n (s, \cdot)\|_{{\mc H}^{\gamma/2}} \lesssim \| f_n (s, \cdot)\|_{{\mc H}^{\gamma/2}}.  
\end{equation*}
By the triangular inequality, averaging and Cauchy-Schwarz inequality, we get that for any $s\in[0,T]$ and any $0<R<1$
\begin{equation*}
\begin{split}
 |f_n(s,0)| &\lesssim \cfrac{1}{R}\int_{0}^{R} |f_n(s,v)| dv \, +\, \| f_n (s, \cdot)\|_{{\mc H}^{\gamma/2}}\\
&\lesssim \cfrac{1}{\sqrt{R}} \, \| f_n (s, \cdot)\| \, + \,  \| f_n (s, \cdot)\|_{{\mc H}^{\gamma/2}} \, \lesssim\, \cfrac{1}{\sqrt{R}} \| f_n (s, \cdot)\|_{{\mc H}^{\gamma/2}}.  
\end{split}
\end{equation*}
It follows that
\begin{equation*}
\int_0^{T} | f_n (s,0)|^2 ds \lesssim \cfrac{1}{R} \int_0^T \| f_n (s, \cdot)\|_{{\mc H}^{\gamma /2}}^2 \, ds.    
\end{equation*}
The right-hand side of the previous inequality goes to $0$, as $n$ goes to infinity, so that the same holds for the left-hand side. This completes the proof.


\subsection{Uniqueness of  weak solutions of \eqref{eq:Dirichlet Equation_infty}}

Consider two  weak solutions $\rho^{\hat \kappa,1}$ and $\rho^{\hat \kappa,2}$, and note that  their difference $\rho^{\hat \kappa}:= \rho^{\hat \kappa,1}-\rho^{\hat \kappa,2}$ belongs to $L^2(0,T;V_1)$. Observe that
\begin{align*}
  \forall H \in C_c^{1,\infty} ( [0,T] \times (0,1) ), \forall s \in [0,T], \quad     \| H (s, \cdot)  \| \lesssim \| H (s, \cdot) \|_{V_1}.
\end{align*}
Since $C_c^{1,\infty} \big( [0,T] \times (0,1) \big)$ is dense in $L^2(0,T;V_1)$, it is enough to approximate $\rho^{\hat \kappa}$ in $L^2(0,T;V_1)$ by a sequence $(H_n)_{n \in \mathbb{N}}$ in $C_c^{1,\infty} \big( [0,T] \times (0,1) \big)$. Repeating the procedure analogous to the one in \cite{BJGO2} for the proof of uniqueness of \eqref{eq:Dirichlet Equation_infty} in the regime $\gamma \in (1,2)$, we get the desired result.

\section{Technical lemmas}
In this section we state and prove the main lemmas we used in this work. 

\subsection{Convergence of a discrete operator to the regional fractional Laplacian}
\label{sec:convergence_disc_cont_oper}
The first lemma we prove in this section establishes the convergence in $L^1$ of the discrete operator $N^\gamma \mathcal{L}_N$ defined by \eqref{LN} to the regional fractional Laplacian $\mathbb L$ when applied to smooth test functions. 
\begin{lem}
	\label{L1}
	For any $G \in C^{\infty}([0,1])$ and $\gamma \in (0,2)$, the following convergence holds
	\begin{equation}
	\lim_{N \rightarrow \infty} N^{-1}\sum_{x \in \Lambda_N}|N^{\gamma}(\mathcal{L}_N G)(\tfrac{x}{N})- (\mathbb{L}G)(\tfrac{x}{N})|=0.
	\label{x>e}
	\end{equation}
\end{lem}
\begin{proof}
	Fix a small $\epsilon >0$, then the  expression inside the limit in \eqref{x>e} is equal to 
	\begin{equation}
	\label{eq:Stef-L2}
	\begin{split}
	 N^{-1}& \Bigg\{\sum_{\substack{1\leq x \leq \epsilon N}}|N^{\gamma}(\mathcal{L}_N G)(\tfrac{x}{N})- (\mathbb{L}G)(\tfrac{x}{N})|\\
	&+\sum_{ \epsilon N<x\leq ({N-1})/{2}}|N^{\gamma}(\mathcal{L}_N G)(\tfrac{x}{N})- (\mathbb{L}G)(\tfrac{x}{N})|\\
	&+\sum_{(N-1)/2<x< N-1-\epsilon N}|N^{\gamma}(\mathcal{L}_N G)(\tfrac{x}{N})- (\mathbb{L}G)(\tfrac{x}{N})|\\
	&+\sum_{N-1-\epsilon N\leq x\leq N-1}|N^{\gamma}(\mathcal{L}_N G)(\tfrac{x}{N})- (\mathbb{L}G)(\tfrac{x}{N})|\Bigg\}.
	\end{split}
	\end{equation}
We will analyze the four terms above when $N \rightarrow \infty$ and afterwards $\epsilon \rightarrow 0$. From  the triangular inequality, the first one is bounded from above by
	\begin{equation}
	 N^{-1}\sum_{1\leq x \leq \epsilon N}|N^{\gamma}(\mathcal{L}_N G)(\tfrac{x}{N})|+ N^{-1}\sum_{\substack{1 \leq x \leq \epsilon N}}| (\mathbb{L}G)(\tfrac{x}{N})|.
	\label{x<e}
	\end{equation}
	Now, notice that in this case ($1\leq x\leq \epsilon N$) we can write 
	\begin{equation}
	(\mathcal{L}_N G)(\tfrac{x}{N})= \sum_{y=1 }^{x-1}p(y)\theta_{\tfrac{x}{N}}(\tfrac{y}{N})+\sum_{y=x}^{N-1-x}p(y)(G(\tfrac{x+y}{N})-G(\tfrac{x}{N})),
	\end{equation}
	where we introduced 
	\begin{equation}
	    \theta_{\tfrac{x}{N}}(z):=G(\tfrac{x}{N}+z)+G(\tfrac{x}{N}-z)-2G(\tfrac{x}{N}).
	\end{equation}
	Therefore, the term on the left-hand side of \eqref{x<e} is bounded from above by 
	\small
	\begin{equation}
		 \frac{1}{N}\sum_{1 \leq  x \leq \epsilon N}\Big| N^{\gamma} \sum_{y=1 }^{x-1}p(y)\theta_{\tfrac{x}{N}}(\tfrac{y}{N})\Big| +  \frac{1}{N} \sum_{1\leq  x \leq \epsilon N}\Big| N^{\gamma}\sum_{y=x}^{N-1-x}p(y)(G(\tfrac{x+y}{N})-G(\tfrac{x}{N}))\Big|.
		\label{2terms}
	\end{equation}
	\normalsize
	Performing a second order Taylor expansion in $\theta_{\tfrac{x}{N}}$, we can bound from above the term on the left-hand side of the previous display in the following way:
	\begin{equation}
	\begin{split}
	 N^{-1}\sum_{1\leq x \leq \epsilon N}\Big| N^{\gamma} \sum_{y=1 }^{x-1}p(y)\theta_{\tfrac{x}{N}}(\tfrac{y}{N})\Big| &\lesssim  N^{-1}\sum_{1 \leq x \leq \epsilon N}\, N^{\gamma-2} \sum_{y=1 }^{x-1}y^{1-\gamma}\\
	&\lesssim  N^{\gamma-3}\sum_{1\leq x \leq \epsilon N} x^{2-\gamma} \lesssim \epsilon^{3-\gamma}
	\end{split}
	\end{equation}
	which goes to $0$ as $\epsilon \rightarrow 0$.
	In order to prove that the first term of \eqref{x<e} goes to $0$  as $N \rightarrow \infty$ and then $\epsilon \rightarrow \infty$, we still have to treat the term on the right-hand side of \eqref{2terms}. Performing a Taylor expansion on $G$, we get 
	\begin{equation}
	\begin{split}
	& N^{-1} \sum_{1\leq x \leq \epsilon N}\Big| N^{\gamma} \sum_{y=x }^{N-1-x}p(y)(G(\tfrac{x+y}{N})-G(\tfrac{x}{N}))\Big|\\
	&\lesssim  N^{-1}\sum_{1\leq x \leq \epsilon N}N^{\gamma-1} \sum_{y=x }^{N-1-x}{y}^{-\gamma}\\&\lesssim  N^{\gamma-2} \sum_{y=1 }^{N-1} \sum_{\substack{1\leq x \leq \epsilon N \\ x\leq \inf\{y,N-1-y\}}}y^{-\gamma}\\
	&\lesssim N^{\gamma-2}\Big[\sum_{y\le \epsilon N } y^{-\gamma+1}+\sum_{y>\epsilon N} \epsilon N y^{-\gamma}\Big]\lesssim \epsilon^{-\gamma+2}
	\end{split}
	\end{equation}
and so it goes to $0$ for $\gamma <2$ when we pass to the limit $\epsilon \rightarrow 0$. 

	Now, we have to estimate the term on the right-hand side of \eqref{x<e}. In order to do that, we first get an upper bound for $\big| \mathbb{L}G(\tfrac{x}{N})\big|$ for $1\leq x\leq \epsilon N$. By definition of the operator $\mathbb{L}$, we have  that
			\begin{equation}
			\mathbb{L}G(\tfrac{x}{N}):=c_{\gamma}\lim_{\epsilon' \rightarrow 0}\int_0^1 \mathbb{1}_{|\tfrac{x}{N}-\nu|>\epsilon'}\frac{G(\nu)-G(\tfrac{x}{N})}{|\tfrac{x}{N}-\nu|^{\gamma+1}}d\nu. 
			\end{equation}
We use a first order Taylor expansion to write			
			\begin{equation}
			\label{eq:mistake1}
			\int_0^1 \mathbb{1}_{|\tfrac{x}{N}-\nu|>\epsilon'}\frac{G(\nu)-G(\tfrac{x}{N})}{|\tfrac{x}{N}-\nu|^{\gamma+1}}d\nu =
			G'\big(\tfrac{x}{N} \big)\int_0^1 \mathbb{1}_{|\tfrac{x}{N}-\nu|>\epsilon'}\frac{\nu-\tfrac{x}{N}}{|\tfrac{x}{N}-\nu|^{\gamma+1}}d\nu \; + \;  R_{N,x,\epsilon, \epsilon', G},
			\end{equation}
where the remainder term can be bounded as follows (since $\gamma < 2$)
\begin{equation*}
| R_{N,x,\epsilon, \epsilon', G} | \lesssim \int_0^1 \mathbb{1}_{|\tfrac{x}{N}-\nu|>\epsilon'}|\tfrac{x}{N}-\nu|^{1-\gamma}d\nu \lesssim 1.   
\end{equation*}
			 Moreover, by changing variables and taking $\epsilon'\leq 1/N \leq \epsilon$, we can write the integral on the right-hand side of \eqref{eq:mistake1} as
			\begin{equation}
			\begin{split}
			\int_{-\tfrac{x}{N}}^{1-\tfrac{x}{N}}&\mathbb{1}_{\{|z|>\epsilon'\}} \, z|z|^{-\gamma-1}dz = \int_{\epsilon'}^{1-\tfrac{x}{N}} \, z z^{-\gamma -1}dz+ \int_{-\tfrac{x}{N}}^{-\epsilon'}\, z (-z)^{-\gamma -1} dz\\
			&=\int_{\epsilon'}^{1-\tfrac{x}{N}}z^{-\gamma}dz-\int_{\epsilon'}^{\tfrac{x}{N}}z^{-\gamma}dz=\frac{(1-\tfrac{x}{N})^{-\gamma+1}-(\tfrac{x}{N})^{-\gamma+1}}{-\gamma+1},
			\end{split}
			\end{equation}
assuming $\gamma \neq 1$. Therefore, since $G'$ is uniformly bounded, if $\gamma \in (1,2)$ we have  the following bound:
			\begin{equation}
			 N^{-1}\sum_{1\leq x \leq \epsilon N}| (\mathbb{L}G)(\tfrac{x}{N})|\lesssim  N^{-1}\sum_{1\leq x \leq \epsilon N} \Big[(\tfrac{x}{N})^{1-\gamma} +1\Big]\lesssim \epsilon^{2-\gamma}
			\end{equation}
			which goes to $0$ if we send $\epsilon$ to $0$.
			Otherwise, if $\gamma \in (0,1)$ we have  
			\begin{equation}
			 N^{-1}\sum_{1\leq  x \leq \epsilon N}| (\mathbb{L}G)(\tfrac{x}{N})|\lesssim  N^{-1}\sum_{1 \leq  x \leq \epsilon N} 1 = \epsilon
			\end{equation}
			which also goes to $0$ as $\epsilon \rightarrow 0$. It remains to consider the case $\gamma=1$, in which the integral on the right-hand side of \eqref{eq:mistake1} is
			\begin{align*}
	\int_{-\tfrac{x}{N}}^{1-\tfrac{x}{N}}&\mathbb{1}_{\{|z|>\epsilon'\}} \, z|z|^{-\gamma-1}dz = \int_{\tfrac{x}{N}}^{1-\tfrac{x}{N}}z^{-1}dz = \log \Big( \frac{N-x}{x} \Big),		
			\end{align*}
which leads to
\begin{align*}
& N^{-1}\sum_{1\leq  x \leq \epsilon N}| (\mathbb{L}G)(\tfrac{x}{N})|\lesssim  N^{-1}\sum_{1 \leq  x \leq \epsilon N} [ \log(N-x) - \log(x) +1] \\
\lesssim&  \ N^{-1} \Big( \int_{N-\epsilon N}^{N-2} \log(v) dv - \int_2^{\epsilon N} \log(v) dv \Big) +\epsilon=  \epsilon[1- \log(\epsilon)] - (1- \epsilon) \log(1- \epsilon) ,
\end{align*}						
		which goes to zero as $\epsilon \rightarrow 0$.	
			This concludes the proof of the fact that the first term of \eqref{eq:Stef-L2} vanishes as $N \rightarrow \infty$ and then $\epsilon \rightarrow 0$.

Let us now focus on the second term  of  \eqref{eq:Stef-L2}, namely the one involving the sites $\epsilon N < x <(N-1)/2$. We can bound this term by 
	\begin{equation}	
	\label{est}
	\begin{split}
	N^{-1}\sum_{\epsilon N < x <(N-1)/2} &|N^{\gamma}(\mathcal{L}_N G)(\tfrac{x}{N})- (\mathbb{L}G)(\tfrac{x}{N})|  \\
	&\leq N^{-1}\sum_{\epsilon N < x <(N-1)/2}|N^{\gamma}(\mathcal{L}_N G)(\tfrac{x}{N})- (\mathbb{L}_{\epsilon}G)(\tfrac{x}{N})|\\
	& + N^{-1}\sum_{\epsilon N < x <(N-1)/2}|(\mathbb{L}_{\epsilon}G)(\tfrac{x}{N})- (\mathbb{L}G)(\tfrac{x}{N})|,
	\end{split}
	\end{equation}
where we defined 
	\begin{equation}
	\mathbb{L}_{\epsilon}G(\tfrac{x}{N}):=c_{\gamma}\int_0^1 \mathbb{1}_{|\tfrac{x}{N}-\nu|>\epsilon}\frac{G(\nu)-G(\tfrac{x}{N})}{|\tfrac{x}{N}-\nu|^{\gamma+1}}d\nu.
	\end{equation}
From a change of  variables in $\mathbb{L}_{\epsilon}G(\tfrac{x}{N})$, we can rewrite it  as
	\begin{equation}
	\label{eq:LepsG}
	\mathbb{L}_{\epsilon}G(\tfrac{x}{N})=c_{\gamma} \int_{\epsilon}^{\tfrac{x}{N}}|z|^{-(1+\gamma)}\theta_{\tfrac{x}{N}}(z)dz+c_{\gamma} \int_{\tfrac{x}{N}-1}^{-\tfrac{x}{N}}\frac{G(\tfrac{x}{N}-z)-G(\tfrac{x}{N})}{|z|^{\gamma+1}}dz,
	\end{equation}
since $\epsilon N < x < (N-1)/2$. Notice now that we can in the same way re-write 
	\begin{equation}
	    	\mathbb{L}G(\tfrac{x}{N})=\lim_{\epsilon'\rightarrow 0}c_{\gamma}  \left\{\int_{\epsilon'}^{\tfrac{x}{N}}|z|^{-(1+\gamma)}\theta_{\tfrac{x}{N}}(z)dz + \int_{\tfrac{x}{N}-1}^{-\tfrac{x}{N}}\frac{G(\tfrac{x}{N}-z)-G(\tfrac{x}{N})}{|z|^{\gamma+1}}dz \right\}.
	\end{equation}
Thus the second term at the right-hand side of \eqref{est} can be rewritten as
	\begin{equation}
	\begin{split}
 & N^{-1}\sum_{\epsilon N < x <(N-1)/2}| (\mathbb{L}_{\epsilon}G)(\tfrac{x}{N}) - \lim_{\substack{\epsilon' \rightarrow 0\\ \epsilon' <\epsilon}}(\mathbb{L}_{\epsilon'}G)(\tfrac{x}{N})|\\
 &=\lim_{\substack{\epsilon' \rightarrow 0\\ \epsilon'<\epsilon}}N^{-1}\sum_{\epsilon N < x <(N-1)/2}\bigg|\int_{\epsilon'}^{\epsilon}z^{-(1+\gamma)}\theta_{\tfrac{x}{N}}(z)dz \bigg|\\
	&\lesssim \lim_{\substack{\epsilon' \rightarrow 0\\ \epsilon'<\epsilon}}(\epsilon^{2-\gamma}-(\epsilon')^{2-\gamma}) =\epsilon^{2-\gamma},
	\end{split}
	\end{equation}
where the last inequality is obtained by performing a Taylor expansion of the second order on $G$ and using the fact that $||G''||_{\infty}$ is bounded. So, we have proved that this term goes to $0$ when $N \rightarrow \infty$ and then $\epsilon$ goes to $0$ {since} $\gamma < 2$. 
	
We estimate now the first term on the right-hand side of \eqref{est}. Observe that, by changing variables, we can write in this case ($\epsilon N < x <(N-1)/2$)  
	\begin{equation}
	\begin{split}
	(\mathcal{L}_N G)(\tfrac{x}{N}) &= \sum_{\epsilon N\leq y\leq x-1}p(y)\theta_{\tfrac{x}{N}}(\tfrac{y}{N})+ \sum_{1\leq y <\epsilon N }p(y)\theta_{\tfrac{x}{N}}(\tfrac{y}{N})\\
	&+\sum_{x\leq y\leq N-1-x}p(y)(G(\tfrac{x+y}{N})-G(\tfrac{x}{N})).
	\end{split}
	\end{equation}
	Then, recalling \eqref{eq:LepsG}, we have to estimate the following quantity: 
	\begin{equation}
	\begin{split}
	&N^{-1}\sum_{\epsilon N < x <(N-1)/2}|N^{\gamma}(\mathcal{L}_N G)(\tfrac{x}{N})- (\mathbb{L}_{\epsilon}G)(\tfrac{x}{N})|\\
	&\lesssim N^{-1}\sum_{\epsilon N < x <(N-1)/2}\bigg\{ \bigg|N^{\gamma}\sum_{y=\epsilon N}^{x-1}p(y)\theta_{\tfrac{x}{N}}(\tfrac{y}{N})-c_{\gamma}\int_{\epsilon}^{\tfrac{x}{N}}|z|^{-(1+\gamma)}\theta_{\tfrac{x}{N}}(z)dz \bigg| \\
	&\hspace{3cm} +\bigg|N^{\gamma}\sum_{y=1 }^{\epsilon N-1} p(y) \theta_{\tfrac{x}{N}}(\tfrac{y}{N})\bigg|\\
	& + \bigg|N^{\gamma}\sum_{y=x}^{N-1-x}p(y)[G(\tfrac{x+y}{N})-G(\tfrac{x}{N})]-c_{\gamma}\int_{\tfrac{x}{N}}^{1-\tfrac{x}{N}}|z|^{-(1+\gamma)}[G(\tfrac{x}{N}+z)-G(\tfrac{x}{N})]dz \bigg|\bigg\}\\
	&=N^{-1}\sum_{\epsilon N < x <(N-1)/2}\{(I)+(II)+(III)\}.
	\end{split}
	\end{equation}
	We can treat the term $(I)$ in the previous sum as 
	\small
	\begin{equation}
	\begin{split}
	(I)&=\bigg|\sum_{y = \epsilon N}^{x-1}c_{\gamma}\int_{\tfrac{y}{N}}^{\tfrac{y+1}{N}}\Big[(\tfrac{y}{N})^{-(1+\gamma)}\theta_{\tfrac{x}{N}}(\tfrac{y}{N})-z^{-(1+\gamma)}\theta_{\tfrac{x}{N}}(z)\Big] dz\bigg|\\&\lesssim \bigg|\sum_{y = \epsilon N}^{x-1}\int_{\tfrac{y}{N}}^{\tfrac{y+1}{N}}\big[(\tfrac{y}{N})^{-(1+\gamma)}-z^{-(1+\gamma)}\big] \theta_{\tfrac{x}{N}}(\tfrac{y}{N})dz\bigg| +\bigg|\sum_{y = \epsilon N}^{x-1}\int_{\tfrac{y}{N}}^{\tfrac{y+1}{N}}z^{-(1+\gamma)}\left(\theta_{\tfrac{x}{N}}(\tfrac{y}{N})-\theta_{\tfrac{x}{N}}(z)\right) dz\bigg|\\
	&\lesssim \sum_{y = \epsilon N}^{x-1}\int_{\tfrac{y}{N}}^{\tfrac{y+1}{N}}\big\Vert\theta_{\tfrac{x}{N}}(\cdot)\big\Vert_{\infty} [(\tfrac{y}{N})^{-(1+\gamma)}-z^{-(1+\gamma)}]dz+\sum_{y = \epsilon N}^{x-1}\int_{\tfrac{y}{N}}^{\tfrac{y+1}{N}}z^{-(1+\gamma)}\big\Vert\theta'_{\tfrac{x}{N}}(\cdot)\big\Vert_{\infty}|\tfrac{y}{N}-z| dz\\
	&\lesssim \sum_{y = \epsilon N}^{x-1}\int_{\tfrac{y}{N}}^{\tfrac{y+1}{N}}[(\tfrac{y}{N})^{-(1+\gamma)}-z^{-(1+\gamma)}]dz+N^{-1}\int_{\epsilon}^{1}z^{-(1+\gamma)} dz.
	\end{split}
	\end{equation}
	\normalsize
	Above we have used the fact that  $\Vert \theta_{\tfrac{x}{N}}(\cdot)\Vert_{\infty}$ and $\Vert \theta^{\prime}_{\tfrac{x}{N}}(\cdot)\Vert_{\infty}$ are uniformly bounded in $x$. Now, it is not difficult to check that 
	\begin{equation}
	\int_{\tfrac{y}{N}}^{\tfrac{y+1}{N}}[(\tfrac{y}{N})^{-(1+\gamma)}-z^{-(1+\gamma)}]dz \lesssim N^{\gamma}y^{-(2+\gamma)}.
	\end{equation}
	Thanks to these computations we conclude that (recall that $\epsilon$ will go to $0$ after $N$ being sent to infinity) 
	\begin{equation}
	\begin{split}
	\lim _{N\rightarrow \infty}N^{-1}\sum_{\epsilon N < x <(N-1)/2}(I)&\lesssim \lim _{N\rightarrow \infty}N^{-1}\sum_{\epsilon N < x <(N-1)/2}(N^{\gamma}\sum_{y = \epsilon N}^{x-1}y^{-(2+\gamma)}+{\epsilon}^{-\gamma} N^{-1}) =0.\\
	\end{split}
	\end{equation} 
	Now we estimate the term involving $(II)$:
	\begin{equation}
	\begin{split}
	 N^{-1}\sum_{\epsilon N < x <(N-1)/2}(II) &\leq  N^{-1}\sum_{\epsilon N < x <(N-1)/2}\Big| N^{\gamma} \sum_{y=1 }^{\epsilon N}p(y)\theta_{\tfrac{x}{N}}(\tfrac{y}{N})\Big| \\& \lesssim  N^{-1}\sum_{\epsilon N < x <(N-1)/2}\Big|N^{\gamma-2} \sum_{y=1 }^{\epsilon N}y^{1-\gamma}\Big|\\
	&\lesssim  N^{-1}\sum_{\epsilon N < x <(N-1)/2} \epsilon^{2-\gamma} \leq \epsilon^{2-\gamma},
	\end{split}
	\end{equation}
	which goes to $0$, as $\epsilon$ goes to $0$. 

	In order to conclude, we have to prove that the term involving $(III)$ goes to $0$. We have that
	\begin{equation}
	\begin{split}
&(III)\\
&=\Big| \sum_{y=x}^{N-1-x}c_{\gamma}\int_{\tfrac{y}{N}}^{\tfrac{y+1}{N}} \Big[(\tfrac{y}{N})^{-(\gamma+1)}(G(\tfrac{x+y}{N})-G(\tfrac{x}{N}))-z^{-(\gamma+1)}(G(\tfrac{x}{N} + z)-G(\tfrac{x}{N})) \Big] dz\Big|\\
	&\lesssim \Big| \sum_{y=x}^{N-1-x}\int_{\tfrac{y}{N}}^{\tfrac{y+1}{N}} \Big[(\tfrac{y}{N})^{-(\gamma+1)} -z^{-(\gamma +1)} \big] (G(\tfrac{x+y}{N})-G(\tfrac{x}{N})) dz\Big|\\
	&+ \Big| \sum_{y=x}^{N-1-x}\int_{\tfrac{y}{N}}^{\tfrac{y+1}{N}} z^{-(\gamma+1)}\Big[G(\tfrac{x+y}{N})-G(\tfrac{x}{N} + z) \Big] dz\Big|
	\end{split}
	\end{equation}
	
	By the mean value theorem applied to the function $f:u \to u^{-(\gamma+1)}$ and to the function $G$ we have 
	\begin{equation}
	\begin{split}
(III) &\; \lesssim\;  \frac{1}{N^2}\sum_{y=x}^{N-1-x} \Big(\tfrac{y}{N} \Big)^{-(\gamma+1)}\;  \le \;   \frac{1}{N}  \sum_{y=x}^{N-1-x} \int_{\tfrac{y-1}{N}}^{\tfrac{y}{N}} u^{-(\gamma+1)} du \; \lesssim\;  \frac{1}{N}  \, \Big( \frac{x}{N}\Big)^{-\gamma}. 
	\end{split}
	\end{equation}
It follows that
\begin{equation}
			\begin{split}
			N^{-1}&\sum_{\epsilon N < x <(N-1)/2}(III)\lesssim N^{\gamma-2}\sum_{\epsilon N < x <(N-1)/2}\frac{1}{x^{\gamma}}.
	\end{split}
			\end{equation}

	If $\gamma \in(1,2)$ the series $\sum_{x\ge 1} x^{-\gamma}$ is converging and so the term to estimate is of order $N^{\gamma-2}$ (uniformly in $\epsilon$) and goes to $0$, as $N$ goes to infinity. If $\gamma <1$, for fixed $\epsilon>0$ the sum is of order $N^{-\gamma+1}$ and so the global order is $N^{-1}$ and then goes to $0$, as $N$ goes to infinity. Finally, if $\gamma=1$, the sum can be bounded by $\int_{\varepsilon N}^{N} u^{-1} du$ and the term to estimate is of order $N^{-1} \log(\varepsilon^{-1})$ which goes to zero, since first $N$ goes to infinity and afterwards $\varepsilon$ goes to zero. This proves that the first two terms of \eqref{eq:Stef-L2} vanishes, as $N$ goes to infinity and then $\epsilon$ goes to $0$. 
	
	The analysis of the last two terms in \eqref{eq:Stef-L2} can be done in the same way of the first two, just performing the following change of variables $x \mapsto N-1-x$.
\end{proof}

\color{purple}

\color{black}

\subsection{Fixing the value of the profile at the boundary}\label{sec:fixing}
The aim of this subsection is to show that, in the case $0\leq \theta < \gamma-1$ and $\gamma \in (1,2)$, the profile $\rho(\cdot)$ satisfies the  Dirichlet boundary conditions stated in item iii) of Definition \ref{Def. Dirichlet0 Condition}. This will be a consequence of Proposition \ref{prop:fixing} which is a consequence of a series of results that we state below. Some of these auxiliary lemmas will also be used in the case $\theta=\gamma-1$. First, we define the following empirical averages:
\begin{equation}
\label{mean}
\overrightarrow{\eta}^\ell_s(x)=\frac{1}{\ell}\sum_{y=x+1}^{x+\ell}\eta^N_s(y) \quad \text{and} \quad \overleftarrow{\eta}^\ell_s(x)=\frac{1}{\ell}\sum_{y=x-\ell}^{x-1}\eta^N_s(y),
\end{equation} 
where $x\in \{0,\dots,N\}$ and $\ell$ is a positive integer such that, in the sum on the left-hand side we have  $x+\ell\leq N-1$ and on the right-hand side we have  $x-\ell\geq 1$. 

\begin{rem}
We always assume that in the size of the boxes involved, the number of sites on which we compute the mean is an integer. When it is not we approximate with the closest greater integer, for example $\overrightarrow \eta^{\epsilon N}(0)$ shall be read as $\overrightarrow \eta^{\lceil \epsilon N\rceil}(0)$. We do not write it in each step in order to leave the computations more readable.
\end{rem}

\begin{prop}\label{prop:fixing}
	\label{prop}
	In the regime $0\leq \theta <\gamma-1$ and $\gamma \in (1,2)$, for any $T>0$ and  $\epsilon>0$, the following convergences hold
	\begin{equation}\label{prop1}
	\limsup_{\epsilon \rightarrow 0}\limsup_{N \rightarrow \infty}\mathbb{E}_{\mu_N}\Bigg[\Big|\int_0^T(\alpha-\overrightarrow\eta_s^{\epsilon N}(0))ds\Big|\Bigg]=0,
	\end{equation}
	\begin{equation}
	\label{prop2}
	\limsup_{\epsilon \rightarrow 0}\limsup_{N \rightarrow \infty}\mathbb{E}_{\mu_N}\Bigg[\Big|\int_0^T(\beta-\overleftarrow\eta_s^{\epsilon N}(N))ds\Big|\Bigg]=0,
	\end{equation}
	where $\overrightarrow\eta_s^{\epsilon N}(0)$ and $\overleftarrow \eta_s^{\epsilon N}(N)$ are defined in \eqref{mean}.
\end{prop}
\begin{proof}
The proof of this lemma is a trivial consequence of Lemmas \ref{rep1}, \ref{lem:multi} and \ref{repalpha} in which is proved \eqref{prop1}. In Remark \ref{remark} we explain how to adapt this proof to show \eqref{prop2}. 
\end{proof}

We present first how to pass from the occupation variable $\eta_s^N(1)$ to the mean of the values of the occupations variables in the first $\epsilon N$ sites. It is done in two steps and it is the content of the next two lemmas. Finally, we show in Lemma \ref{repalpha} that $\eta_s^N(1)$, in some way, is a good approximation of $\alpha$. \\
	
	In the next lemma we show that it is possible to replace the occupation variable $\eta_s^N(1)$ by its empirical mean in a box of size $\ell_0=\epsilon N^{\gamma-1}$ sites in the double ordered limit $N\to \infty$ and then $\epsilon \to 0$.
		\color{black}

\begin{lem}[First approximation]
In the regime \footnote{We observe that this lemma is true for any value of  $\theta\geq 0$ and $\gamma<2$, but we restrict to these values of $\theta,\gamma$ since we only need the result  in this regime.}  $0\leq \theta \leq \gamma-1$ and $\gamma \in (1,2)$, for any $T>0$,  $\epsilon>0$ and $\ell_0=\epsilon N^{\gamma-1}$, it holds
\begin{equation}
\limsup_{\epsilon \rightarrow 0}\limsup_{N \rightarrow \infty}\mathbb{E}_{\mu_N}\Bigg[\Big|\int_0^T(\eta^N_s(1)-\overrightarrow\eta_s^{\ell_0}(0))ds\Big|\Bigg]=0,
\end{equation}
where $\overrightarrow\eta_s^{\ell_0}(0)$ is defined in \eqref{mean}.
\label{rep1}
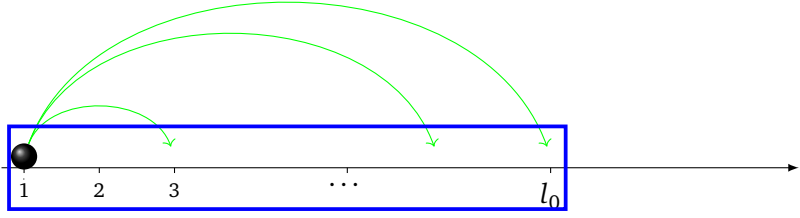
\begin{figure}[htb!]
		\begin{center}
		\begin{tikzpicture}

		\draw[shift={(-5.01,-0.15)}, color=black] (0pt,0pt) -- (0pt,0pt) ;

		\draw[-latex] (-5.3,0) -- (5.3,0) ;
		\foreach \x in {-5.0,-4,...,-3}
		\pgfmathsetmacro\result{\x+6}
		\draw[shift={(\x,0)},color=black] (0pt,0pt) -- (0pt,-2pt) node[below]{\scriptsize \pgfmathprintnumber{\result}};
		\draw[color=black] (-20pt,0pt) -- (-20pt,-2pt) node[below]{\dots};
		\draw[shift={(2,0)},color=black] (0pt,0pt) -- (0pt,-2pt) node[below]{$l_0$};
		
		\node[ball color=black, shape=circle, minimum size=0.3cm] (B) at (-5,0.15) {};

		\node[draw=none] (R) at (2,0.15) {};
		\node[draw=none] (L) at (-5,0.15) {};
		\node[draw=none] (M) at (0.5,0.15) {};
		\node[draw=none] (N) at (-3,0.15) {};

		\path [<-] (R) edge[bend right =70, color=green]node[above] {}(B);	
		\path [<-] (M) edge[bend right =70, color=green]node[above] {}(B);	
		\path [<-] (N) edge[bend right =70, color=green]node[above] {}(B);			

		\draw[, line width = 0.5 mm , draw=blue] (-5.2,0.55) rectangle (2.2,-0.55);

		\end{tikzpicture}
		\caption{First step, we can replace $\eta_s(1)$ with its empirical mean in the box $\{ 1,\dots,\ell_0 \}$.}	
	\end{center}	
\end{figure}
\end{lem}
\begin{proof}
	Consider an arbitrary Lipschitz function $\rho:[0,1]\rightarrow [\alpha,\beta]$ such that $\rho (0)=\alpha, \rho (1)=\beta$ and define the product measure $\nu^N_{\rho(\cdot)}$ on $\Omega_N$ by \eqref{product}.
	 By entropy and Jensen's inequalities we bound the expectation in the statement, for any $B>0$, by
	\begin{equation}
	\frac{H(\mu_N|\nu^N_{\rho(\cdot)})}{BN}+\frac{1}{BN}\log \mathbb E_{\nu_{\rho(\cdot)}^N}\Bigg[e^{BN\big|\int_0^T(\eta^N_s(1)-\overrightarrow\eta_s^{\ell_0}(0))ds\big|}\Bigg].
	\end{equation}
	It is possible to remove the absolute value in the exponential above observing that $e^{|x|}\leq \max\{e^x, e^{-x}\}$.
	Now, using the entropy bound \eqref{H} and Feynman-Kac's formula, there exists a constant $C_0>0$ independent of $N$ such that the last expression can be bounded by above by 
	\begin{equation}
	\frac{C_0}{B}+\sup_{f}\Bigg\{\int(\eta(1)-\overrightarrow\eta^{\ell_0}(0)) f(\eta) d\nu^N_{\rho(\cdot)} +\dfrac{N^{\gamma-1}}{B}\langle L_N\sqrt{f},\sqrt{f}\rangle_{\nu^N_{\rho(\cdot)}} \Bigg\}
	\label{FK}
	\end{equation}
	where the supremum is carried on all the probability densities with respect to $\nu^N_{\rho(\cdot)}$. We use \eqref{lipschitz}  which, in this case, becomes (for a constant $C>0$ independent of $N$)
	\begin{equation}
	\dfrac{N^{\gamma-1}}{B}\langle L_N\sqrt{f},\sqrt{f}\rangle _{\nu^N_{\rho(\cdot)}} \leq -\dfrac{N^{\gamma-1}}{4B}D_N(\sqrt{f},\nu^N_{\rho(\cdot)}) + C\dfrac{\kappa N^{-\theta}+1}{B}.
	\label{eest}
	\end{equation}
	
	Observe  that the integral inside the supremum in \eqref{FK} can be rewritten, using the definition of $\overrightarrow\eta^{\ell_0}(0)$, as
	\begin{equation}
	\frac{1}{\ell_0}\int \sum_{x =1}^{\ell_0}(\eta(1)-\eta(x))f(\eta)d\nu^N_{\rho(\cdot)}=\frac{1}{\ell_0}\int \sum_{x =1}^{\ell_0}\sum_{z=1}^{x-1}(\eta(z)-\eta(z+1))f(\eta)d\nu^N_{\rho(\cdot)}
	\label{tel}
	\end{equation}
	where the equality follows from a simple telescopic argument. 
Now, by writing \eqref{tel} as twice its half and summing and subtracting appropriate terms, we can write it as  the sum of the following two terms:
	\begin{equation}
	\frac{1}{2{\ell_0}}\sum_{x =1}^{\ell_0}\sum_{z=1}^{x-1}\int (\eta(z)-\eta(z+1))(f(\eta)-f(\sigma^{z,z+1}\eta))d\nu^N_{\rho(\cdot)},
	\label{-}
	\end{equation}
	\begin{equation}
	\frac{1}{2{\ell_0}}\sum_{x =1}^{\ell_0}\sum_{z=1}^{x-1}\int (\eta(z)-\eta(z+1))(f(\eta)+f(\sigma^{z,z+1}\eta))d\nu^N_{\rho(\cdot)}.
	\label{+}
	\end{equation}
	Let us now focus on \eqref{-}, which is equal to
	\begin{equation}
	\frac{1}{2{\ell_0}}\int \sum_{x =1}^{\ell_0}\sum_{z=1}^{x-1}(\eta(z+1)-\eta(z))(\sqrt{f}(\eta)-\sqrt{f}(\sigma^{z,z+1}\eta))(\sqrt{f}(\eta)+\sqrt{f}(\sigma^{z,z+1}\eta))d\nu^N_{\rho(\cdot)} \ ,
	\end{equation}
	and from Young's inequality, it can be bounded from above, for any $A>0$, by the sum of 
	\begin{equation}
	\frac{1}{2A{\ell_0}}\int \sum_{x =1}^{\ell_0}\sum_{z=1}^{x-1}(\sqrt{f}(\eta)-\sqrt{f}(\sigma^{z,z+1}\eta))^2d\nu^N_{\rho(\cdot)}
	\label{D}
	\end{equation}
	and
	\begin{equation}
	\frac{A}{2{\ell_0}}\int \sum_{x =1}^{\ell_0}\sum_{z=1}^{x-1}(\eta(z)-\eta(z+1))^2(\sqrt{f}(\eta)+\sqrt{f}(\sigma^{z,z+1}\eta))^2d\nu^N_{\rho(\cdot)}.
	\label{D2}
	\end{equation}
	Observe now that if $D_{NN}(\sqrt{f},\nu^N_{\rho(\cdot)}):=\sum_{z=1}^{N-2}\int(\sqrt{f}(\eta)-\sqrt{f}(\sigma^{z,z+1}\eta))^2d\nu^N_{\rho(\cdot)}$ we have that  $D_{NN}(\sqrt{f},\nu^N_{\rho(\cdot)})\lesssim D_{N}(\sqrt{f},\nu^N_{\rho(\cdot)})$. This means that we can bound \eqref{D}  from above by a constant times
	\begin{equation*}
	\frac{1}{2A{\ell_0}}\sum_{x =1}^{\ell_0} D_{N}(\sqrt{f},d\nu^N_{\rho(\cdot)})\leq \frac{1}{2A}D_{N}(\sqrt{f},d\nu^N_{\rho(\cdot)}).
	\end{equation*}
	
	Let us now treat \eqref{D2}. Using  $(a+b)^2 \lesssim a^2+b^2$, changing variables ($\eta \mapsto \sigma^{z,z+1}\eta$) and remembering that $\nu^N_{\rho(\cdot)}$ is a product measure with $\rho(\cdot)$ Lipschitz profile, we can bound from above \eqref{D2} by a constant times
	\begin{equation*}
		\frac{A}{2{\ell_0}}\int \sum_{x =1}^{\ell_0}\sum_{z=1}^{x-1}(\eta(z)-\eta(z+1))^2{f}(\eta)d\nu^N_{\rho(\cdot)}.
	\end{equation*}
	Therefore, since $f$ is a density and the difference $(\eta(z)-\eta(z+1))^2\leq 1$, we can bound the previous display from above by a constant times
	\begin{equation*}
	\frac{A}{2{\ell_0}}\int \sum_{x =1}^{\ell_0}\sum_{z=1}^{x-1}(\eta(z)-\eta(z+1))^2{f}(\eta)d\nu^N_{\rho(\cdot)} \lesssim {A{\ell_0}}.
	\end{equation*}
We have thus obtained that
\begin{equation}
\label{eq:536}
 \eqref{-} \lesssim \frac{1}{2A}D_{N}(\sqrt{f},d\nu^N_{\rho(\cdot)}) + A\ell_0.
\end{equation}

	 In order to treat \eqref{+}, we introduce the configuration $\tilde{\eta}^z\in \Omega_N^{z,z+1}:=\{0,1\}^{\Lambda_{N}^{z,z+1}}$ where $\Lambda_N^{z,z+1}:=\Lambda_N \setminus \{z, z+1\}$. Using the notation $f(\eta)=f(\tilde{\eta},\eta(z),\eta(z+1))$ we can rewrite the sum of densities inside the integral in \eqref{+} as
	\begin{equation}
	\begin{split}
	\sum_{\tilde{\eta} \in \Omega_N^{z,z+1}}&\Big\{f(\tilde{\eta},0,1)+f(\tilde{\eta},1,0)\Big\}\Big[\rho(\tfrac{z}{N})(1-\rho(\tfrac{z+1}{N}))-\rho(\tfrac{z+1}{N})(1-\rho(\tfrac{z}{N}))\Big]\tilde{\nu}^N_{\rho(\cdot)}(\tilde{\eta}),
	\end{split}
	\label{renato}
	\end{equation}
	where $\tilde{\nu}^N_{\rho(\cdot)}$ is the marginal density associated to $\nu^N_{\rho(\cdot)}$ restricted to the configurations $\tilde{\eta} \in \Omega_N^{z,z+1}$.
	Then, since $\rho(\cdot)$ was chosen  to be Lipschitz continuous and $f$ is a density, then \eqref{renato} is of order $1/N$. Then we obtained that
	\begin{equation}
	\label{eq:537}
	\eqref{+} \lesssim \frac{1}{\ell_0}\sum_{x =1}^{\ell_0}\sum_{z=1}^{x-1} \frac{1}{N}\lesssim \frac{\ell_0}{N}.
	\end{equation}
	Hence \eqref{eest}, \eqref{eq:536} and \eqref{eq:537} give that \eqref{FK}  is bounded from above by 
	\begin{equation*}
	C\Big(\frac{1}{B}+\frac{1}{2A}D_{N}(\sqrt{f},d\nu^N_{\rho(\cdot)})+A\ell_0+\frac{\ell_0}{N}\Big)-\dfrac{N^{\gamma-1}}{4B}D_N(\sqrt{f},\nu^N_{\rho(\cdot)}) + C\dfrac{\kappa N^{-\theta}+1}{B}.
	\end{equation*}
where  $C$ is a constant depending on $\gamma, \alpha, \beta$. By choosing $A=2BN^{-\gamma+1}C^{-1}$, then, for the choice $\ell_0=\epsilon N^{\gamma-1}$, the last display is less or equal than a constant times $$\dfrac{1}{B} + B\epsilon + \epsilon N^{\gamma-2}+\dfrac{\kappa N^{-\theta}+1}{B}.$$ Since $\gamma \in (1,2)$ and $\theta\ge 0$, the previous display and thus \eqref{FK} vanishes as $N\rightarrow \infty$, then $\epsilon \to 0$ and finally $B\to\infty$. This concludes the proof of the lemma.
\end{proof}

Finally, the following lemma shows that we can approximate (in some sense) the empirical mean of the occupation variable with respect to the first $\ell_0$ sites by the one with respect to the first $\epsilon N$ sites. It is done using an iterative procedure, doubling the number of sites involved in the empirical mean at each step up to reach the final size $\epsilon N$.

\begin{lem}[Multi-scale argument]\label{lem:multi}
Consider the case 
$0\leq \theta \leq \gamma-1$ and $\gamma \in (1,2)$. Then, for any $T>0$,
	\begin{equation}
	\limsup_{\epsilon\rightarrow 0}\limsup_{N\rightarrow \infty}\mathbb{E}_{\mu_N}\Bigg[\Big|\int_0^T(\overrightarrow \eta^{\ell_0}_s(0)-\overrightarrow \eta_s^{\epsilon N}(0))ds\Big|\Bigg]=0.
	\label{bounderr}
	\end{equation}
	
	\begin{figure}[htb!]
		\begin{center}
			\begin{tikzpicture}

			\draw[shift={(-5.01,-0.15)}, color=black] (0pt,0pt) -- (0pt,0pt) ;

			\draw[-latex] (-5.3,0) -- (5.3,0) ;
			\draw[shift={(-4.5,0)}, color=black] (0pt,0pt) -- (0pt,-2pt) node[below]{\dots};
			\draw[shift={(-3.5,0)},color=black] (0pt,0pt) -- (0pt,-2pt) node[below]{$\ell_0$};
			\draw[shift={(-5,0)},color=black] (0pt,0pt) -- (0pt,-2pt) node[below]{$1$};
			\draw[shift={(-2,0)},color=black] (0pt,0pt) -- (0pt,-2pt) node[below]{$\ell_1$};
			\draw[shift={(1,0)},color=black] (0pt,0pt) -- (0pt,-2pt) node[below]{$\ell_2$};

			\node[draw=none] (R) at (0.5,-0.2) {};
			\node[draw=none] (B) at (-4,0.2) {};
			\node[draw=none] (M) at (-2.5,-0.2) {};
			\node[draw=none] (N) at (-2.5,0.2) {};

			\path [line width=0.5mm, ->] (M) edge[bend right =70, color=green]node[above] {}(R);	
			\path [line width= 0.5mm, <-] (N) edge[bend right =90, color=blue]node[above] {}(B);			
			
			\draw[, line width = 0.5 mm , draw=blue] (-5.2,0.54) rectangle (-3.35,-0.54);
			\draw[, line width = 0.5 mm , draw=green] (-5.25,0.59) rectangle (-1.8,-0.59);
			\draw[, line width = 0.5 mm , draw=red] (-5.3,0.64) rectangle (1.2,-0.64);

			\end{tikzpicture}
			\caption{Multi-scale argument: we double the box $\{ 1,\dots,\ell_0 \}$ a finite number of times $M$  until reaching  the final box of size $2^M\ell_0=\ell_M=\epsilon N$. Here we  illustrate the first two steps, until reaching $\ell_2=2^2\ell _0$. }	
		\end{center}	
	\end{figure}
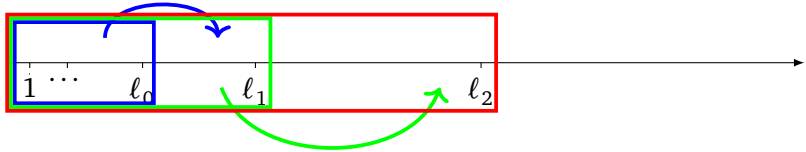
	
	\end{lem}
\begin{proof}
	    Consider a Bernoulli product measure $\nu_{\rho(\cdot)}^N$, as defined in \eqref{product}, with a Lipschitz profile $\rho:[0,1]\rightarrow (0,1)$ such that $\rho(0)=\alpha$ and $\rho(1)=\beta$.
		Now, proceeding exactly as we did in the beginning of the proof of Lemma \ref{rep1}, using entropy and Jensen's inequality and then Feynman-Kac's formula, we are reduced to prove that 
\begin{equation}
	\frac{C}{B}+\sup_f\Bigg\{\int (\overrightarrow \eta^{\ell_0}(0)-\overrightarrow \eta^{\epsilon N}(0))f(\eta)d\nu^N_{\rho(\cdot)}+\dfrac{N^{\gamma-1}}{B}\langle L_N\sqrt{f},\sqrt{f}\rangle_{\nu^N_{\rho(\cdot)}} \Bigg\},
	\label{FK2}
	\end{equation}
	vanishes in the triple limit $N\to \infty$, then $\epsilon \to0$ and finally $B\to \infty$. In the previous formula the supremum is carried over all the densities $f$ with respect to $\nu^N_{\rho(\cdot)}$.
	
Let us call $\ell_i:=2^i\ell_0$, and choose $M:=M_N$ such that $2^M\ell_0=\epsilon N$. Note that this $M$ is independent of $\epsilon$ and of order $\log N$ since $\ell_0=\epsilon N^{\gamma-1}$. With this notation we can write
	\begin{equation}
	\begin{split}
	\overrightarrow \eta^{\ell_0}(0)-\overrightarrow\eta^{\epsilon N}(0)=\sum_{i=1}^M\Big[\overrightarrow \eta^{\ell_{i-1}}(0)-\overrightarrow\eta^{\ell_i}(0)\Big]&=\sum_{i=1}^M\left[\frac{1}{\ell_{i-1}}\sum_{y=1}^{\ell_{i-1}}\eta(y)-\frac{1}{\ell_{i}}\sum_{y=1}^{\ell_{i}}\eta(y)\right]\\&=\sum_{i=1}^M\frac{1}{\ell_i}\sum_{y=1}^{\ell_{i-1}}\Big[\eta(y)-\eta(y+\ell_{i-1})\Big],
	\end{split}
	\label{path1}
	\end{equation}
	where the last equality comes from the definition of the $\ell_i$'s.
	
Thanks to the previous display, the first term in the supremum of \eqref{FK2} can be rewritten as the sum of
	\begin{equation}
	\frac{1}{2}\int \sum_{i=1}^M\frac{1}{\ell_i}\sum_{y=1}^{\ell_{i-1}}(\eta(y)-\eta(y+\ell_{i-1}))(f(\eta)-f(\sigma^{y,y+\ell_{i-1}}\eta))d\nu^N_{\rho(\cdot)}
	\label{--}
	\end{equation}
	and
	\begin{equation}
\frac{1}{2}	\int \sum_{i=1}^M\frac{1}{\ell_i}\sum_{y=1}^{\ell_{i-1}}(\eta(y)-\eta(y+\ell_{i-1}))(f(\eta)+f(\sigma^{y,y+\ell_{i-1}}\eta))d\nu^N_{\rho(\cdot)}.
	\label{++}
	\end{equation}

	Recall \eqref{Ixy}. Using Young's inequality as we did in Lemma \ref{rep1}, we bound from above \eqref{--} by the sum of
	\begin{equation}
	\begin{split}
	&\int \frac{1}{4}\sum_{i=1}^M\frac{1}{A_i\ell_i}\sum_{y=1}^{\ell_{i-1}}(\sqrt f(\eta)-\sqrt f(\sigma^{y,y+\ell_{i-1}}\eta))^2d\nu^N_{\rho(\cdot)}\\
	&= \frac{1}{4}\sum_{i=1}^M\frac{1}{A_i\ell_i}\sum_{y=1}^{\ell_{i-1}}I_{y,y+\ell_{i-1}}(\sqrt{f},\nu_{\rho(\cdot)}^N)
	\label{---}
	\end{split}
	\end{equation}
	and
	\begin{equation}
		\int \frac{1}{4}\sum_{i=1}^M\frac{A_i}{\ell_i}\sum_{y=1}^{\ell_{i-1}}(\eta(y)-\eta(y+\ell_{i-1}))^2(\sqrt f(\eta)+\sqrt f(\sigma^{y,y+\ell_{i-1}}\eta))^2d\nu^N_{\rho(\cdot)} \ ,
	\label{+++}
	\end{equation}
for some arbitrary positive $A_i$ that we are going to properly choose. Choosing $$A_i=A_0 \frac{B\ell_{i-1}^{\gamma}}{N^{\gamma-1}\ell_i}$$ with a suitable $A_0>0$ we can apply Lemma \ref{posMPL} to bound the right-hand side of \eqref{---} by {\footnote{We choose $A_0$ to be exactly the inverse of the constant which appear implicitly in Lemma \ref{posMPL}.  }} $$ \frac{N^{\gamma-1}}{4B}D_N(\sqrt{f},\nu_{\rho(\cdot)}^N).$$
	Let us estimate now \eqref{+++}. Since $f$ is a density and the $\eta(y)$'s are uniformly bounded, recalling that $2^M\ell_0=\epsilon N$, we can bound from above \eqref{+++} by a constant times $$\frac{B}{4N^{\gamma-1}}\sum_{i=1}^M\frac{\ell_{i-1}^{\gamma}}{\ell^2_i}\sum_{y=1}^{\ell_{i-1}}1=\frac{B}{4N^{\gamma-1}}\sum_{i=1}^M\frac{\ell_{i-1}^{\gamma+1}}{\ell^2_i}=\frac{B\ell_0^{\gamma-1}}{4N^{\gamma-1}}\frac{1}{2^{\gamma+1}}\sum_{i=1}^M 2^{i(\gamma-1)}\lesssim 
	B \epsilon^{\gamma-1}.$$

Hence we obtained that there exists a constant $C>0$ (independent of $N, \epsilon,B$) such that
\begin{equation}
\label{eq:5490}
\eqref{--} \leq  \frac{N^{\gamma-1}}{4B}D_N(\sqrt{f},\nu_{\rho(\cdot)}^N) +  C B \epsilon^{\gamma-1}.
\end{equation}

We still have to estimate \eqref{++}. Reasoning as we did in \eqref{renato}, we can rewrite the sum of the densities inside the integral in \eqref{++} as
	\begin{equation}
		\begin{split}
	\sum_{\tilde{\eta} \in \Omega_{N}^{{y},y+\ell_{i-1}}} &\Big\{f(\tilde{\eta},0,1)+f(\tilde{\eta},1,0)\Big\}\\
	& \times \Big[\rho(\tfrac{y}{N})(1-\rho(\tfrac{y+\ell_{i-1}}{N}))-\rho(\tfrac{y+\ell_{i-1}}{N})(1-\rho(\tfrac{y}{N}))\Big] \tilde{\nu}^N_{\rho(\cdot)} (\tilde\eta),
	\end{split}
	\end{equation}
	where now we denote $\tilde{\eta}$ as the configuration $\eta$ restricted to the sites in $\Lambda_N\backslash\{y, y+\ell_{i-1}\}$. So, the configuration $(\tilde{\eta},0,1)$ coincides with the configuration $\eta$ where $\eta(y+\ell_{i-1})=1$ and $\eta(y)=0$, and $\tilde{\nu}^N_{\rho(\cdot)}$ is the marginal density associated to $\nu^N_{\rho(\cdot)}$ restricted to the configurations $\tilde{\eta} \in \Omega_{N}^{y,y+\ell_{i-1}}:= \{0,1\}^{\Lambda_N\setminus \{y,y+\ell_{i-1}\}}$. 	Since $\rho(\cdot)$ was chosen to be Lipschitz continuous and $f$ is a density, the previous display is of order $\ell_{i-1}/N$. Therefore we get that
	\begin{equation}
	\label{eq:lastone}
\eqref{++} \lesssim \sum_{i=1}^M\frac{1}{\ell_i}\sum_{y=1}^{\ell_{i-1}}\frac{\ell_{i-1}}{N} \lesssim \frac{1}{N}\sum_{i=1}^M\frac{\ell_{i}^2}{\ell_i}\lesssim\frac{1}{N}\sum_{i=1}^M2^i\ell_0\lesssim 	\frac{2^M\ell_0}{N}=\epsilon
	\end{equation}
	since $M$ is such that $2^M\ell_0=\epsilon N$.

Hence, recalling \eqref{lipschitz}, \eqref{eq:5490} and \eqref{eq:lastone}, we get that the term inside the supremum in \eqref{FK2} can be bounded from above by a constant times 
$$  B\epsilon^{\gamma-1} + \epsilon + \dfrac{N^{\gamma} (\kappa + N^\theta)}{B N^{\gamma+\theta}}.$$
It follows that 
\begin{equation*}
\limsup_{B\to \infty}\limsup_{\epsilon \to 0}\limsup_{N\to \infty}\, \eqref{FK2} =0. 
\end{equation*}
This concludes the proof.
\end{proof}

\color{black}

In the next  lemma  we prove that, in the regime $\theta \in  [0, \gamma-1) $ and $\gamma \in (1,2)$, it is possible to replace, when integrated in time,  the occupation variable  $\eta_s^N(1)$ by $\alpha$ (and, analogously, $\eta_s^N(N-1)$ by $\beta$), as $N$ goes to infinity.

\color{black}

\begin{lem}
	For any $T>0$,   $0\leq\theta<\gamma-1$ and $\gamma \in (1,2)$ it holds:
	\begin{equation}
	\lim_{N \rightarrow \infty}\mathbb{E}_{\mu_N}\Bigg[\Big|\int_0^T(\eta^N_s(1)-\alpha)ds\Big|\Bigg]=0,
	\label{state}
	\end{equation}
	\begin{equation}
	\lim_{N \rightarrow \infty}\mathbb{E}_{\mu_N}\Bigg[\Big|\int_0^T(\eta^N_s(N-1)-\beta)ds\Big|\Bigg]=0,
	\label{state2}
	\end{equation}
	\label{repalpha}
\end{lem}

\begin{proof}
	We present the proof for the first expectation, but the proof of the other one is exactly the same. The first step in this proof is analogous to the beginning of the proof of Lemma \ref{rep1} therefore we omit some of the steps. We use Feynman-Kac's formula, entropy and Jensen's inequality to prove \eqref{state} it is sufficient to show that in the double limit ($N\to \infty$ and then $B\to \infty$) the term 
	\begin{equation}
	\cfrac{C}{B}+\sup_{f}\Bigg\{\int (\eta (1)-\alpha)f(\eta)d\nu^N_{\rho(\cdot)}+\cfrac{N^{\gamma-1}}{B}\langle L_N\sqrt f, \sqrt f\rangle _{\nu^N_{\rho(\cdot)}  } \Bigg\}
	\label{eq:oufa}
	\end{equation}
	vanishes. Here $C$ is a  constant independent from $N$ and the supremum is carried over all the densities with respect to  $\nu^N_{\rho(\cdot)}$ defined in \eqref{product}where $\rho (\cdot)$ is a Lipschitz continuous profile such that for all $u\in [0,1]$, $\alpha \leq \rho(u)\leq \beta$ and $\rho(0)=\alpha$, $\rho(1)=\beta$.  The first term inside the supremum above can be  rewritten as 
	\begin{equation}
	\label{fk}
	\cfrac{1}{2}\int (\eta(1)-\alpha)(f(\eta)-f(\sigma^1 \eta))d\nu_{\rho(\cdot)}+\cfrac{1}{2}\int (\eta(1)-\alpha)(f(\eta)+f(\sigma^1 \eta))d\nu^N_{\rho(\cdot)}.
	\end{equation}
	Recall \eqref{Ixy}. From Young's inequality and $(a-b)=(\sqrt a- \sqrt b) (\sqrt a + \sqrt b)$, we can bound the first term in the previous display from above by
	\begin{equation*}
	\cfrac{A}{2} I_1^{\alpha}(\sqrt f, \nu^N_{\rho(\cdot)})+
	\cfrac{1}{2A}\int (\eta(1)-\alpha)^2(\sqrt f(\eta)+\sqrt f(\sigma^1 \eta))^2d\nu^N_{\rho(\cdot)}
	\end{equation*}
	for any arbitrary positive constant $A$. Now  observe first that by \eqref{left_dir_form}, we have that $D_N^\ell(\sqrt f,\nu^N_{\rho(\cdot)})\leq r_1^-(\tfrac 1N)I_1^\alpha(\sqrt f,\nu^N_{\rho(\cdot)})$ and  secondly, that since $f$ is a density and $\rho(\cdot)$ is Lipschitz, the second term in the last display is bounded from above  by a constant times $1/N +1/A$ (see Lemma 5.5 of \cite{BGJO} for details). Hence we get that the first term in \eqref{fk} is bounded from above by  
	\begin{equation}
	\cfrac{A}{2r_1^-(\tfrac 1N)} D_N^l(\sqrt f, \nu^N_{\rho(\cdot)})+ C \left(\dfrac{1}{A} + {\dfrac{1}{N}} \right)
	\end{equation}
for some constant $C>0$ independent of $N,B, A$. Therefore, if we chose $$A=\tfrac{1}{2}\kappa B^{-1} N^{\gamma-\theta-1}r^-_1 (1/N)$$ we can bound the previous display by
	\begin{equation}
	\dfrac{\kappa N^{\gamma}}{4BN^{\theta+1}}D_N^l (\sqrt f, \nu^N_{\rho(\cdot)})+ C \left[\dfrac{2BN^{\theta+1}}{\kappa {N}^{\gamma} r^-_1(1/N)}+ \dfrac{1}{N} \right].
	\label{fk2}
	\end{equation}
	Since $\kappa N^{-\theta}D^l_N(\sqrt f, \nu^N_{\rho(\cdot)}) \leq D_N(\sqrt f, \nu^N_{\rho(\cdot)})$, by using \eqref{lipschitz}, we get that \eqref{eq:oufa} is bounded from above by a constant times 
	$$  \dfrac{1}{B} + \dfrac{2BN^{\theta+1}}{\kappa {N}^{\gamma} r^-_1(1/N)}+ \dfrac{1}{N} + \dfrac{N^{\gamma} (\kappa + N^\theta)}{B N^{\gamma+\theta}}.$$
Since $r_N^-(1/N)$ is bounded from bellow by a constant independent of $N$ and $\theta<\gamma-1$ the proof ends by choosing  $B=N^{\frac{\gamma-1-\theta}{2}}  $ .	
	\end{proof}

\begin{rem}
	\label{remark}
	The proof of \eqref{prop2} is completely analogous and it is a consequence of Lemma \ref{repalpha}, Lemmas \ref{rep1} and \ref{lem:multi}, when we replace $\alpha$ by $\beta$, $\eta^N_s(1)$ by $\eta^N_s(N-1)$ and $\overrightarrow\eta^{\ell}_s(0)$ by $\overleftarrow\eta^{\ell}_s(N)$.  The details are left to the reader.
\end{rem}

\subsection{Moving particle lemma}\label{MPL}
In this section we present a useful result to estimate the `cost' of a change of position of particles in our model. In other words, how much the exchange of position of a particle from a site $y \in \Lambda_N$ to a site $y+\ell\in \Lambda_N$, for some $\ell>0$, changes the `energy' of the system in terms of the quantity \eqref{D_false}.

\begin{lem} \label{posMPL}
	Fix $\ell_0 \leq N-1$ and define $\ell_i:=2^{i} \ell_0$ for any $i \in \{1,\dots, M\}$ where $M$ is a positive integer such that $\ell_M<N-1$. Let $f$ be a density with respect to $\nu_{\rho(\cdot)}^N$ on $\Omega_{N}$, where $\nu^N_{\rho(\cdot)}$ is defined as in \eqref{product} with $\rho:[0,1]\rightarrow (0,1)$ such that $\rho(0)=\alpha$ and $\rho(1)=\beta$. Then
	\begin{align*}
	\sum_{i=1}^M \sum_{y=1}^{\ell_{i-1}}  \cfrac{I_{y, y+\ell_{i-1}} \Big(\sqrt{f}, \nu_{\rho(\cdot)}^N \Big)}{\ell_{i-1}^{\gamma}} \lesssim   D_N \Big(\sqrt{f}, \nu_{\rho(\cdot)}^N \Big ) 
	\end{align*}
	where $D_N(\sqrt{f}, \nu_{\rho(\cdot)}^N  )$ is defined in \eqref{D_false} and $I_{y, y+\ell_{i-1}}$ in \eqref{Ixy} .
\end{lem}
\begin{proof}
We can assume without loss of generality that $\ell_0$ is even (the argument is easy to extend to an odd $\ell_0$) and then $\ell_{i-1}$ is an even number for any $i \in \{1, \ldots,M\}$. 

Fix $i \in \{1, \ldots, M\}$. For every $y \in \{1, \ldots, \ell_{i-1}\}$ consider the $\frac{\ell_{i-1}}{2}$ possibilities that a particle has to jump from $y$ to $y + \ell_{i-1}$ in at most two steps. Hence for any $j \in \{1, \ldots, \frac{\ell_{i-1}}{2}\}$, define 
\begin{equation}
\label{eq:zdefinition}
z_{0,j}:=z_{0,j}^i (y)=y,\quad z_{1,j}:=z_{1,j}^i (y) = y + \frac{\ell_{i-1}}{2}+j\quad z_{2,j}:=z_{2,j}^i (y)=y + \ell_{i-1} ,
\end{equation}
which correspond to one jump of length $\frac{\ell_{i-1}}{2}+j$ from $z_{0,j}=y$ to $z_{1,j}$  and one jump of length $\frac{\ell_{i-1}}{2}-j$ from $z_{1,j}$ to $z_{2,j}=y+\ell_{i-1}$. Observe that for $j=\frac{\ell_{i-1}}{2}$, we are moving from $y$ to $y+\ell_{i-1}$ with only one jump and $z_{1,j} = z_{2,j}$. 
%
%
%

Recall \eqref{Ixy}. Observe that
\begin{equation}\label{problem_cedric}
\sqrt{f \left( \sigma^{y,y+\ell_{i-1}}\eta \right) } - \sqrt{f \left( \eta \right) } = \sqrt{f \left( \sigma^{z_{0,j},z_{2,j}}\eta \right) } - \sqrt{f \left( \eta \right) }  
\end{equation}
is not $0$ if, and only if, $\eta(z_{0,j})\neq \eta(z_{2,j})$. We want to rewrite \eqref{problem_cedric} using the point $z_{1,j}$. To do that we consider separately the possible combinations of values of the $z_{i,j}$ with $i\in \{0,1,2\}$. First, assume that $\eta(z_{0,j})=1$ and $\eta(z_{2,j})=0$. In this case we have two possibilities:
\begin{enumerate}
	\item[a.] when $\eta(z_{1,j})=0$, we observe that the action of the operator $\sigma^{z_{0,j},z_{2,j}}$ on the configuration $\eta$ is equivalent to the action in sequence (and in that order) of $\sigma^{z_{0,j},z_{1,j}}$ and then  $\sigma^{z_{1,j},z_{2,j}}$. So, in this particular case we can write \eqref{problem_cedric} as \begin{equation}\label{list1} \Big[\sqrt{f ( \sigma^{z_{1,j},z_{2,j}}( \sigma^{z_{0,j},z_{1,j}}\eta )) } - \sqrt{f \left( \sigma^{z_{0,j},z_{1,j}}\eta \right) } \Big] + \Big[\sqrt{f (  \sigma^{z_{0,j},z_{1,j}}\eta ) } - \sqrt{f ( \eta ) } \Big].\end{equation}
	\item[b.] when $\eta(z_{1,j})=1$, the action of the operator $\sigma^{z_{0,j},z_{2,j}}$ on the configuration $\eta$ is equivalent to the action in sequence (and in that order) of $\sigma^{z_{1,j},z_{2,j}}$ and then  $\sigma^{z_{0,j},z_{1,j}}$. So, in this particular case we can write \eqref{problem_cedric} as \begin{equation}\label{list2}\Big[\sqrt{f ( \sigma^{z_{0,j},z_{1,j}}( \sigma^{z_{1,j},z_{2,j}}\eta )) } - \sqrt{f \left( \sigma^{z_{1,j},z_{2,j}}\eta \right) } \Big] + \Big[\sqrt{f (  \sigma^{z_{1,j},z_{2,j}}\eta ) } - \sqrt{f ( \eta ) } \Big].\end{equation}
\end{enumerate} 
Then, we have to consider the case in which $\eta(z_{1,j})=0$ and $\eta(z_{2,j})=1$. Then, reasoning similarly to what we did above, if  $\eta(z_{1,j})=0$ we can rewrite \eqref{problem_cedric} as \eqref{list2}, otherwise, if $\eta(z_{1,j})=1$ we can rewrite \eqref{problem_cedric} as \eqref{list1}.

Let us now consider the configurations in which we can rewrite \eqref{problem_cedric} as \eqref{list1} and call the set of these configurations 
\begin{equation*}
\begin{split}
\tilde \Omega_N^1 = \{ \eta \in \Omega_N: &\eta(z_{0,j})=1,\eta(z_{1,j})=0,\eta(z_{2,j})=0 \\
&\text{ or } \eta(z_{0,j})=0,\eta(z_{1,j})=1,\eta(z_{2,j})=1\}.
\end{split}
\end{equation*}
In a similar way let us define the following set of configurations 
\begin{equation*}
\begin{split}
\tilde \Omega_N^2= \{\eta \in \Omega_N: &\eta(z_{0,j})=1,\eta(z_{1,j})=1,\eta(z_{2,j})=0\\
& \text{ or } \eta(z_{0,j})=0,\eta(z_{1,j})=0,\eta(z_{2,j})=1\},
\end{split}
\end{equation*}
which are the ones corresponding to the case in which we can rewrite \eqref{problem_cedric} as \eqref{list2}. Observe that $\tilde \Omega_N^1\cap \tilde \Omega_N^2 = \emptyset$. Now, thanks to the reasoning that we did above and using the inequality $(a+b)^2 \le 2(a^2+b^2)$, we can write
\begin{equation}
\begin{split}
&I_{y,y+\ell_{i-1}}  (\sqrt{f}, \nu_{\rho(\cdot)}^N )= \int (\sqrt{f \left( \eta^{y,y+\ell_{i-1}} \right) } - \sqrt{f \left( \eta \right) } )^2 d \nu_{\rho(\cdot)}^N  \\
&= \int_{\tilde \Omega_N^1} \Big( \Big[\sqrt{f \left( \sigma^{z_{1,j},z_{2,j}}( \sigma^{z_{0,j},z_{1,j}}\eta \right)) } - \sqrt{f \left( \sigma^{z_{0,j},z_{1,j}}\eta \right) } \Big] \\
&\hspace{4cm} + \Big[\sqrt{f (  \sigma^{z_{0,j},z_{1,j}}\eta ) } - \sqrt{f ( \eta ) } \Big] \Big)^2 d\nu_{\rho(\cdot)}^N \\ 
& \; \; +\int_{\tilde \Omega_N^2} \Big( \Big[\sqrt{f \left( \sigma^{z_{0,j},z_{1,j}}( \sigma^{z_{1,j},z_{2,j}}\eta \right)) } - \sqrt{f \left( \sigma^{z_{1,j},z_{2,j}}\eta \right) } \Big] \\
&\hspace{4cm}+ \Big[\sqrt{f (  \sigma^{z_{1,j},z_{2,j}}\eta ) } - \sqrt{f ( \eta ) } \Big] \Big)^2 d\nu_{\rho(\cdot)}^N \\
& \lesssim  \int_{\tilde \Omega_N^1}  \left(\sqrt{f \left( \sigma^{z_{1,j},z_{2,j}}( \sigma^{z_{0,j},z_{1,j}}\eta \right)) } - \sqrt{f \left( \sigma^{z_{0,j},z_{1,j}}\eta \right) }\right)^2 d\nu_{\rho(\cdot)}^N\\
&\quad\quad  + \int_{\tilde \Omega_N^1}\left(\sqrt{f (  \sigma^{z_{0,j},z_{1,j}}\eta ) } - \sqrt{f ( \eta ) } \right)^2 d\nu_{\rho(\cdot)}^N  \\
 &+\int_{\tilde \Omega_N^2} \left(\sqrt{f \left( \sigma^{z_{0,j},z_{1,j}}( \sigma^{z_{1,j},z_{2,j}}\eta \right)) } - \sqrt{f \left( \sigma^{z_{1,j},z_{2,j}}\eta \right) }\right)^2d\nu_{\rho(\cdot)}^N \\
 &\quad\quad  + \int_{\tilde \Omega_N^2} \left(\sqrt{f (  \sigma^{z_{1,j},z_{2,j}}\eta ) } - \sqrt{f ( \eta ) } \right)^2 d\nu_{\rho(\cdot)}^N.
 \end{split}
 \end{equation}
 Note that $ \int_{\tilde \Omega_N^i} \big(\sqrt{f (  \sigma^{x,y}\eta ) } - \sqrt{f ( \eta ) } \big)^2 d\nu_{\rho(\cdot)}^N \leq I_{x,y} \big(\sqrt{f},\nu_{\rho(\cdot)}^N\big)$, for any $i=1,2$ and any $x,y \in \Omega_N$. Moreover, by the properties of $\nu_{\rho(\cdot)}^N$ (product measure associated to a Lipschitz profile which is never $0$ nor $1$), the  Radon-Nikodym derivatives of the form $\tfrac{d\nu_{\rho(\cdot)}^N(\sigma^{x,y}\eta)}{d\nu_{\rho(\cdot)}^N(\eta)} $, for any $x,y \in \Lambda_N$, are bounded by a positive constant independent of $x$, $y$ and $N$. Hence, last display can be bounded from above by a constant times
 \begin{equation}
I_{z_{1,j},z_{2,j}}  (\sqrt{f}, \nu_{\rho(\cdot)}^N)+I_{z_{0,j},z_{1,j}}  (\sqrt{f}, \nu_{\rho(\cdot)}^N )+ I_{z_{0,j},z_{1,j}}  (\sqrt{f}, \nu_{\rho(\cdot)}^N )+ I_{z_{1,j},z_{2,j}}  (\sqrt{f}, \nu_{\rho(\cdot)}^N).
 \end{equation}
 Therefore, we obtained that 
 \begin{equation}
 I_{y,y+\ell_{i-1}}  \Big(\sqrt{f}, \nu_{\rho(\cdot)}^N \Big) \; \lesssim \; \sum_{k=1}^2  I_{z_{k-1,j},z_{k,j}}  \Big(\sqrt{f}, \nu_{\rho(\cdot)}^N \Big).
 \end{equation}

Observe now that by construction we have $[p (z_{k,j} - z_{k-1,j})]^{-1}  \lesssim \ell_{i-1}^{1+\gamma}$ (the longest possible jump is equal to $\ell_{i-1}$). Hence, we have
\begin{equation}
\begin{split}
 I_{y,y+\ell_{i-1}}  (\sqrt{f}, \nu_{\rho(\cdot)}^N ) & \lesssim\sum_{k=1}^2 I_{z_{k-1,j},z_{k,j}}  (\sqrt{f}, \nu_{\rho(\cdot)}^N)\\
 &\lesssim \ell_{i-1}^{1+\gamma} \sum_{k=1}^2   p (z_{k,j} - z_{k-1,j})I_{z_{k-1,j},z_{k,j}}  (\sqrt{f}, \nu_{\rho(\cdot)}^N).
\end{split}
\end{equation}
This is true for any $j \in \{1,\dots, \tfrac{\ell_{i-1}}{2}\}$, so we can write
\begin{equation}
\label{65}
\ell_{i-1}I_{y,y+\ell_{i-1}}  (\sqrt{f}, \nu_{\rho(\cdot)}^N )	\lesssim  \ell_{i-1}^{1+\gamma}  \sum_{j=1}^{\ell_{i-1}/2} \sum_{k=1}^2   p (z_{k,j} - z_{k-1,j})I_{z_{k-1,j},z_{k,j}}  (\sqrt{f}, \nu_{\rho(\cdot)}^N ),
\end{equation}
which implies that
\begin{equation}
\label{Iprev}
\sum_{i=1}^M \sum_{y=1}^{\ell_{i-1}}  \frac{I_{y,y+\ell_{i-1}}(\sqrt{f}, \nu_{\rho(\cdot)}^N )}{\ell_{i-1}^{\gamma}} \lesssim  \sum_{i=1}^M \sum_{y=1}^{\ell_{i-1}} \sum_{j=1}^{\ell_{i-1}/2} \sum_{k=1}^2   p (z_{k,j} - z_{k-1,j})I_{z_{k-1,j},z_{k,j}}  (\sqrt{f}, \nu_{\rho(\cdot)}^N ).
\end{equation}
Recall \eqref{eq:zdefinition} and, in particular, that the $z_{k,j}$'s depend in fact on $i$ and $y$. We claim that when $i,y,j,k$ describe the sets involved in last sum, the couples $(z_{k-1,j} , z_{k,j}):=(z_{k-1, j}^i (y), z_{k,j}^i (y) )$ are all different, i.e.
\begin{equation}
\label{eq:phi}
\Phi: (i,y,j,k) \to (z_{k-1, j}^i (y), z_{k,j}^i (y) ) \in \Lambda_N \times \Lambda_N \quad \text{is injective.}
\end{equation}
Therefore, recalling \eqref{D_false}, we can bound from above the term on the right-hand side of \eqref{Iprev} by
\begin{equation}\label{64}
\underset{v\leq w}{ \sum_{v,w \in \Lambda_N}} p(w-v) I_{v,w} (\sqrt{f}, \nu_{\rho(\cdot)}^N ) \lesssim D_N (\sqrt{f},\nu_{\rho(\cdot)}^N ).
\end{equation}
Putting together \eqref{64} and \eqref{Iprev} we get the statement.

We still have to prove \eqref{eq:phi} to conclude the proof. Let us assume that
\begin{equation*}
\Phi (i,y,j,k) = \Phi (i',y',j',k')
\end{equation*}
and let us prove that $(i,y,j,k)=(i',y',j',k')$. We distinguish four cases according to the values of $k$ and $k'$.

\begin{itemize}
\item $k=k'=1$: then $z_{0,j}^i (y)=z_{0,j'}^{i'} (y')$ and $z_{1,j}^i (y)=z_{1,j'}^{i'} (y')$ imply that 
$$y=y', \quad \tfrac{\ell_{i-1}}{2} +j = \tfrac{\ell_{i'-1}}{2} +j' .$$ 
Since $1 \le j \le \ell_{i-1} /2$ and $1 \le j' \le \ell_{i'-1}/2$ we have that
\begin{equation}
\label{eq:k11}
1 + \tfrac{\ell_{i-1}}{2} \le \tfrac{\ell_{i-1}}{2} +j \le \ell_{i-1}\quad \text{and}\quad 1 + \tfrac{\ell_{i'-1}}{2} \le \tfrac{\ell_{i'-1}}{2} +j' \le \ell_{i'-1}.
\end{equation}
If $i\le i' -1$ then $\ell_{i-1} \le \tfrac{\ell_{i' -1}}{2}<  \tfrac{\ell_{i' -1}}{2} +1$ and the equality $\tfrac{\ell_{i-1}}{2} +j = \tfrac{\ell_{i'-1}}{2} +j'$ is then in contradiction with \eqref{eq:k11}. If $i'\le i -1$ then $\ell_{i'-1} \le \tfrac{\ell_{i -1}}{2}<  1+ \tfrac{\ell_{i -1}}{2}$ the equality $\tfrac{\ell_{i-1}}{2} +j = \tfrac{\ell_{i'-1}}{2} +j'$ is then again in contradiction with \eqref{eq:k11}. Hence $i=i'$ and consequently $j=j'$. We are done.\\

\item $k=1$ and $k'=2$: then $z_{0,j}^i (y)=z_{1,j'}^{i'} (y')$ and $z_{1,j}^i (y)=z_{2,j'}^{i'} (y')$ imply that 
\begin{equation*}
y= y'+ \tfrac{\ell_{i'-1}}{2} +j', \quad  y+ \tfrac{\ell_{i-1}}{2} +j =y'+ \ell_{i'-1} ,
\end{equation*}
and hence by replacing $y$ in the second equality by $ y'+ \tfrac{\ell_{i'-1}}{2} +j'$, we get
\begin{equation*}
y= y'+ \tfrac{\ell_{i'-1}}{2} +j', \quad  \tfrac{\ell_{i-1}}{2} +j =\tfrac{\ell_{i'-1}}{2} -j'. 
\end{equation*}
Since $1\le y \le \ell_{i-1}$, $1\le y'\le \ell_{i'-1}$, $1 \le j \le \ell_{i-1} /2$ and $1 \le j' \le \ell_{i'-1}/2$ we have that
\begin{equation}
\label{eq:k12}
\begin{split}
& 1\le y \le \ell_{i-1} \quad \text{and} \quad 2 + \tfrac{\ell_{i'-1}}{2} \le y \le 2 \ell_{i'-1} , \\
& 1 + \tfrac{\ell_{i-1}}{2} \le \tfrac{\ell_{i-1}}{2} +j \le \ell_{i-1} \quad \text{and} \quad 0 \le  \tfrac{\ell_{i-1}}{2} +j \le \tfrac{\ell_{i'-1}}{2} -1 .
\end{split}
\end{equation}
If $i\le i'-1$ then $\ell_{i-1} \le \tfrac{\ell_{i' -1}}{2}<2 + \tfrac{\ell_{i'-1}}{2}$ and there is a contradiction with the  first line of \eqref{eq:k12}. If $i'\le i$ then $\ell_{i'-1} \le {\ell_{i -1}}$, hence $\tfrac{\ell_{i'-1}}{2} -1 < 1 + \tfrac{\ell_{i-1}}{2}$, which is in contradiction with the second line of \eqref{eq:k12}. Hence this case is not possible and we are done.\\

 \item $k=2$ and $k'=1$: by symmetry this case is equivalent to the previous one.\\

 \item $k'=2$ and $k=2$:  then $z_{1,j}^i (y)=z_{1,j'}^{i'} (y')$ and $z_{2,j}^i (y)=z_{2,j'}^{i'} (y')$ imply that 
$$\tfrac{\ell_{i-1}}{2} +j +y  = \tfrac{\ell_{i'-1}}{2} +j' +y', \quad {\ell_{i-1}}+y  = {\ell_{i'-1}} +y' .$$ 
The second equality and the fact that $1 \le y \le \ell_{i-1}$, resp. $1\le y' \le \ell_{i'-1}$ implies that
\begin{equation}
\label{eq:k22}
1+\ell_{i-1} \le y +\ell_{i-1} \le 2\ell_{i-1} \quad \text{and}\quad 1+\ell_{i'-1} \le y +\ell_{i-1} \le 2\ell_{i'-1}.
\end{equation}
If $i\le i'-1$ then $2\ell_{i-1} \le {\ell_{i' -1}}< {\ell_{i' -1}}+1 $ which is in contradiction with \eqref{eq:k22}. Similarly if $i'\le i-1$ then $2\ell_{i'-1} \le {\ell_{i -1}}< {\ell_{i -1}}+1 $ is in contradiction with \eqref{eq:k22}. Hence $i=i'$ and consequently we deduce that $y=y'$ and $j=j'$.\\
\end{itemize}

This concludes the proof of the lemma.

\end{proof}

\subsection{Auxiliary result }
The next proposition is used in order to define a notion of weak solution in the case $\theta=\gamma-1$ and $\gamma \in (1,2)$. The strategy of the proof is similar to the one used in the previous lemmas, so we will just sketch its proof.

\begin{lem}\label{replacementrobin}
	Fix $\theta=\gamma-1$, $\gamma \in (1,2)$. For any $T>0$, for any $G \in C^{\infty}([0,1])$ it holds 	\begin{equation*}
	\limsup_{\epsilon \rightarrow 0}\limsup_{N \rightarrow \infty}\mathbb{E}_{\mu_N}\Bigg[\Big|\int_0^T\sum_{x\in \Lambda_N}G(\tfrac{x}{N})r^{-}_N(\tfrac{x}{N})(\eta^N_s(x)-\overrightarrow\eta_s^{\ell_0}(0))ds\Big|\Bigg]=0,
	\end{equation*}
	\label{rep2}
	\begin{equation*}
	\limsup_{\epsilon \rightarrow 0}\limsup_{N \rightarrow \infty}\mathbb{E}_{\mu_N}\Bigg[\Big|\int_0^T\sum_{x\in \Lambda_N}G(\tfrac{x}{N})r^{+}_N(\tfrac{x}{N})(\eta^N_s(x)-\overleftarrow\eta_s^{\ell_0}(N))ds\Big|\Bigg]=0,
	\end{equation*}
where $\ell_0=\epsilon N^{\gamma-1}$.
\end{lem}
\begin{proof} 
	We prove only the part involving $r^-_N$, since the other one can be proved analogously. By the triangular inequality we can bound the expectation in the statement by the sum of
	\begin{equation}\label{real_statement}
		\mathbb{E}_{\mu_N}\Bigg[\Big|\int_0^T\sum_{\substack{x\in \Lambda_N\\x\leq \ell_0}}G(\tfrac{x}{N})r^{-}_N(\tfrac{x}{N})(\eta^N_s(x)-\overrightarrow\eta_s^{\ell_0}(0))ds\Big|\Bigg]
	\end{equation}
	and 
	\begin{equation}
	\mathbb{E}_{\mu_N}\Bigg[\Big|\int_0^T\sum_{\substack{x\in \Lambda_N\\x> \ell_0}}G(\tfrac{x}{N})r^{-}_N(\tfrac{x}{N})(\eta^N_s(x)-\overrightarrow\eta_s^{\ell_0}(0))ds\Big|\Bigg].
	\end{equation}
	The last display can be bounded from above by a constant times 
	\begin{equation*}
\sum_{\substack{x\in \Lambda_N\\x> \ell_0}}|G(\tfrac{x}{N})||r^{-}_N(\tfrac{x}{N})|\leq ||G||_{\infty} \sum_{\substack{x\in \Lambda_N\\x> \ell_0}}|r^{-}_N(\tfrac{x}{N})|.
	\end{equation*}
	Since $r^-_N(\tfrac{x}{N})\lesssim x^{-\gamma}$ the right-hand side of the last display can be bounded from above, using Riemann integral approximation, by a constant times   $$(\ell_0)^{-\gamma+1}=\epsilon^{-\gamma+1}(N^{\gamma-1})^{-\gamma+1}=\epsilon^{-\gamma+1} N^{-(\gamma-1)^2}, $$  and it vanishes in the limit $N\to \infty, \epsilon \to 0$. This means that in order to prove the statement of the lemma we only need to show that the limit \eqref{real_statement} is equal to zero.

	We proceed as we did in the proof of Lemma \ref{rep1}. Let $\rho \in (0,1)$ be a constant. By Jensen's and entropy inequalities plus Feynman-Kac's formula, we are reduced to show that
		\begin{equation}
	\frac{C_0}{B}+\sup_{f}\Bigg\{\int\sum_{\substack{x\in \Lambda_N\\x\leq \ell_0}}G(\tfrac{x}{N})r^{-}_N(\tfrac{x}{N})(\eta (x)-\overrightarrow\eta^{\ell_0}(0))f(\eta)d\nu_{\rho}^N+\dfrac{N^{\gamma-1}}{B} \langle L_N\sqrt{f},\sqrt{f}\rangle_{\nu^N_{\rho}} \Bigg\}
	\label{FKPat}
	\end{equation}
	vanishes in the limit $N\to \infty$, $\epsilon \to 0$ and then $B\to \infty$. As usual $C_0$ is a  constant independent from $N$ and the supremum is carried over all the densities with respect to  the Bernoulli product measure $\nu^N_{\rho}$. From \eqref{dir_est_const}  and recalling that $\theta=\gamma-1$, the term on the right-hand side of last display becomes  bounded from above by
	\begin{equation}\label{eq:dir_bound}
		 -\dfrac{N^\gamma}{4NB}D_{N}(\sqrt{f},\nu^N_{\rho}) +\frac{C\kappa}{B}
		\end{equation}
		where $C$ is a constant independent of $N$, $B$ (and $\epsilon$). The first term inside the supremum of \eqref{FKPat}  can be written  as
		\begin{equation*}
\frac{1}{\ell_0}\int\sum_{\substack{x\in \Lambda_N\\x\leq \ell_0}}G(\tfrac{x}{N})r^{-}_N(\tfrac{x}{N})\Big[\sum_{y=1}^{x}\sum_{z=y}^{x-1}(\eta(z+1)-\eta(z))+\sum_{y=x}^{\ell_0}\sum_{z=x}^{y-1}(\eta(z)-\eta(z+1))\Big]f(\eta)d\nu_{\rho}^N.
	\end{equation*}
	Now, since the profile in the product measure $\nu_\rho^N$ is a constant we can rewrite the last term, by a change of variable as
	\begin{equation}
	\begin{split}
	\frac{1}{2\ell_0}\int\sum_{\substack{x\in \Lambda_N\\x\leq \ell_0}}G(\tfrac{x}{N})r^{-}_N(\tfrac{x}{N})&\Big[\sum_{y=1}^{x}\sum_{z=y}^{x-1}(\eta(z+1)-\eta(z))(f(\eta)-f(\sigma^{z,z+1}\eta))\\&+\sum_{y=x}^{\ell_0}\sum_{z=x}^{y-1}(\eta(z)-\eta(z+1))(f(\eta)-f(\sigma^{z,z+1}\eta))\Big]d\nu_{\rho}^N.
	\end{split}
	\end{equation}
	Hence, using Young's inequality on both the terms in the last display with the same constant $A>0$ we can bound from above the whole term by
	\begin{equation}
	\begin{split}
		\frac{1}{4\ell_0}\int\sum_{\substack{x\in \Lambda_N\\x\leq \ell_0}} \left\vert G(\tfrac{x}{N})r^{-}_N(\tfrac{x}{N}) \right\vert &\Big[A\sum_{y=1}^{x}\sum_{z=y}^{x-1}(\eta(z+1)-\eta(z))^2(\sqrt{f}(\eta)+\sqrt{f}(\sigma^{z,z+1}\eta))^2\\&+A\sum_{y=x}^{\ell_0}\sum_{z=x}^{y-1}(\eta(z)-\eta(z+1))^2(\sqrt{f}(\eta)+\sqrt{f}(\sigma^{z,z+1}\eta))^2\\&+\frac{1}{A}\sum_{y=1}^{\ell_0}\sum_{z=1}^{\ell_0}(\sqrt{f}(\eta)-\sqrt{f}(\sigma^{z,z+1}\eta))^2\Big]d\nu_{\rho}^N.
	\end{split}
	\end{equation}

	Since $f$ is a density and $|\eta(z)|\leq 1$ for any $z$, we can bound from above the whole last term by (recall the definition \eqref{Ixy})
	\begin{equation}
	\label{eq:finalou}
	\begin{split}
	\frac{1}{4\ell_0} & \int \sum_{\substack{x\in \Lambda_N\\x\leq \ell_0}} \left\vert G(\tfrac{x}{N})r^{-}_N(\tfrac{x}{N}) \right\vert\left[A\sum_{y=1}^{\ell_0}\sum_{z=1}^{\ell_0} 2+\frac{1}{A}\sum_{y=1}^{\ell_0}\sum_{z=1}^{\ell_0}I_{z,z+1}(\sqrt{f},\nu_{\rho}^N)\right]\\
	&\leq \frac{C' A\ell_0}{2}+\frac{1}{4A}D_{N}(\sqrt{f},\nu^N_{\rho}),
	\end{split}
	\end{equation} 
	where $C'$ is a constant independent of $N, \epsilon, B$ (this because $\sum_{\substack{x\in \Lambda_N\\x\leq \ell_0}} \vert G(\tfrac{x}{N})r^{-}_N(\tfrac{x}{N})  \vert \lesssim 1$ and $\sum_{z=1}^{\ell_0}I_{z,z+1}(\sqrt{f},\nu_{\rho}^N)\leq D_{N}(\sqrt{f},\nu^N_{\rho})$).
	
	Then, taking $A=BN^{1-\gamma}$ the second term on the right-hand side of \eqref{eq:finalou} cancels with the first of \eqref{eq:dir_bound}. Summarising we bounded \eqref{FKPat} and hence the expectation in \eqref{real_statement} by
	\begin{equation}
	\frac{C \kappa}{B}+ C' \frac{B N^{1-\gamma}\ell_0}{2}\lesssim \frac{1}{B}+B\epsilon.
	\end{equation}
	So, sending $\epsilon$ to $0$ and then $B$ to infinity, the proof is concluded.
\end{proof}

\begin{rem}
	\label{robin_rem} Observe that if we apply Lemma \ref{replacementrobin} and then Lemma \ref{lem:multi} it is easy to show that 
	\begin{equation*}
	\limsup_{\epsilon \rightarrow 0}\limsup_{N \rightarrow \infty}\mathbb{E}_{\mu_N}\Bigg[\Big|\int_0^T\sum_{x\in \Lambda_N}G(\tfrac{x}{N})r^{-}_N(\tfrac{x}{N})(\eta^N_s(x)-\overrightarrow\eta_s^{\epsilon N}(0))ds\Big|\Bigg]=0.
	\end{equation*}
	Indeed the expectation above can be bounded from above by
	\begin{equation}
	\begin{split}
	&\mathbb{E}_{\mu_N}\Bigg[\Big|\int_0^T\sum_{x\in \Lambda_N}G(\tfrac{x}{N})r^{-}_N(\tfrac{x}{N})(\eta^N_s(x)-\overrightarrow\eta_s^{\ell_0}(0))ds\Big|\Bigg]\\
	&+\mathbb{E}_{\mu_N}\Bigg[\Big|\int_0^T\sum_{x\in \Lambda_N}G(\tfrac{x}{N})r^{-}_N(\tfrac{x}{N})(\overrightarrow\eta^{\ell_0}_s(0)-\overrightarrow\eta_s^{\epsilon N}(0))ds\Big|\Bigg] .
	\end{split}
	\end{equation}
	The first expectation above vanishes, as $N$ goes to infinity and $\epsilon$ goes to $0$, thanks to Lemma \ref{replacementrobin}. Now, call $C_N:=\sum_{x\in \Lambda_N}G(\tfrac{x}{N})r^{-}_N(\tfrac{x}{N})$, which is of order $1$ since the sum is convergent. Then,  the second expectation above vanishes, as $N$ goes to infinity and $\epsilon$ goes to $0$, thanks to Lemma \ref{lem:multi}.
	Moreover, note that by the symmetry property of this model, the results that we used works also when we study the right boundary (see Remark \ref{remark}). So, it is possible to show in the same way as above that
	\begin{equation*}
	\limsup_{\epsilon \rightarrow 0}\limsup_{N \rightarrow \infty}\mathbb{E}_{\mu_N}\Bigg[\Big|\int_0^T\sum_{x\in \Lambda_N}G(\tfrac{x}{N})r^{+}_N(\tfrac{x}{N})(\eta^N_s(x)-\overleftarrow\eta_s^{\epsilon N}(N))ds\Big|\Bigg]=0.
	\end{equation*}
\end{rem}
\color{black}

\vspace{1cm}
\thanks{ {\bf{Acknowledgements:}} This project has received funding from the European Research Council (ERC) under  the European Union's Horizon 2020 research and innovative programme (grant agreement   n. 715734). This work has been supported by the project LSD ANR-15-CE40-0020-01 of the French National Research Agency (ANR).  P.C. thanks FCT/Portugal for financial support through the project Lisbon Mathematics PhD (LisMath). P.G., S.S. and P.C. thank  FCT/Portugal for support through CAMGSD, IST-ID, projects UIDB/04459/2020 and UIDP/04459/2020. }


\nocite{*}
\bibliography{biblio}
\bibliographystyle{plain}

\Addresses
\end{document}